\documentclass[a4paper,11pt]{article}
\usepackage{graphicx}
\usepackage{amssymb}
\usepackage[french, english]{babel}
\usepackage[utf8]{inputenc}
\usepackage{amsmath,amsfonts,amssymb,amsthm}
\usepackage[T1]{fontenc}
\usepackage{fullpage}
\usepackage{hyperref}
\usepackage{mathrsfs}
\usepackage{relsize}
\usepackage{tikz}
\usepackage{appendix}
\usepackage{bbm}
\usepackage{enumerate}
\usepackage{mathtools}   

\usetikzlibrary{patterns}

\newcommand{\eps}{\varepsilon}

\newcommand{\N}{\mathbb N}
\newcommand{\Z}{\mathbb Z}

\newcommand{\R}{\mathbb R}

\newcommand{\F}{\mathscr F}

\newcommand{\E}{\mathbb E}
\newcommand{\T}{\mathbb T}
\renewcommand{\H}{\mathbb H}
\renewcommand{\S}{\mathbb S}
\renewcommand{\P}{\mathbb P}
\newcommand{\p}{\mathcal{P}}
\newcommand{\hp}{\mathcal{P}^h}

\newcommand{\argsh}{\mathrm{argsh}\,}
\newcommand{\argch}{\mathrm{argch}\,}

\theoremstyle{definition}

\newtheorem{thm}{Theorem}

\newenvironment{idproof}{
  \noindent \textit{Sketch of proof. }}{\hfill \(\square\) }

\newtheorem{defn}{Definition}[section]
\newtheorem{rem}[defn]{Remark}

\newtheorem{prop}[defn]{Proposition}
\newtheorem{corr}[defn]{Corollary}

\newtheorem{lem}[defn]{Lemma}

\definecolor{darkgreen}{rgb}{0,0.3,0}

\tikzstyle{every node}=[circle, draw, fill=black!50, inner sep=0pt, minimum width=4pt]
\tikzstyle{noir}=[circle, draw, fill=black, inner sep=0pt, minimum width=6pt]
\tikzstyle{blanc}=[circle, draw, fill=white, inner sep=0pt, minimum width=6pt]
\tikzstyle{bleu}=[circle, draw, fill=blue, inner sep=0pt, minimum width=4pt]
\tikzstyle{rouge}=[circle, draw, fill=red, inner sep=0pt, minimum width=4pt]
\tikzstyle{voilet}=[circle, draw, fill=violet, inner sep=0pt, minimum width=4pt]
\tikzstyle{texte}=[draw=none, fill=none]
\tikzstyle{marche}=[circle, draw, fill=orange, inner sep=0pt, minimum width=8pt]
\tikzstyle{edge}=[rectangle, draw, fill=violet, inner sep=0pt, minimum width=4pt, minimum height=4pt]
\tikzstyle{edge'}=[rectangle, draw, fill=violet, inner sep=0pt, minimum width=2pt, minimum height=2pt]
\tikzstyle{rootedge}=[rectangle, draw, fill=red, inner sep=0pt, minimum width=4pt, minimum height=4pt]
\tikzstyle{dual}=[circle, draw=red, fill=white, inner sep=0pt, minimum width=4pt]

\title{\bf{Infinite geodesics in hyperbolic random triangulations}}
\author{Thomas \bsc{Budzinski} \footnote{ENS Paris and Université Paris-Saclay, \url{thomas.budzinski@ens.fr}}}

\begin{document}

\maketitle

\begin{abstract}
We study the structure of infinite geodesics in the Planar Stochastic Hyperbolic Triangulations $\T_{\lambda}$ introduced in \cite{CurPSHIT}. We prove that these geodesics form a supercritical Galton--Watson tree with geometric offspring distribution. The tree of infinite geodesics in $\T_{\lambda}$ provides a new notion of boundary, which is a realization of the Poisson boundary. By scaling limit arguments, we also obtain a description of the tree of infinite geodesics in the hyperbolic Brownian plane. Finally, by combining our main result with a forthcoming paper \cite{B18}, we obtain new hyperbolicity properties of $\T_{\lambda}$: they satisfy a weaker form of Gromov-hyperbolicity and admit bi-infinite geodesics.
\end{abstract}

\section*{Introduction}

The construction and study of random infinite triangulations has been a very active field of reasearch in recent years. The first such triangulation that was built is the UIPT \cite{AS03, Ang03}. A key feature in the study of this object is its spatial Markov property, which motivated the introduction of a one-parameter family $(\T_{\lambda})_{0<\lambda \leq \lambda_c}$ of type-I triangulations\footnote{To be exact, type-II triangulations (i.e. with no loops) were considered in \cite{CurPSHIT}, while the type-I triangulations (with loops) were built in \cite{B16}. In this work, unless specified otherwise, we only consider type-I triangulations.} with $\lambda_c=\frac{1}{12\sqrt{3}}$, satisfying a similar property \cite{CurPSHIT, B16} (see also \cite{AR13} for similar constructions in the halfplanar case). The case $\lambda=\lambda_c$ corresponds to the UIPT, whereas for $\lambda<\lambda_c$ the triangulation $\T_{\lambda}$ has hyperbolic behaviour. For example, it was proved that $\T_{\lambda}$ has a.s. exponential volume growth and that the simple random walk on it has positive speed \cite{CurPSHIT}. The goal of this work is to establish hyperbolicity properties of these maps related to their geodesics.

\paragraph{Leftmost geodesic rays.}
Our first goal in this work is to describe precisely the structure of infinite geodesics in the triangulations $\T_{\lambda}$. More precisely, all the triangulations considered here are \emph{rooted}, that is, equipped with a distinguished oriented edge called the \emph{root edge}. The \emph{root vertex}, that we write $\rho$, is the starting point of the root edge. For any two vertices $x$ and $y$ in $\T_{\lambda}$, we call a geodesic $\gamma$ from $x$ to $y$ \emph{leftmost} if for any geodesic $\gamma'$ from $x$ to $y$, the union of $\gamma$ and $\gamma'$ cuts $\T_{\lambda}$ in two parts, and the part on the left of $\gamma$ is infinite. It is easy to see that for any vertices $x$ and $y$ there is a unique leftmost geodesic from $x$ to $y$. A \emph{leftmost geodesic ray} is a sequence of vertices $\left( \gamma(n) \right)_{n \geq 0}$ such that $\gamma(0)=\rho$ and for any $n \geq 0$, the path $\left( \gamma(i) \right)_{0 \leq i \leq n}$ is a leftmost geodesic from $\rho$ to $\gamma(n)$. We denote by $\mathbf{T}^g_{\lambda}$ the union of all the leftmost geodesic rays of $\T_{\lambda}$. We can see this set of vertices as a graph by relating two vertices if they are consecutive on the same geodesic ray. By uniqueness of the leftmost geodesic between two points, the graph $\mathbf{T}^g_{\lambda}$ is an infinite tree with no leaf. Moreover, the tree $\mathbf{T}^g_{\lambda}$ divides $\T_{\lambda}$ into infinite maps with geodesic boundaries, that we call \emph{strips} (see the left part of Figure \ref{TreeAndSlices}). Our first main result describes the distribution of $\mathbf{T}^g_{\lambda}$ and of these strips. Let $0<\lambda \leq \lambda_c$. Let $h \in \left( 0, \frac{1}{4} \right]$ be such that
\begin{equation}\label{eqn_h_lambda}
\lambda=\frac{h}{(1+8h)^{3/2}},
\end{equation}
and let
\begin{equation}\label{eqn_m_lambda}
m_{\lambda}=\frac{1-2h-\sqrt{1-4h}}{2h} \leq 1.
\end{equation}

\begin{thm}\label{thm1_GW}
\begin{itemize}
\item
The tree $\mathbf{T}^g_{\lambda}$ is a Galton--Watson tree with offspring distribution $\mu_{\lambda}$, where $\mu_{\lambda}(0)=0$ and $\mu_{\lambda}(k)=m_{\lambda} (1-m_{\lambda})^{k-1}$ for $k \geq 1$.
\item
There are two random infinite strips $S^0_{\lambda}$ and $S^1_{\lambda}$ such that the following holds. Conditionally on $\mathbf{T}^g_{\lambda}$:
\begin{enumerate}
\item
the strips delimited by $\mathbf{T}^g_{\lambda}$ are independent,
\item
the strip containing the face lying on the right of the root edge has the same distribution as $S^1_{\lambda}$,
\item
all the other strips have the same distribution as $S^0_{\lambda}$.
\end{enumerate}
\end{itemize}
\end{thm}

Note that for $\lambda=\lambda_c$, we have $h=\frac{1}{4}$ so $m_{\lambda_c}=1$ and $\mu_{\lambda_c}(1)=1$, so the tree $\mathbf{T}^g_{\lambda_c}$ consists of a single ray. This is reminiscent of the geodesics confluence properties already observed in \cite{CMMinfini} for the UIPQ (the natural analog of $\T_{\lambda_c}$ for quadrangulations): there are infinitely many points that lie on every geodesic ray. See also \cite{CM18} for similar results in the UIPT. However, the results of \cite{CMMinfini} are about all the geodesic rays whereas we only study the leftmost ones, so our result for $\lambda=\lambda_c$ does not obviously imply those of \cite{CMMinfini}. On the other hand, for $\lambda<\lambda_c$, we have $m_{\lambda}<1$, so the offspring distribution $\mu_{\lambda}$ is supercritical and there are infinitely many leftmost geodesic rays. Moreover, the rate of exponential growth $\mu_{ \lambda}$ of $\mathbf{T}^g_{\lambda}$ is the same as the rate of exponential volume growth of $\T_{\lambda}$.

We will also describe the distributions of $S^0_{\lambda}$ and $S^1_{\lambda}$ explicitly in terms of reverse Galton--Watson trees. For $\lambda<\lambda_c$, these strips should be thought of as "thin", in the sense that their width is of constant order as the distance from the root goes to $+\infty$.

We also state right now a consequence of Theorem \ref{thm1_GW} that will be useful later. Let $\gamma_{\ell}$ (resp. $\gamma_r$) be the path in $\mathbf{T}^g_{\lambda}$ that bounds $S^1_{\lambda}$ on its right (resp. on its left). Then $\gamma_ {\ell}$ (resp. $\gamma_r$) can be thought of as the leftmost (resp. rightmost) path in $\mathbf{T}^g_{\lambda}$, seen from the root (see Figure \ref{TreeAndSlices}). We write $\S_{\lambda}$ for the part of $\T_{\lambda}$ lying between $\gamma_{\ell}$ and $\gamma_r$, including the initial segment that $\gamma_{\ell}$ and $\gamma_r$ have in common (cf. Figure \ref{TreeAndSlices}). Then $\S_{\lambda}$ can be seen as a gluing of infinitely many independent copies of $S^0_{\lambda}$ along $\mathbf{T}^g_{\lambda}$. This implies that $\S_{\lambda}$ has an interesting self-similarity property. Indeed, let $r>0$. We condition on $B_r \left( \mathbf{T}^g_{\lambda} \right)$, the finite subtree of $\mathbf{T}^g_{\lambda}$ formed by those vertices lying at distance at most $r$ from $\rho$. Let $x$ be a vertex of $\mathbf{T}^g_{\lambda}$ such that $d(\rho,x)=r$. Let $\gamma_{\ell}^x$ (resp. $\gamma_r^x$) be the leftmost (resp. rightmost) infinite path in $\mathbf{T}^g_{\lambda}$ started from $\rho$ and passing through $x$. Then the part of $\T_{\lambda}$ lying between $\gamma_{\ell}^x$ and $\gamma_r^x$ above $x$ has the same distribution as $\S_{\lambda}$ (see the right part of Figure \ref{TreeAndSlices}). Indeed, this part consists of a gluing of i.i.d. copies of $S^0_{\lambda}$ in the faces of the tree of descendants of $x$ in $\mathbf{T}^g_{\lambda}$, which has the same distribution (conditionally on $B_r \left( \mathbf{T}^g_{\lambda} \right)$) as $\mathbf{T}^g_{\lambda}$. We will denote this part of $\T_{\lambda}$ by $\S[x]$. In all the rest of this work, "thin" maps such as $S^0_{\lambda}$ and $S^1_{\lambda}$ will be refered to as \emph{strips}, and "thick" maps like $\S_{\lambda}$ as \emph{slices}.

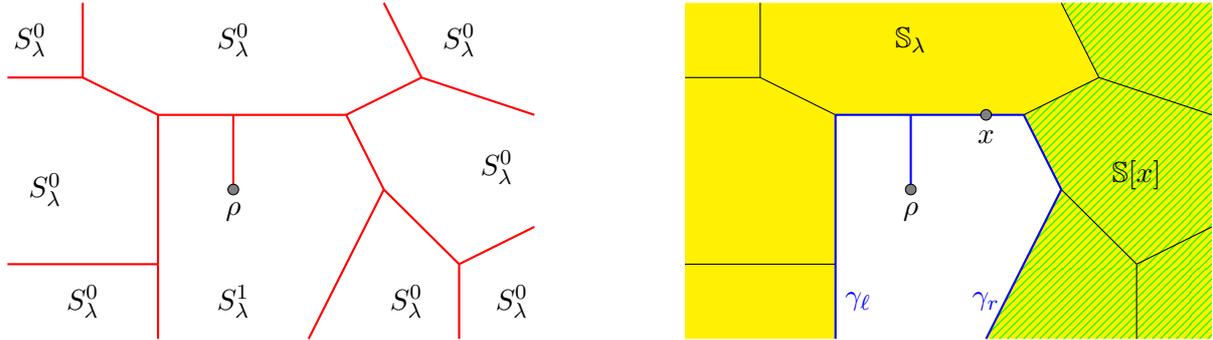
\begin{figure}
\begin{center}
\begin{tikzpicture}[scale=0.99]
\draw[red, thick] (0,0)--(0,1);
\draw[red, thick] (0,1)--(1.5,1);
\draw[red, thick] (0,1)--(-1,1);
\draw[red, thick] (-1,1)--(-1,-1);
\draw[red, thick] (-1,-1)--(-1,-2);
\draw[red, thick] (-1,1)--(-2,1.5);
\draw[red, thick] (-2,1.5)--(-2,2.5);
\draw[red, thick] (-2,1.5)--(-3,1.5);
\draw[red, thick] (-1,-1)--(-3,-1);
\draw[red, thick] (1.5,1)--(2.5,1.5);
\draw[red, thick] (1.5,1)--(2,0);
\draw[red, thick] (2,0)--(1,-2);
\draw[red, thick] (2,0)--(3,-1);
\draw[red, thick] (3,-1)--(3,-2);
\draw[red, thick] (3,-1)--(4,-0.5);
\draw[red, thick] (2.5,1.5)--(2,2.5);
\draw[red, thick] (2.5,1.5)--(4,1);

\draw (0,0) node{};

\draw (0,-0.3) node[texte]{$\rho$};
\draw (0,-1.5) node[texte]{$S^1_{\lambda}$};
\draw (-2,-1.5) node[texte]{$S^0_{\lambda}$};
\draw (-2.5,0) node[texte]{$S^0_{\lambda}$};
\draw (-2.7,2) node[texte]{$S^0_{\lambda}$};
\draw (0,2) node[texte]{$S^0_{\lambda}$};
\draw (3,2) node[texte]{$S^0_{\lambda}$};
\draw (3.5,0.3) node[texte]{$S^0_{\lambda}$};
\draw (2.3,-1.5) node[texte]{$S^0_{\lambda}$};
\draw (3.7,-1.5) node[texte]{$S^0_{\lambda}$};

\begin{scope}[shift={(9,0)}]
\fill[yellow] (-1,-2)--(-3,-2)--(-3,2.5)--(4,2.5)--(4,-2)--(1,-2)--(2,0)--(1.5,1)--(-1,1)--(-1,-2);
\fill[pattern color=green, pattern=north east lines] (2,2.5)--(4,2.5)--(4,-2)--(1,-2)--(2,0)--(1.5,1)--(2.5,1.5)--(2,2.5);

\draw[blue, thick] (0,0)--(0,1);
\draw[blue, thick] (0,1)--(1.5,1);
\draw[blue, thick] (0,1)--(-1,1);
\draw[blue, thick] (-1,1)--(-1,-1);
\draw[blue, thick] (-1,-1)--(-1,-2);
\draw (-1,1)--(-2,1.5);
\draw (-2,1.5)--(-2,2.5);
\draw (-2,1.5)--(-3,1.5);
\draw (-1,-1)--(-3,-1);
\draw (1.5,1)--(2.5,1.5);
\draw[blue, thick] (1.5,1)--(2,0);
\draw[blue, thick] (2,0)--(1,-2);
\draw (2,0)--(3,-1);
\draw (3,-1)--(3,-2);
\draw (3,-1)--(4,-0.5);
\draw (2.5,1.5)--(2,2.5);
\draw (2.5,1.5)--(4,1);

\draw (0,0) node{};
\draw (1,1) node{};

\draw (0,-0.3) node[texte]{$\rho$};
\draw (1,0.7) node[texte]{$x$};
\draw[blue] (-0.7,-1.5) node[texte]{$\gamma_{\ell}$};
\draw[blue] (1,-1.5) node[texte]{$\gamma_r$};
\draw (0,2) node[texte]{$\S_{\lambda}$};
\draw (3,0.2) node[texte]{$\S[x]$};

\end{scope}
\end{tikzpicture}
\end{center}
\caption{The strip decomposition of $\T_{\lambda}$. On the left, the tree $\mathbf{T}^g_{\lambda}$ is in red. On the right $\S_{\lambda}=\S[\rho]$ is colored in yellow. The part $\S[x]$ of $\S_{\lambda}$ that is shaded in green has the same distribution as $\S_{\lambda}$.}\label{TreeAndSlices}
\end{figure}

\paragraph{Hyperbolicity properties related to geodesics.}
By Theorem \ref{thm1_GW}, for $\lambda<\lambda_c$, the map $\T_{\lambda}$ contains a supercritical Galton--Watson tree. By combining this with the results of our forthcoming paper \cite{B18}, we obtain that $\T_{\lambda}$ satisfies two metric hyperbolicity properties: a weak form of Gromov-hyperbolicity, and the existence of bi-infinite geodesics.

More precisely, we recall that a graph $G$ is hyperbolic in the sense of Gromov if there is a constant $k \geq 0$ such that all the triangles are $k$-thin in the following sense. Let $x$, $y$ and $z$ be vertices of $G$ and $\gamma_{xy}$, $\gamma_{yz}$, $\gamma_{zx}$ be geodesics from $x$ to $y$, from $y$ to $z$ and from $z$ to $x$. Then for any vertex $v$ on $\gamma_{xy}$, the graph distance between $v$ and $\gamma_{yz} \cup \gamma_{zx}$ is at most $k$. As usual, such a strong, uniform statement cannot hold for $\T_{\lambda}$ since any finite triangulation appears somewhere in $\T_{\lambda}$. Therefore, we need to study an "anchored" version.

\begin{defn}
Let $M$ be a planar map. We say that $M$ is \emph{weakly anchored hyperbolic} if there is $k \geq 0$ such that the following holds. Let $x$, $y$ and $z$ be three vertices of $M$, and let $\gamma_{xy}$ (resp. $\gamma_{yz}$, $\gamma_{zx}$) be a geodesic from $x$ to $y$ (resp. $y$ to $z$, $z$ to $x$). Assume the triangle formed by $\gamma_{xy}$, $\gamma_{yz}$ and $\gamma_{zx}$ surrounds the root vertex $\rho$. Then
\[ d \left( \rho, \gamma_{xy} \cup \gamma_{yz} \cup \gamma_{zx} \right) \leq k.\]
\end{defn}

Another property studied in \cite{B18} is the existence of bi-infinite geodesics, i.e. paths $\big( \gamma(i) \big)_{i \in \Z}$ such that for any $i$ and $j$, the graph distance between $\gamma(i)$ and $\gamma(j)$ is exactly $|i-j|$. This is not strictly speaking a hyperbolicity property, since such geodesics exist e.g. in $\Z^2$. However, they are expected to disappear after perturbations of the metric like first-passage percolation (see for example \cite{Kes86SF}). On the other hand, bi-infinite geodesics are much more stable in hyperbolic graphs \cite{BT17}. In \cite{B18}, we prove that any random planar map containing a supercritical Galton--Watson tree with no leaf is a.s. weakly anchored hyperbolic, and contains bi-infinite geodesics. In particular, the following result follows from Theorem \ref{thm1_GW}.

\begin{corr}
Let $0<\lambda<\lambda_c$. Almost surely, the map $\T_{\lambda}$ is weakly anchored hyperbolic and admits bi-infinite geodesics.
\end{corr}

The existence of bi-infinite geodesics answers a question of Benjamini and Tessera \cite{BT17}. Once again, there is a sharp contrast between the hyperbolic setting and "usual" random planar maps. For example, the results of \cite{CMMinfini} imply that such bi-infinite geodesics do not exist in the UIPQ.

\paragraph{Poisson boundary.}
Another goal of this work is to give a new description of the Poisson boundary of $\T_{\lambda}$ for $0<\lambda<\lambda_c$ in terms of the tree $\mathbf{T}^g_{\lambda}$. Let $G$ be an infinite, locally finite graph,  and let $G \cup \partial G$ be a compactification of $G$, i.e. a compact metric space in which $G$ is dense. Let also $(X_n)$ be the simple random walk on $G$. We say that $\partial G$ is a \emph{realization of the Poisson boundary} of $G$ if the following two properties hold:
\begin{itemize}
\item
$(X_n)$ converges a.s. to a point $X_{\infty} \in \partial G$,
\item
every bounded harmonic function $h$ on $G$ can be written in the form
\[h(x)=\mathbb{E}_x \left[ g(X_{\infty}) \right],\]
where $g$ is a bounded measurable function from $\partial G$ to $\R$.
\end{itemize}
A first realization of the Poisson boundary of $\T_ {\lambda}$ is given by a work of Angel, Hutchcroft, Nachmias and Ray \cite{AHNR15}: let $\partial_{CP} \T_{\lambda}$ be the boundary of the circle packing of $\T_{\lambda}$ in the unit disk. We may equip $\T_{\lambda} \cup \partial_{CP} \T_{\lambda}$ with the topology induced by the usual topology on the unit disk. Then almost surely, $\partial_{CP} \T_{\lambda}$ is a realization of the Poisson boundary of $\T_{\lambda}$. Moreover, almost surely, the distribution of the limit point $X_{\infty}$ has full support and no atoms in $\partial_{CP} \T_{\lambda}$.

We write $\partial \mathbf{T}^g_{\lambda}$ for the space of infinite rays of $\mathbf{T}^g_{\lambda}$. If $\gamma, \gamma' \in \partial \mathbf{T}^g_{\lambda}$, we write $\gamma \sim \gamma'$ if $\gamma=\gamma'$ or if $\gamma$ and $\gamma'$ are the left and right boundaries of the same strip. Then $\sim$ is a.s. an equivalence relation in which countably many equivalence classes have cardinal $2$, and all the others have cardinal $1$. We write $\widehat{\partial} \mathbf{T}^g_{\lambda}= \partial \mathbf{T}^g_{\lambda} / \! \sim$. There is a natural way to equip $\T_{\lambda} \cup \widehat{\partial} \mathbf{T}^g_{\lambda}$ with a topology that makes it a compact space, see Section 3.1 for more details. Hence, $\T_{\lambda} \cup \widehat{\partial} \mathbf{T}^g_{\lambda}$ can be seen as a compactification of the infinite graph $\T_{\lambda}$. We show that $\widehat{\partial} \mathbf{T}^g_{\lambda}$ is also a realization of the Poisson boundary.

\begin{thm}\label{thm2_Poisson}
Let $0<\lambda<\lambda_c$. Then almost surely:
\begin{enumerate}
\item
the limit $\lim X_n=X_{\infty}$ exists, and its distribution has full support and no atoms in $\widehat{\partial} \mathbf{T}^g_{\lambda}$,
\item
$\widehat{\partial} \mathbf{T}^g_{\lambda}$ is a realization of the Poisson boundary of $\T_{\lambda}$.
\end{enumerate}
\end{thm}

Note that, by a result of Hutchcroft and Peres \cite{HP17}, the second point will follow from the first one. Note also that once we have Theorem \ref{thm2_Poisson}, since the exit measure on $\widehat{\partial} \mathbf{T}^g_{\lambda}$ is nonatomic and only countably many pairs of points of $\partial \mathbf{T}^g_{\lambda}$ are identified, we also have almost sure convergence of $(X_n)$ to a point of $\partial \mathbf{T}^g_{\lambda}$.
Finally, we show in \cite{B18} that the existence of the limit $X_{\infty}$ and the fact that it has full support are true in the more general setting of planar maps obtained by "filling the faces" of a supercritical Galton--Watson tree with i.i.d. strips. However, we did not manage to prove non-atomicity and to obtain a precise description of the Poisson boundary in this general setting. Our proof of non-atomicity here uses an argument specific to $\T_{\lambda}$, based on its peeling process.

\paragraph{Geodesic rays in the hyperbolic Brownian plane.}
Finally, the last goal of this work is to take the scaling limit of Theorem \ref{thm1_GW} to obtain results about geodesics in continuum objects. Indeed, another purpose of the theory of random planar maps is to build continuum random surfaces. The first such surface that was introduced is the Brownian map \cite{LG11,Mie11}, which is now known to be the scaling limit of a wide class of finite planar maps conditioned to be large \cite{Ab13,AA13,BLG13,BJM13,CLGmodif,Mar16}. A noncompact version $\p$ called the Brownian plane was introduced in \cite{CLGplane} and is the scaling limit of the UIPQ and also of the UIPT \cite{B16}. Finally, it was shown in \cite{B16} that the hyperbolic random triangulations have a near-critical scaling limit called the hyperbolic Brownian plane and denoted $\hp$. More precisely, let $(\lambda_n)$ be a sequence of numbers in $(0,\lambda_c]$ satisfying
\[\lambda_n=\lambda_c \left( 1-\frac{2}{3n^4} +o \left( \frac{1}{n^4} \right) \right).\]
Then $\frac{1}{n} \T_{\lambda_n}$ converges for the local Gromov--Hausdorff distance to $\hp$. By taking the scaling limit of Theorem \ref{thm1_GW} and checking that the tree of infinite leftmost geodesics behaves well in the scaling limit, we obtain a precise description of the geodesic rays in $\hp$. Let $\mathbf{B}$ be the infinite tree in which every vertex has exactly two children, except the root which has only one.

\begin{thm} \label{thm3_hbp}
The infinite geodesic rays of $\hp$ form a tree $\mathbf{T}^g(\hp)$ that is distributed as a Yule tree with parameter $2 \sqrt{2}$, i.e. the tree $\mathbf{B}$ in which the lengths of the edges are i.i.d. exponential variables with parameter $2 \sqrt{2}$.
\end{thm}

Once again, this behaviour is very different from the non-hyperbolic setting: in the Brownian plane, there is only one geodesic ray (this is Proposition 15 of \cite{CLGplane}, and is an easy consequence of the local confluence of geodesics proved in \cite{LG09} for the Brownian map). We also note that the rate $2\sqrt{2}$ of exponential growth of $\mathbf{T}^g(\hp)$ is the same as the rate of exponential growth of the perimeters and volumes of the hulls of $\hp$ \cite[Corollary 1]{B16}.

\paragraph{The skeleton decomposition.}
Our main tool for proving these results will be the skeleton decomposition of planar triangulations introduced by Krikun \cite{Kri04} for the type-II UIPT. See also \cite{CLGmodif} for the adaptation to the (slightly easier) type-I case. This decomposition encodes a triangulation by a reverse forest, where leftmost geodesics from the root pass between the trees and between their branches. The infinite forest describing the UIPT consists of a single tree, which can be seen as a reverse Galton--Watson tree with critical offspring distribution, started at time $-\infty$, and conditioned to have exactly $1$ vertex at time $0$ and to die at time $1$. We obtain a similar description for the infinite forest encoding $\T_{\lambda}$ for $0<\lambda<\lambda_c$, but here the offspring distribution is subcritical. A key feature is that the forest now contains infinitely many infinite trees. The parts of $\T_{\lambda}$ described by each of these trees are the strips delimited by the leftmost geodesic rays. Finally, let us highlight that the construction of infinite reverse Galton--Watson trees that we present in Section \ref{subsec:construction_reversee_trees} holds for any subcritical or critical Galton--Watson process and might be of independent interest.

\paragraph{Structure of the paper.}
The structure of the paper is as follows. In Section 1, we recall some combinatorial results about planar triangulations, and we recall the definition and some basic properties of the maps $\T_{\lambda}$ and their halfplanar analogs. In Section 2, we prove Theorem \ref{thm1_GW} by computing the skeleton decomposition of $\T_{\lambda}$. Section 3 is devoted to the proof of Theorem \ref{thm2_Poisson}, and Section 4 to the proof of Theorem \ref{thm3_hbp}. In Appendix A, we prove a technical result needed in Section 4, which shows that a wide class of events related to geodesics are closed for the Gromov--Hausdorff distance.

\paragraph{Acknowledgments:} I thank Nicolas Curien for carefully reading several earlier versions of this manuscript, and Itai Benjamini for his question about bi-infinite geodesics. I am also grateful to the anonymous reviewers for their useful comments. I acknowledge the support of ANR Liouville (ANR-15-CE40-0013) and ANR GRAAL (ANR-14-CE25-0014).

\tableofcontents

\section{Combinatorics and preliminaries}

\subsection{Combinatorics}

For $n \geq 0$ and $p \geq 1$, a \emph{triangulation of the $p$-gon with $n$ inner vertices} is a planar map with $n+p$ vertices and a distinguished face called the \emph{outer face}, in which all faces except perhaps the outer face are triangles, and such that the boundary of the outer face is a simple cycle of length $p$. It is equipped with a root edge such that the outer face touches the root edge on its right.
We consider type-I triangulations, which means we allow triangulations containing loops and multiple edges. We denote by $\mathscr{T}_{n,p}$ the set of triangulations of the $p$-gon with $n$ inner vertices.

The number of triangulations with fixed volume and perimeter can be computed by a result of Krikun, as a special case of the main theorem of \cite{Kri07}:
\begin{equation}\label{enumeration}
\# \mathscr{T}_{n,p} = \frac{p(2p)!}{(p!)^2} \frac{4^{n-1} (2p+3n-5)!!}{n! (2p+n-1)!!} \underset{n \to +\infty}{\sim} c(p) \lambda_c^{-n} n^{-5/2},
\end{equation}
where $\lambda_c=\frac{1}{12 \sqrt{3}}$ and 
\begin{equation} \label{formula_Cp}
c(p) = \frac{3^{p-2} p (2p)!}{4 \sqrt{2 \pi} (p!)^2} \underset{p \to +\infty}{\sim} \frac{1}{36 \pi \sqrt{2}} 12^p \sqrt{p}.
\end{equation}
For $p \geq 1$ and $\lambda \geq 0$, we write $w_{\lambda}(p)=\sum_{n \geq 0} \# \mathscr{T}_{n,p} \lambda^n$. Note that by the asymptotics \eqref{enumeration}, we have $w_{\lambda}(p)<+\infty$ if and only if $\lambda \leq \lambda_c$. We finally write $ W_{\lambda}(x)=\sum_{p \geq 1} w_{\lambda}(p) x^p$. Formula (4) of \cite{Kri07} computes $W_{\lambda}$ after a simple change of variables:
\begin{equation}\label{G}
W_{\lambda}(x) = \frac{\lambda}{2} \bigg( \Big(1-\frac{1+8h}{h}x \Big) \sqrt{1-4(1+8h)x} -1+\frac{x}{\lambda}\bigg),
\end{equation}
where $h \in \big( 0, \frac{1}{4} \big]$ is given\footnote{Note that our $h$ corresponds to the $h^3$ of Krikun.} by \eqref{eqn_h_lambda}. From this formula, we easily get
\begin{equation}\label{Z_1}
w_{\lambda}(1)=\frac{1}{2}-\frac{1+2h}{2\sqrt{1+8h}}
\end{equation}
and, for $p \geq 2$,
\begin{equation}\label{Z_p}
w_{\lambda}(p)=( 2+16h)^{p}\frac{(2p-5)!!}{p!} \frac{(1-4h)p+6h}{4 (1+8h)^{3/2} }.
\end{equation}
We also define a \emph{Boltzmann triangulation of the $p$-gon with parameter $\lambda$} as a random triangulation $T$ such that $\P \left( T=t \right)=\frac{\lambda^n}{w_{\lambda}(p)}$ for every $n \geq 0$ and $t \in \mathscr{T}_{n,p}$.

\subsection{Planar and halfplanar hyperbolic type-I triangulations}

We recall from \cite{B16} the definition of the random triangulations $\T_{\lambda}$ for $0<\lambda \leq \lambda_c$. A \emph{finite triangulation with a hole of perimeter $p$} is a rooted map with a distinguished face called the \emph{hole} in which all the faces are triangles except perhaps the hole, and where the boundary of the hole is a simple cycle of length $p$. The difference with a triangulation of the $p$-gon is that we do not require the root to lie on the boundary.
Let $t$ be a finite triangulation with a hole of perimeter $p$ and let $T$ be an infinite, one-ended triangulation of the plane. We write $t \subset T$ if $T$ can be obtained by filling the hole of $t$ with an infinite triangulation of the $p$-gon. For $0 < \lambda \leq \lambda_c$, the distribution of $\T_{\lambda}$ can be characterized as follows. For any finite triangulation $t$ with a hole of perimeter $p$, we have
\[ \mathbb{P} \left( t \subset \T_{\lambda} \right)=c_{\lambda}(p) \lambda^{|t|}, \]
where $|t|$ is the total number of vertices of $t$ and
\begin{equation}\label{Cformula}
c_{\lambda}(p)=\frac{1}{\lambda} \Big( 8+\frac{1}{h}\Big)^{p-1} \sum_{q=0}^{p-1} \binom{2q}{q} h^{q},
\end{equation}
where $h$ is as in \eqref{eqn_h_lambda}. Equivalently, we can compute the generating function
\begin{equation}\label{generating_cplambda}
C_{\lambda}(x)=\sum_{p \geq 1} c_{\lambda}(p) x^p=\frac{x}{\lambda \Big(1-\frac{1+8h}{h}x \Big) \sqrt{1-4(1+8h)x}}.
\end{equation}
Note that the numbers $c_{\lambda_c}(p)$ are equal to the $c(p)$ defined by \eqref{formula_Cp} and $\T_{\lambda_c}$ corresponds to the type-I UIPT \cite{CLGmodif,St18}. As in the type-II case \cite{CurPSHIT}, the triangulations $\T_{\lambda}$ exhibit a spatial Markov property similar to that of the UIPT: if $t$ is a finite triangulation with a hole of perimeter $p$, conditionally on $t \subset \T_{\lambda}$, the distribution of the infinite triangulation that fills the hole of $t$ only depends on $p$. We denote by $\T^p_{\lambda}$ a triangulation with this distribution. By a simple root transformation (more precisely duplicating the root edge, adding a loop in between and rooting the map at this new loop, see Figure 2 of \cite{CLGmodif}), triangulations of the plane are equivalent to infinite triangulations of the $1$-gon. In particular, the image of $\T_{\lambda}$ under this root transformation is $\T_{\lambda}^1$, so studying one or the other are equivalent. In particular, the root transformation does not affect leftmost geodesics from the root, so all the results we will first obtain about geodesic rays in $\T_{\lambda}^1$ are immediate to transfer to $\T_{\lambda}$.

We also define the halfplanar analog of $\T_{\lambda}$, that we will denote by $\mathbb{H}_{\lambda}$. The goal of the next paragraphs is to adapt to the type-I setting results from \cite{AR13} and \cite{CurPSHIT} that are already well-known in the type-II setting. This will only be used in Section 3 to study the random walk on $\T_{\lambda}$, so all the rest of Section 1 can be skipped in the first read-through. A \emph{triangulation of the halfplane} is an infinite planar map in which all the faces are triangles except one called the \emph{outer face}, whose boundary is simple and infinite. Triangulations of the halfplane are rooted in such a way that the root edge touches the outer face on its right.
We note that Angel and Ray build in \cite{AR13} a family $(\mathbb{H}^{\mathrm{II}}_{\alpha})_{2/3<\alpha<1}$ of hyperbolic triangulations of the halfplane in the type-II setting and explain in Section 3.4 how to "add loops" to obtain type-I triangulations. The triangulations $\mathbb{H}_{\lambda}$ we define are a particular case of their construction. However, as in \cite{B16} and in order to limit the computations, we prefer to construct the maps $\mathbb{H}_{\lambda}$ directly instead of relying on the type-II case.

The law of $\mathbb{H}_{\lambda}$ is characterized by the following. Let $t$ be a triangulation of the $p$-gon with a marked segment of edges $\partial_{\mathrm{out}} t$ on its boundary, such that $\partial_{\mathrm{out}} t$ contains the root edge. Such triangulations will be called \emph{marked triangulations}. Let $\partial_{\mathrm{in}} t=\partial t \backslash \partial_{\mathrm{out}} t$. We write $|t|_{\mathrm{in}}$ for the number of vertices of $t$ that do not lie on $\partial_{\mathrm{out}} t$. We also write $| \partial_{\mathrm{in}} t|$ (resp. $| \partial_{\mathrm{out}} t |$) for the number of edges on $\partial_{\mathrm{in}} t$ (resp. $\partial_{\mathrm{out}} t$). We write $t \subset \H_{\lambda}$ if $\H_{\lambda}$ can be obtained by gluing a triangulation of the halfplane $H$ to $t$ in such a way that $\partial t \cap \partial H=\partial_{\mathrm{in}} t$.

\begin{prop}
For any $0< \lambda \leq \lambda_c$, there is a unique (in distribution) random triangulation $\mathbb{H}_{\lambda}$ of the halfplane such that for every marked triangulation $t$, we have
\[\P \left( t \subset \mathbb{H}_{\lambda} \right)= \left( 8+\frac{1}{h} \right)^{|\partial_{\mathrm{in}} t| - |\partial_{\mathrm{out}} t|} \lambda^{|t|_{\mathrm{in}}}. \]
\end{prop}

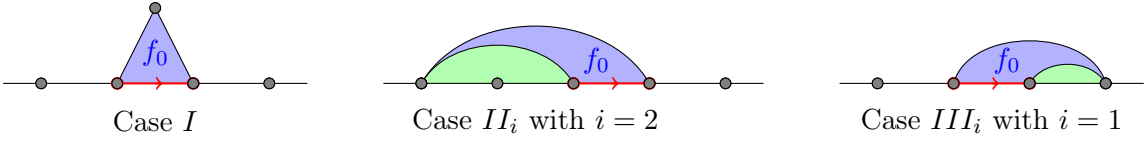
\begin{figure}
\begin{center}
\begin{tikzpicture}
\fill[blue!30] (1,0) -- (1.5,1)--(2,0)--(1,0);
\draw (-0.5,0) -- (0,0) node{};
\draw (0,0) node{} -- (1,0) node{};
\draw[red, thick, ->] (1,0) node{} -- (1.6,0);
\draw[red, thick] (1.5,0) -- (2,0) node{};
\draw (2,0) node{} -- (3,0) node{};
\draw (3,0) node{} -- (3.5,0);
\draw (1,0) node{}-- (1.5,1) node{};
\draw (2,0) node{}-- (1.5,1)node{};
\draw[blue] (1.5,0.4) node[texte]{$f_0$};
\draw (1.5,-0.5) node[texte]{Case $I$};

\begin{scope}[shift={(6,0)}]
\fill[blue!30] (1,0) to[bend right=60] (-1,0) to[bend left=60] (2,0)--(1,0);
\fill[green!30] (-1,0) to[bend left=60] (1,0)--(-1,0);
\draw (-1.5,0) -- (-1,0) node{};
\draw (-1,0) node{} -- (0,0) node{};
\draw (0,0) node{} -- (1,0) node{};
\draw[red, thick, ->] (1,0) node{} -- (1.6,0);
\draw[red, thick] (1.5,0) -- (2,0) node{};
\draw (2,0) node{} -- (3,0) node{};
\draw (3,0) node{} -- (3.5,0);
\draw (1,0) node{} to[bend right=60] (-1,0) node{};
\draw (2,0) node{} to[bend right=60] (-1,0)node{};
\draw[blue] (1.3,0.3) node[texte]{$f_0$};
\draw (0.5,-0.5) node[texte]{Case $II_i$ with $i=2$};
\end{scope}

\begin{scope}[shift={(11,0)}]
\fill[blue!30] (1,0) to[bend left=75] (3,0) to[bend right=60] (2,0)--(1,0);
\fill[green!30] (2,0) to[bend left=60] (3,0)--(2,0);
\draw (-0.5,0) -- (0,0) node{};
\draw (0,0) node{} -- (1,0) node{};
\draw[red, thick, ->] (1,0) node{} -- (1.6,0);
\draw[red, thick] (1.5,0) -- (2,0) node{};
\draw (2,0) node{} -- (3,0) node{};
\draw (3,0) node{} -- (3.5,0);
\draw (1,0) node{} to[bend left=75] (3,0) node{};
\draw (2,0) node{} to[bend left=60] (3,0)node{};
\draw[blue] (1.7,0.3) node[texte]{$f_0$};
\draw (1.5,-0.5) node[texte]{Case $III_i$ with $i=1$};
\end{scope}
\end{tikzpicture}
\end{center}
\caption{The three cases of peeling. In the two last cases, the parameter $i \geq 0$ corresponds to the number of edges of $\partial \mathbb{H}_{\lambda}$ that $f_0$ separates from infinity.} \label{construction_peeling}
\end{figure}
\begin{proof}
Uniqueness is stanard. To prove existence, we construct this triangulation by peeling along the same lines as in \cite{AR13}. If $\mathbb{H}_{\lambda}$ exists, let $f_0$ be its triangular face that is adjacent to the root. Then $f_0$ has one of the three forms described by Figure \ref{construction_peeling}.
Moreover, we have $\P \left( \mbox{Case $I$ occurs} \right) =\lambda \left(8+\frac{1}{h} \right)=\frac{1}{\sqrt{1+8h}}$. By summing over all possible ways to fill the green zone, we also have
\[ \P \left( \mbox{Case $II_i$ occurs} \right)=\P \left( \mbox{Case $III_i$ occurs} \right)=\left( 8+\frac{1}{h}\right)^{-i} w_{\lambda}(i+1) \]
for all $i \geq 0$. If we sum up these probabilities, we obtain
\begin{equation}
\frac{1}{\sqrt{1+8h}}+2 \sum_{i \geq 0} \left( 8+\frac{1}{h} \right)^{-i} w_{\lambda}(i+1) = \frac{1}{\sqrt{1+8h}}+2\left( 8+\frac{1}{h} \right) W_{\lambda} \left( \frac{h}{1+8h} \right)=1
\end{equation}
by \eqref{G}. Since these probabilities sum up to $1$, we can construct $\mathbb{H}_{\lambda}$ by peeling with the transitions described above. Everytime case $II_i$ or $III_i$ occurs, we fill the green bounded region with a Boltzmann triangulation of the $(i+1)$-gon with parameter $\lambda$. As in \cite{AR13} (see also \cite{CurPSHIT} in the planar case), we can check that we indeed obtain a triangulation of the halfplane, that its distribution does not depend on the choice of the peeling algorithm, and that the random triangulation we obtain has the right distribution.
\end{proof}


We now state a coupling result between $\T_{\lambda}$ and $\mathbb{H}_{\lambda}$ similar to the one stated in \cite{ANR14} and (implicitly) \cite{CurPSHIT} in the type-II case. We recall that a \emph{peeling algorithm} is a way to assign to every triangulation with a hole an edge on its boundary (see e.g. \cite[Section 1.3]{CurPSHIT}). To any peeling algorithm is naturally associated a filled-in exploration of $\T_{\lambda}$. By \emph{filled-in}, we mean that every time the face just explored separates a finite region from infinity, the interior of the finite region is entirely discovered.

\begin{lem}\label{halfplane_peeling}
\begin{itemize}
\item[(i)]
For $0 < \lambda \leq \lambda_c$, the triangulation $\mathbb{H}_{\lambda}$ is the local limit as $p \to +\infty$ of $\T_{\lambda}^p$.
\item[(ii)]
For $0 < \lambda<\lambda_c$, consider a filled-in peeling algorithm $\mathscr{A}$ with infinitely many peeling steps, and let $T$ be the part of $\T_{\lambda}$ that is discovered by $\mathscr{A}$. Let $(T_i)_{i \in I}$ be the infinite connected components of $\T_{\lambda} \backslash T$ and, for every $i \in I$, let $e_i$ be an edge of $\partial T$ that is glued to $T_i$, chosen in a way that only depends on $T$. Then conditionally on $T$, the maps $T_i$ rooted at $e_i$ are independent copies of $\mathbb{H}_{\lambda}$.
\end{itemize}
\end{lem}

\begin{proof}
\begin{itemize}
\item[(i)]
If $p \geq 1$ and $t$ is a marked triangulation with $|\partial_{\mathrm{out}} t| \leq p$, let $t_0$ be a triangulation with a hole of perimeter $p$ and let $t_0+t$ be a triangulation obtained by gluing $\partial_{\mathrm{out}} t$ to a segment of $\partial t_0$. Then by the definition of $\T_{\lambda}^p$, we have
\begin{eqnarray*}
\P \left( t \subset \T_{\lambda}^p \right)=\frac{ \P \left( t_0+t \subset \T_{\lambda} \right)}{\P \left( t_0 \subset \T_{\lambda} \right)} &=& \frac{c_{\lambda}(p+|\partial_{\mathrm{in}} t|-|\partial_{\mathrm{out}} t|)} {c_{\lambda}(p)} \lambda^{|t|_{\mathrm{in}}}\\
&\xrightarrow[p \to +\infty]{}& \left( 8+\frac{1}{h} \right)^{|\partial_{\mathrm{in}} t| - |\partial_{\mathrm{out}} t|} \lambda^{|t|_{\mathrm{in}}}\\
&=& \P \left( t \subset \mathbb{H}_{\lambda} \right),
\end{eqnarray*}
which is enough to conclude.
\item[(ii)]
We first note that there are infinitely many peeling steps and all the finite holes are filled-in, so every connected component of $\mathbb{T}_{\lambda} \backslash T$ is halfplanar. The second point then follows from the first one since the perimeter of the region discovered after $n$ peeling steps a.s. goes to $+\infty$ as $n \to +\infty$. See the proof of Lemma 2.16 in \cite{ANR14} for the same result in the type-II case. Note that for $\lambda= \lambda_c$, we have $T=\T_{\lambda}$ a.s. by Corollary 7 of \cite{CLGpeeling}, so the statement (ii) is irrelevant.
\end{itemize}
\end{proof}

\section{The skeleton decomposition of hyperbolic triangulations}

The aim of this section is to prove Theorem \ref{thm1_GW}. It is organized as follows. In Section 2.1, we describe the finite skeleton decomposition, which associates to every finite triangulation a finite forest, and its infinite counterpart. In Section 2.2, we compute the distribution of the skeletons of the hulls of $\T_{\lambda}$. This characterizes the skeleton of $\T_{\lambda}$ entirely, but in a form that is not convenient for the proof of Theorem \ref{thm1_GW}. In Section 2.3, we explain why the infinite skeleton of $\T_{\lambda}$ is related to infinite leftmost geodesics in the triangulation. Section 2.4 contains a description of the strips $S^0_{\lambda}$ and $S^1_{\lambda}$ by the distribution of their hulls, without a proof of their existence. In Section 2.5, we use all that precedes to prove Theorem \ref{thm1_GW}. Finally, Section 2.6 is devoted to the construction of $S^0_{\lambda}$ and $S^1_{\lambda}$.

\subsection{The skeleton decomposition of finite and infinite triangulations}

\paragraph{The finite setting: skeleton decomposition of triangulations of the cylinder.}
We first recall the skeleton decomposition of triangulations introduced by Krikun \cite{Kri04, Kri05} for type-II triangulations and quadrangulations, and described in \cite{CLGmodif} for type-I triangulations (see also \cite{AR18,M16}). This decomposition applies to so-called triangulations of the cylinder. Most of the presentation here is adapted from \cite{CLGmodif}.

\begin{defn}
Let $r \geq 1$. A \emph{triangulation of the cylinder of height $r$} is a rooted planar map in which all faces are triangles except two distinguished faces called the top and the bottom faces, such that the following properties hold. The boundaries of the top and bottom faces are simple cycles. The bottom face lies on the right of the root edge. Finally, every vertex incident to the top face is at distance $r$ from the boundary of the bottom face, and every edge adjacent to the top face is also adjacent to a face whose third vertex is at distance $r-1$ from the boundary of the bottom face.
\end{defn}

If $\Delta$ is a triangulation of the cylinder of height $r$, we write $\partial \Delta$ and $\partial_* \Delta$ for the boundaries of the bottom and top faces. Let $p=|\partial \Delta|$ and $q=|\partial_* \Delta|$. The skeleton decomposition encodes $\Delta$ by a forest of $q$ plane trees and a family of triangulations of polygons indexed by the vertices of this forest. For $1 \leq j \leq r-1$, we define the ball $B_j(\Delta)$ to be the map formed by all the faces of $\Delta$ having at least one vertex at distance at most $j-1$ from $\partial \Delta$, along with their vertices and edges. We also define the hull $B_j^{\bullet} (\Delta)$ as the union of $B_j(\Delta)$ and all the connected components of its complement, except the one that contains $\partial_* \Delta$. It is easy to see that $B_j^{\bullet} (\Delta)$ is a triangulation of the cylinder of height $j$. We denote by $\partial_j \Delta$ the top boundary of $B_j^{\bullet}(\Delta)$, with the conventions $\partial_0 \Delta=\partial \Delta$ and $\partial_r \Delta=\partial_* \Delta$.

If $1 \leq j \leq r$, every edge of $\partial_j \Delta$ is incident to exactly one triangle whose third vertex belongs to $\partial_{j-1} \Delta$. Such triangles are called \emph{downward triangles at height $j$}. We can define a genealogy on $\bigcup_{j=0}^r \partial_j \Delta$ by saying that $e \in \partial_j \Delta$ for $j \geq 1$ is the parent of $e' \in \partial_{j-1} \Delta$ if the downward triangle adjacent to $e$ is the first one that one encounters when moving along $\partial_{j-1} \Delta$ in clockwise order starting from the middle of the edge $e'$ (see Figure \ref{Skeleton_decomposition}). We obtain $q$ trees rooted on $\partial_* \Delta$. Let $(t_1, \dots, t_q)$ be the forest obtained by listing these trees in clockwise order in such a way that the root edge of $\Delta$ lies in $t_1$. This forest is called the \emph{skeleton of $\Delta$} and we denote it by $\mathrm{Skel}(\Delta)$. Note that $t_1$ has a distinguished vertex at height $r$. The set of possible values of $\mathrm{Skel}(\Delta)$ is called the set of \emph{$(p,q,r)$-admissible forests} and is described by the next definition.

\begin{defn}\label{admissible}
Let $p,q,r \geq 1$: a $(p,q,r)$-pre-admissible forest is a sequence $f=(t_1, \dots, t_q)$ of plane trees equipped with a distinguished vertex $\rho$ such that:
\begin{itemize}
\item
the maximal height of the trees $t_i$ is $r$,
\item
the total number of vertices at height $r$ in the trees $t_i$ is $p$,
\item
$\rho$ lies at height $r$.
\end{itemize}
We write $\F_{p,q,r}$ for the set of $(p,q,r)$-pre-admissible forests. If furthermore $\rho \in t_1$, we say that $f$ is $(p,q,r)$-admissible, and we write $\F'_{p,q,r}$ for the set of $(p,q,r)$-admissible forests.
\end{defn}

Let $f=(t_1, \dots, t_p) \in \F_{p,q,r}$. Most of the time, we will represent $f$ with the roots of the trees $t_i$ on the top. Hence, if $x \in t_i$ is a vertex of $f$, we define the \emph{reverse height of $x$ in $f$} as $r$ minus the height of $x$ in $t_i$, and we write it $h_f^{\mathrm{rev}}(v)$. In particular, the roots of the trees $t_i$ have reverse height $r$ and $\rho$ has reverse height $0$. Although quite unusual, this convention is natural because the reverse heights in $\mathrm{Skel}(\Delta)$ match the distances to the root in the triangulation $\Delta$. It will also be more convenient when we will deal with reverse forests with infinite height.

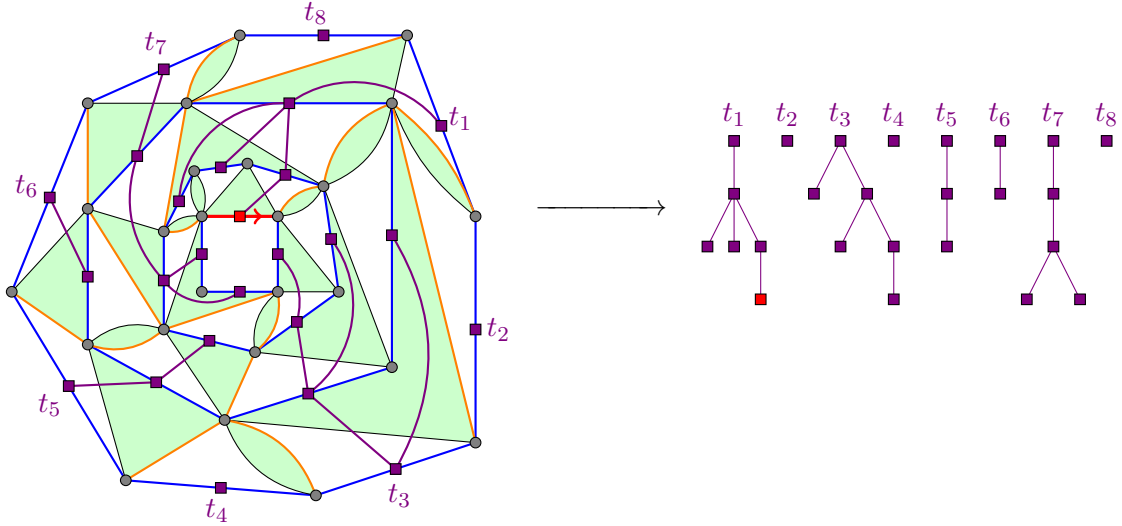
\begin{figure}
\begin{center}
\begin{tikzpicture}
\fill[green!20] (-0.5,-0.5)--(0,1)--(0,0)--(1,0);
\fill[green!20] (1.8,0)--(1,1)--(1,0);
\fill[green!20] (0,1)--(1,1)--(0.6,1.7);
\fill[green!20] (1.6,1.4)--(2.5,-1)--(0.7,-0.8)--(1.8,0);
\fill[green!20] (0.3,-1.7)--(0.7,-0.8)--(-0.5,-0.5);
\fill[green!20] (-1.5,1.1)--(-0.5,0.8)--(-0.5,-0.5);
\fill[green!20] (-0.5,0.8)--(-0.2,2.5)--(1.6,1.4)--(0.6,1.7)--(-0.1,1.6);
\fill[green!20] (-0.2,2.5)--(2.7,3.4)--(2.5,2.5);
\fill[green!20] (0.3,-1.7)--(3.6,-2)--(2.5,2.5)--(2.5,-1);
\fill[green!20] (0.3,-1.7)--(-1,-2.5)--(-1.5,-0.7);
\fill[green!20] (-1.5,-0.7)--(-2.5,0)--(-1.5,1.1);
\fill[green!20] (-1.5,1.1)--(-1.5,2.5)--(-0.2,2.5);
\fill[green!20] (0,1)to[bend left](-0.1,1.6)to[bend left](0,1);
\fill[green!20] (0,1)to[bend left](-0.5,0.8)to[bend left](0,1);
\fill[green!20] (1,0)to[bend left](0.7,-0.8)to[bend left](1,0);
\fill[green!20] (1,1)to[bend left](1.6,1.4)to[bend left](1,1);
\fill[green!20] (2.5,2.5)to[bend right](1.6,1.4)to[bend right] (2.5,2.5);
\fill[green!20] (-1.5,-0.7)to[bend right](-0.5,-0.5)to[bend right] (-1.5,-0.7);
\fill[green!20](2.5,2.5)to[bend left=15](3.6,1)to[bend left=15](2.5,2.5);
\fill[green!20](0.3,-1.7)to[bend left](1.5,-2.7)to[bend left](0.3,-1.7);
\fill[green!20](-0.2,2.5)to[bend left](0.5,3.4)to[bend left](-0.2,2.5);

\draw[red, very thick, ->] (0,1) --(0.8,1);
\draw[red, very thick] (0.7,1) --(1,1);
\draw[blue, thick] (0,0) --(1,0);
\draw[blue, thick] (0,0) --(0,1);
\draw[blue, thick] (1,1) --(1,0);

\draw[blue, thick] (-0.5,-0.5)--(0.7,-0.8);
\draw[blue, thick] (1.8,0)--(0.7,-0.8);
\draw[blue, thick] (1.8,0)--(1.6,1.4);
\draw[blue, thick] (0.6,1.7)--(1.6,1.4);
\draw[blue, thick] (0.6,1.7)--(-0.1,1.6);
\draw[blue, thick] (-0.5,0.8)--(-0.1,1.6);
\draw[blue, thick] (-0.5,0.8)--(-0.5,-0.5);

\draw[blue, thick] (-1.5,1.1)--(-1.5,-0.7);
\draw[blue, thick] (-1.5,1.1)--(-0.2,2.5);
\draw[blue, thick] (2.5,2.5)--(-0.2,2.5);
\draw[blue, thick] (2.5,2.5)--(2.5,-1);
\draw[blue, thick] (0.3,-1.7)--(2.5,-1);
\draw[blue, thick] (0.3,-1.7)--(-1.5,-0.7);

\draw[blue, thick] (-2.5,0)--(-1.5,2.5);
\draw[blue, thick] (0.5,3.4)--(-1.5,2.5);
\draw[blue, thick] (0.5,3.4)--(2.7,3.4);
\draw[blue, thick] (3.6,1)--(2.7,3.4);
\draw[blue, thick] (3.6,1)--(3.6,-2);
\draw[blue, thick] (1.5,-2.7)--(3.6,-2);
\draw[blue, thick] (1.5,-2.7)--(-1,-2.5);
\draw[blue, thick] (-2.5,0)--(-1,-2.5);

\draw(0,1)--(0.6,1.7);
\draw(0,1)to[bend left](-0.1,1.6);
\draw(0,1)to[bend right](-0.1,1.6);
\draw[thick, orange](0,1)to[bend left](-0.5,0.8);
\draw(0,1)to[bend right](-0.5,0.8);
\draw(0,1)--(-0.5,-0.5);
\draw[thick, orange](1,0)--(-0.5,-0.5);
\draw[thick, orange](1,0)to[bend left](0.7,-0.8);
\draw(1,0)to[bend right](0.7,-0.8);
\draw(1,0)--(1.8,0);
\draw(1,1)--(1.8,0);
\draw[thick, orange](1,1)to[bend left](1.6,1.4);
\draw(1,1)to[bend right](1.6,1.4);
\draw(1,1)--(0.6,1.7);

\draw(-0.2,2.5)--(1.6,1.4);
\draw[thick, orange](2.5,2.5)to[bend right](1.6,1.4);
\draw(2.5,2.5)to[bend left](1.6,1.4);
\draw(2.5,-1)--(1.6,1.4);
\draw(2.5,-1)--(0.7,-0.8);
\draw[thick, orange](0.3,-1.7)--(0.7,-0.8);
\draw(0.3,-1.7)--(-0.5,-0.5);
\draw[thick, orange](-1.5,-0.7)to[bend right](-0.5,-0.5);
\draw(-1.5,-0.7)to[bend left](-0.5,-0.5);
\draw[thick, orange](-1.5,1.1)--(-0.5,-0.5);
\draw(-1.5,1.1)--(-0.5,0.8);
\draw[thick, orange](-0.2,2.5)--(-0.5,0.8);

\draw[thick, orange](-0.2,2.5)--(2.7,3.4);
\draw(2.5,2.5)--(2.7,3.4);
\draw[thick, orange](2.5,2.5)to[bend left=15](3.6,1);
\draw(2.5,2.5)to[bend right=15](3.6,1);
\draw[thick, orange](2.5,2.5)--(3.6,-2);
\draw(0.3,-1.7)--(3.6,-2);
\draw[thick, orange](0.3,-1.7)to[bend left](1.5,-2.7);
\draw(0.3,-1.7)to[bend right](1.5,-2.7);
\draw[thick, orange](0.3,-1.7)--(-1,-2.5);
\draw(-1.5,-0.7)--(-1,-2.5);
\draw[thick, orange](-1.5,-0.7)--(-2.5,0);
\draw(-1.5,1.1)--(-2.5,0);
\draw[thick, orange](-1.5,1.1)--(-1.5,2.5);
\draw(-0.2,2.5)--(-1.5,2.5);
\draw[thick, orange](-0.2,2.5)to[bend left](0.5,3.4);
\draw(-0.2,2.5)to[bend right](0.5,3.4);

\draw[violet, thick] (0,0.5)--(-0.5,0.15);
\draw[violet, thick] (0.5,1)--(1.1,1.55);
\draw[violet, thick] (1,0.5) to[bend left] (1.25,-0.4);
\draw[violet, thick] (0.5,0) to[bend left=45] (-0.5,0.15);

\draw[violet, thick] (0.1,-0.65)--(-0.6,-1.2);
\draw[violet, thick] (1.25,-0.4)--(1.4,-1.35);
\draw[violet, thick] (1.7,0.7) to[bend left=45] (1.4,-1.35);
\draw[violet, thick] (1.1,1.55)--(1.15,2.5);
\draw[violet, thick] (0.25,1.65)--(1.15,2.5);
\draw[violet, thick] (-0.3,1.2) to[bend left=45] (1.15,2.5);
\draw[violet, thick] (-0.5,0.15) to[bend left] (-0.85,1.8);

\draw[violet, thick] (-1.5,0.2)--(-2,1.25);
\draw[violet, thick] (-0.85,1.8)--(-0.5,2.95);
\draw[violet, thick] (1.15,2.5) to[bend left=45] (3.15,2.2);
\draw[violet, thick] (2.5,0.75) to[bend left] (2.55,-2.35);
\draw[violet, thick] (1.4,-1.35)--(2.55,-2.35);
\draw[violet, thick] (-0.6,-1.2)--(-1.75,-1.25);

\draw (0,0.5) node[edge]{};
\draw (0.5,1) node[rootedge]{};
\draw (1,0.5) node[edge]{};
\draw (0.5,0) node[edge]{};

\draw (0.1,-0.65) node[edge]{};
\draw (1.25,-0.4) node[edge]{};
\draw (1.7,0.7) node[edge]{};
\draw (1.1,1.55) node[edge]{};
\draw (0.25,1.65) node[edge]{};
\draw (-0.3,1.2) node[edge]{};
\draw (-0.5,0.15) node[edge]{};

\draw (-1.5,0.2) node[edge]{};
\draw (1.15,2.5) node[edge]{};
\draw (2.5,0.75) node[edge]{};
\draw (1.4,-1.35) node[edge]{};
\draw (-0.6,-1.2) node[edge]{};
\draw (-0.85,1.8) node[edge]{};

\draw (-2,1.25) node[edge]{};
\draw (-0.5,2.95) node[edge]{};
\draw (1.6,3.4) node[edge]{};
\draw (3.15,2.2) node[edge]{};
\draw (3.6,-0.5) node[edge]{};
\draw (2.55,-2.35) node[edge]{};
\draw (0.25,-2.6) node[edge]{};
\draw (-1.75,-1.25) node[edge]{};

\draw (0,0)node{};
\draw (0,1)node{};
\draw (1,1)node{};
\draw (1,0)node{};

\draw (-0.5,-0.5)node{};
\draw (0.7,-0.8)node{};
\draw (1.8,0)node{};
\draw (1.6,1.4)node{};
\draw (0.6,1.7)node{};
\draw (-0.1,1.6)node{};
\draw (-0.5,0.8)node{};

\draw (-1.5,-0.7)node{};
\draw (-1.5,1.1)node{};
\draw (-0.2,2.5)node{};
\draw (2.5,2.5)node{};
\draw (2.5,-1)node{};
\draw (0.3,-1.7)node{};

\draw (-2.5,0)node{};
\draw (-1.5,2.5)node{};
\draw (0.5,3.4)node{};
\draw (2.7,3.4)node{};
\draw (3.6,1)node{};
\draw (3.6,-2)node{};
\draw (1.5,-2.7)node{};
\draw (-1,-2.5)node{};

\draw[violet] (3.4,2.3) node[texte]{$t_1$};
\draw[violet] (3.9,-0.5) node[texte]{$t_2$};
\draw[violet] (2.6,-2.7) node[texte]{$t_3$};
\draw[violet] (0.2,-2.9) node[texte]{$t_4$};
\draw[violet] (-2,-1.5) node[texte]{$t_5$};
\draw[violet] (-2.3,1.4) node[texte]{$t_6$};
\draw[violet] (-0.6,3.3) node[texte]{$t_7$};
\draw[violet] (1.5,3.7) node[texte]{$t_8$};

\begin{scope}[shift={(7,2)}, scale=0.7]

\draw[violet] (0,0)--(0,-1);
\draw[violet] (0,-1)--(-0.5,-2);
\draw[violet] (0,-1)--(0,-2);
\draw[violet] (0,-1)--(0.5,-2);
\draw[violet] (0.5,-2)--(0.5,-3);
\draw[violet] (2,0)--(1.5,-1);
\draw[violet] (2,0)--(2.5,-1);
\draw[violet] (2.5,-1)--(2,-2);
\draw[violet] (2.5,-1)--(3,-2);
\draw[violet] (3,-2)--(3,-3);
\draw[violet] (4,0)--(4,-1);
\draw[violet] (4,-1)--(4,-2);
\draw[violet] (5,0)--(5,-1);
\draw[violet] (6,0)--(6,-1);
\draw[violet] (6,-1)--(6,-2);
\draw[violet] (6,-2)--(5.5,-3);
\draw[violet] (6,-2)--(6.5,-3);

\draw (-2.5,-1.5) node[texte]{$\xrightarrow[\hspace{1.5cm}]{}$};
\draw (0,0) node[edge]{};
\draw (0,-1) node[edge]{};
\draw (-0.5,-2) node[edge]{};
\draw (0,-2) node[edge]{};
\draw (0.5,-2) node[edge]{};
\draw (0.5,-3) node[rootedge]{};
\draw (1,0) node[edge]{};
\draw (2,0) node[edge]{};
\draw (1.5,-1) node[edge]{};
\draw (2.5,-1) node[edge]{};
\draw (2,-2) node[edge]{};
\draw (3,-2) node[edge]{};
\draw (3,-3) node[edge]{};
\draw (3,0) node[edge]{};
\draw (4,0) node[edge]{};
\draw (4,-1) node[edge]{};
\draw (4,-2) node[edge]{};
\draw (5,0) node[edge]{};
\draw (5,-1) node[edge]{};
\draw (6,0) node[edge]{};
\draw (6,-1) node[edge]{};
\draw (6,-2) node[edge]{};
\draw (5.5,-3) node[edge]{};
\draw (6.5,-3) node[edge]{};
\draw (7,0) node[edge]{};

\draw[violet] (0,0.5) node[texte]{$t_1$};
\draw[violet] (1,0.5) node[texte]{$t_2$};
\draw[violet] (2,0.5) node[texte]{$t_3$};
\draw[violet] (3,0.5) node[texte]{$t_4$};
\draw[violet] (4,0.5) node[texte]{$t_5$};
\draw[violet] (5,0.5) node[texte]{$t_6$};
\draw[violet] (6,0.5) node[texte]{$t_7$};
\draw[violet] (7,0.5) node[texte]{$t_8$};
\end{scope}

\end{tikzpicture}
\end{center}
\caption{A triangulation of the cylinder $\Delta$ and its skeleton $\mathrm{Skel}(\Delta)  \in \F'_{4,8,3}$ in purple. The cycles $\partial_j \Delta$ are in blue, and the root edge is in red. The leftmost geodesics from the vertices of $\partial_* \Delta$ to $\partial \Delta$ are in orange. The green holes must be filled by triangulations of polygons.} \label{Skeleton_decomposition}
\end{figure}

The forest $\mathrm{Skel}(\Delta)$ is not enough to completely describe $\Delta$: if we consider all the downward triangles of $\Delta$, there is a family of holes, each of which is naturally associated to an edge of $\bigcup_{j=0}^r \partial_j \Delta$. If $e \in \partial_j \Delta$ with $1 \leq j \leq r$, the associated hole is bounded by the edges of $\partial_{j-1} \Delta$ that are children of $e$ and by two vertical edges connecting the initial vertex of $e$ to two vertices of $\partial_{j-1} \Delta$. This hole has perimeter $c_e+2$, where $c_e$ is the number of children of $e$, so it must be filled by a triangulation of a $(c_e+2)$-gon. If $c_e=0$, it is possible that the hole of perimeter $2$ is filled by the triangulation of the $2$-gon consisting of a single edge, which means that the two vertical edges are simply glued together.

If $f$ is a $(p,q,r)$-admissible forest, let $f^*$ be the set of those vertices $v$ of $f$ such that $h^{\mathrm{rev}}_f(v)>0$. The decomposition we just described is a bijection between triangulations of the cylinder $\Delta$ with height $r$ such that $\partial \Delta=p$ and $\partial_* \Delta=q$, and pairs consisting of a $(p,q,r)$-admissible forest $f$ and a family $(M_v)_{v \in f^*}$ of maps such that $M_v$ is a finite triangulation of a $(c_v+2)$-gon for every $v$.

Moreover, this decomposition encodes informations about leftmost geodesics from the vertices of $\partial_* \Delta$ to $\partial \Delta$: as can be seen on Figure \ref{Skeleton_decomposition}, these geodesics (in orange on Figure \ref{Skeleton_decomposition}) are the paths going between the trees of $\mathrm{Skel}(\Delta)$. These geodesics cut $\Delta$ into $q$ slices, each of which contains an edge of $\partial_* \Delta$. Moreover, the slice containing the $i$-th edge of $\partial_* \Delta$ can be completely described by the $i$-th tree $t_i$ of $\mathrm{Skel}(\Delta)$ and the maps $M_v$ for $v \in t_i$.

\paragraph{The infinite setting: infinite reverse forests.}
If $f \in \mathscr{F}_{p,q,r}$ and $0 \leq j \leq r$, let $x_1^j, \dots, x_s^j$ be the vertices of $f$ lying at reverse height $j$ in $f$, from left to right. For every $1 \leq i \leq s$, let $t_i^j$ be the tree of descendants of $x_i^j$. Let also $i_0$ be the index such that the distinguished vertex belongs to $t^j_{i_0}$. We define $B_j(f)$ as the forest $(t_1^j, \dots, t_s^j)$, with the same distinguished vertex as $f$ (see Figure \ref{ballforest}), and $B'_j(f)$ as the forest $(t_{i_0}^j, t_{i_0+1}^j, \dots, t_s^j, t_1^j, \dots, t_{i_0-1}^j)$. Note that $B_j(f)  \in \F_{p,s,j}$ and $B'_j(f)  \in \F'_{p,s,j}$, but it is not always the case that $B_j(f) \in \F'_{p,s,j}$. Note also that if $v$ is a vertex of $B_j(f)$, then $h^{\mathrm{rev}}_{B_j(f)}(x)=h^{\mathrm{rev}}_{B'_j(f)}(x)=h^{\mathrm{rev}}_{f}(x)$.

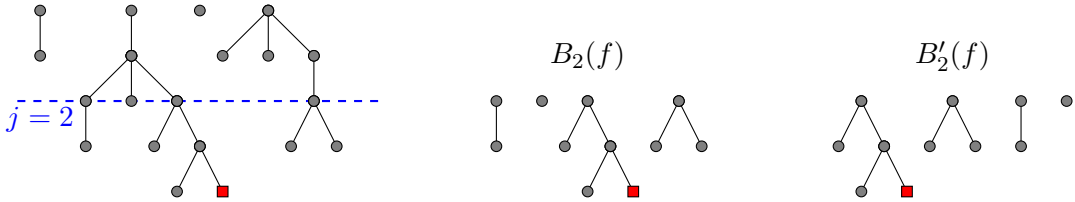
\begin{figure}
\begin{center}
\begin{tikzpicture}[scale=0.6]
\draw[thick, blue, dashed] (-2.5,2)--(5.5,2);

\draw (-2,4) node{}--(-2,3) node{};

\draw (0,4) node{}--(0,3) node{};
\draw (0,3) node{}--(0,2) node{};
\draw (0,3) node{}--(1,2) node{};
\draw (0,3) node{}--(-1,2) node{};
\draw (1,2) node{}--(0.5,1) node{};
\draw (1,2) node{}--(1.5,1) node{};
\draw (-1,2) node{}--(-1,1) node{};
\draw (1.5,1) node{}--(1,0) node{};
\draw (1.5,1) node{}--(2,0);

\draw(2,0) node[rootedge]{};
\draw (1.5,4) node{};

\draw (3,4) node{}--(2,3) node{};
\draw (3,4) node{}--(3,3) node{};
\draw (3,4) node{}--(4,3) node{};
\draw (4,3) node{}--(4,2) node{};
\draw (4,2) node{}--(3.5,1) node{};
\draw (4,2) node{}--(4.5,1) node{};

\draw[blue] (-2,1.6) node[texte]{$j=2$};

\begin{scope}[shift={(8,0)}]
\draw (2,2) node{}--(1.5,1) node{};
\draw (2,2) node{}--(2.5,1) node{};
\draw (2.5,1) node{}--(2,0) node{};
\draw (2.5,1) node{}--(3,0);
\draw (1,2) node{};
\draw (0,2) node{}--(0,1) node{};
\draw (4,2) node{}--(3.5,1) node{};
\draw (4,2) node{}--(4.5,1) node{};
\draw (3,0) node[rootedge]{};
\draw (2,3) node[texte]{$B_2(f)$};
\end{scope}

\begin{scope}[shift={(16,0)}]
\draw (0,2) node{}--(-0.5,1) node{};
\draw (0,2) node{}--(0.5,1) node{};
\draw (0.5,1) node{}--(0,0) node{};
\draw (0.5,1) node{}--(1,0);
\draw (4.5,2) node{};
\draw (3.5,2) node{}--(3.5,1) node{};
\draw (2,2) node{}--(1.5,1) node{};
\draw (2,2) node{}--(2.5,1) node{};
\draw (1,0) node[rootedge]{};
\draw (2,3) node[texte]{$B'_2(f)$};
\end{scope}
\end{tikzpicture}
\end{center}
\caption{A forest $f$ of height $4$ and the forests $B_2(f)$ and $B'_2(f)$. The ancestors lie on the top and the distinguished vertex is in red.}\label{ballforest}
\end{figure}

\begin{defn}\label{defn_reverse_forest}
Let $p \geq 1$. A \emph{$p$-pre-admissible infinite forest} is a sequence $\mathbf{f}=(f_r)_{r \geq 1}$ of forests of plane trees such that
\begin{itemize}
\item[(i)]
for every $r \geq 1$, there is $q \geq 1$ such that the forest $f_r$ is $(p,q,r)$-pre-admissible,
\item[(ii)]
for every $s \geq r \geq 1$, we have $B_r(f_{s})=f_r$.
\end{itemize}
We write $\F_{p, \infty, \infty}$ for the set of $p$-pre-admissible infinite forests. We will also write $B_r(\mathbf{f})=f_r$ for $r \geq 1$.
\end{defn}

Note that $\mathbf{f} \in \F_{p,\infty,\infty}$ can also be seen as an increasing sequence of finite graphs, and therefore as an infinite graph. We call \emph{infinite reverse trees} the connected components of $\mathbf{f}$. If $v$ is a vertex of $\mathbf{f}$, we have $v \in f_r$ for $r$ large enough. Note that $h^{\mathrm{rev}}_{f_r}(v)$ does not depend on $r$ as long as $v \in f_r$. We call it the \emph{reverse height of $v$ in $\mathbf{f}$} and denote it by $h^{\mathrm{rev}}_{\mathbf{f}}(v)$. We also write $\mathbf{f}^*$ for the set of vertices of $\mathbf{f}$ that do not lie at reverse height $0$.

\begin{defn}
A $p$-pre-admissible forest  $\mathbf{f}$ is called \emph{$p$-admissible} if the distinguished vertex lies in the leftmost infinite tree of $\mathbf{f}$, i.e. if for every $r \geq 0$, the leftmost vertex of $\mathbf{f}$ at reverse height $r$ is in the same infinite tree as the distinguished vertex.
We write $\F'_{p, \infty, \infty}$ for the set of $p$-admissible infinite forests.
\end{defn}


\paragraph{Skeleton decomposition of infinite triangulations of the $p$-gon.}
We now introduce the skeleton decomposition of infinite triangulations of the $p$-gon. Let $T$ be an infinite, one-ended triangulation of the $p$-gon. For every $r \geq 1$, the hull $B_r^{\bullet}(T)$ is a triangulation of the cylinder of height $r$, so we can define its skeleton $f'_r=\mathrm{Skel} \left(  B_r^{\bullet}(T) \right)\in \F'_{p,q,r}$ for some $q$. It is also easy to see that the forests $f'_r$ are consistent in the sense that $B'_r(f'_{s})=f'_r$ for every $s \geq r \geq 1$.
We claim that such a family $(f'_r)$ always defines an infinite $p$-admissible forest. More precisely, if $\mathbf{f} \in \F'_{p,\infty,\infty}$ and $r \geq 1$, let $B'_r(\mathbf{f})$ be the reordered ball of radius $r$ in $\mathbf{f}$ (that is, the ball $B_r(\mathbf{f})$ in which the trees have been cyclically permutated so that the distinguished vertex lies in the first tree). Then there is a unique $\mathbf{f} \in \F'_{p,\infty,\infty}$ such that $B'_r(\mathbf{f})=f'_r$ for every $r \geq 1$. We do not prove this formally, but explain how to build $\mathbf{f}$ from $(f'_r)$: for any $s \geq r$, the forest $B_r(f'_{s})$ is a cyclic permutation of $f'_r$, and this cyclic permutation does not depend on $s$ for $s$ large enough. Hence, we can set $f_r=B_r(f'_{s})$ for $s$ large enough and $\mathbf{f}=(f_r)_{r \geq 1}$.
Therefore, there is a unique $p$-admissible forest, that we denote by $\mathrm{Skel}(T)$ and call the \emph{skeleton} of $T$, such that $B'_r \left( \mathrm{Skel}(T) \right)= \mathrm{Skel} \left(  B_r^{\bullet}(T) \right)$ for every $r \geq 1$.
As in the finite case, the skeleton decomposition establishes a bijection between one-ended infinite triangulations of the $p$-gon and pairs consisting of a $p$-admissible infinite forest $\mathbf{f}$ and a family $(M_v)_{v \in \mathbf{f}^*}$ of maps such that $M_v$ is a finite triangulation of a $(c_v+2)$-gon for every $v$.

\begin{rem}
In order to define the skeleton decomposition, it would have been more convenient to define an infinite reverse forest $\mathbf{f}$ as the sequence $\left( B'_r(\mathbf{f}) \right)$ instead of $\left( B_r(\mathbf{f}) \right)$. The reason why we chose this definition is that it will later make the decomposition of an infinite forest in infinite reverse trees much more convenient to define.
\end{rem}

\subsection{Computation of the skeleton decomposition of the hulls of \texorpdfstring{$\T_{ \lambda}$}{TEXT}}

We now compute the law of the skeletons of the hulls of $\T_{\lambda}^p$. The map $\T_{\lambda}^p$ can be seen as a triangulation of the cylinder with infinite height. For every $r \geq 1$, the map $B_r^{\bullet}(\T^p_{\lambda})$ is a triangulation of the cylinder of height $r$ with bottom boundary length equal to $p$.

\begin{lem} \label{skeleton}
Let $0<\lambda \leq \lambda_c$. Let $\Delta$ be a triangulation of the cylinder of height $r$. We write $p$ (resp. $q$) for the length of its bottom (resp. top) boundary. Let $f=\mathrm{Skel}(\Delta) \in \F'_{p,q,r}$. Then
\[ \mathbb{P} \left( B_r^{\bullet}(\T^p_{\lambda})=\Delta \right)=\frac{q \, h_{\lambda}(q)}{p \, h_{\lambda}(p)} \prod_{v \in f^*} \theta_{\lambda}(c_v) \prod_{v \in f^*} \frac{\lambda^{|M_v|}}{w_{\lambda}(c_v+2)},\]
where:
\begin{itemize}
\item
$c_v$ is the number of children of $v$ in $f$,
\item
$h \in \left( 0,\frac{1}{4} \right]$ is given by \eqref{eqn_h_lambda},
\item
$h_{\lambda}(p)=\frac{1}{p} \left( 8+\frac{1}{h}\right)^{-p} c_{\lambda}(p)$,
\item
$\theta_{\lambda}$ is the offspring distribution whose generating function $g_{\lambda}$ is given by
\begin{equation}\label{generating}
 g_{\lambda}(x)=\sum_{i \geq 0} \theta_{\lambda}(i) x^i=\frac{1}{x}-\frac{(1-x)(1-\sqrt{1-4hx})}{2h x^2}.
\end{equation}
\end{itemize}
\end{lem}       

\begin{proof}
The proof is exactly the same as in the case $\lambda=\lambda_c$ (Lemma 2 in \cite{CLGmodif}), up to changes of notation (the $\rho$ of \cite{CLGmodif} corresponds to our $\lambda$ and the $\alpha$ corresponds to $8+\frac{1}{h}$). The same computations yield
\begin{equation}\label{theta_lambda}
\theta_{\lambda}(i)= \frac{1}{\sqrt{1+8h}} \left( \frac{h}{1+8h} \right)^i w_{\lambda}(i+2),
\end{equation}
and the computation of $g_{\lambda}$ follows from \eqref{G}.
\end{proof}

Let $p,q,r \geq 1$ and $f \in \F'_{p,q,r}$. We sum the formula of Lemma \ref{skeleton} over all families $(M_v)_{v \in f^*}$ such that $M_v$ is a triangulation of a $(c_v+2)$-gon for every $v$. By the definition of $w_{\lambda}(i+2)$, we have $\sum_{n \geq 0} \# \mathscr{T}_{i+2,n} \frac{\lambda^n}{w_{\lambda}(i+2)}=1$ for every $i \geq 0$, so we get
\begin{equation}\label{LoiForet}
\mathbb{P} \left( \mathrm{Skel} \left(B_r^{\bullet}(\T^p_{\lambda}) \right) =f \right)=\frac{q h_{\lambda}(q)}{p h_{\lambda}(p)} \prod_{v \in f^*} \theta_{\lambda}(c_v).
\end{equation}

Note that \eqref{LoiForet} describes explicitly the distribution of $B'_r \left( \mathrm{Skel}(\T^p_{\lambda}) \right)$, so we completely know the law of $\mathrm{Skel}(\T^p_{\lambda})$. As we will see in Section 2.3, this is in theory enough to prove Theorem \ref{thm1_GW}. However, infinite leftmost geodesics are not very tractable in this characterization. Hence, we will need to find another construction of the $\mathrm{Skel} \left( \T^p_{\lambda} \right)$ and prove it is equivalent to \eqref{LoiForet}. This will be the main goal of the rest of Section 2. Before moving on to the proof of Theorem \ref{thm1_GW}, we end this subsection with a few remarks about the perimeter process of $\T_{\lambda}$.

We notice that \eqref{LoiForet} can be used to study the perimeter process of $\T_{\lambda}$ in the same way as the perimeter process of $\T_{\lambda_c}$ is studied in \cite{Kri04} and \cite{M16}. More precisely, by the same computation as in the proof of Lemma 3 in \cite{CLGmodif}, by summing \eqref{LoiForet} over all $(p,q,r)$-admissible forests, we obtain
\[\mathbb{P} \left( |\partial B_r^{\bullet} \left( \T^p_{\lambda} \right)|=q \right) = \frac{h_{\lambda}(q)}{h_{\lambda}(p)} \mathbb{P}_q \left( X_{\lambda}(r)=p \right),\]
where $X_{\lambda}$ is the Galton--Watson process with offspring distribution $\theta_{\lambda}$.
Since we know that $\left( |\partial B_r^{\bullet} (\T_{\lambda})|\right)_{r \geq 0}$ is a Markov chain (by the spatial Markov property) and has the same transitions as the perimeter process of $\T^1_{\lambda}$, we even get, for every $p,q \geq 1$ and $s \geq r \geq 0$,
\begin{equation}\label{branching2}
\mathbb{P} \left( |\partial B_s^{\bullet} \left( \T_{\lambda} \right)|=q \big| |\partial B_r^{\bullet} \left( \T_{\lambda} \right)|=p \right)=\frac{h_{\lambda}(q)}{h_{\lambda}(p)} \mathbb{P}_q \left( X_{\lambda}(s-r)=p \right).
\end{equation}

Let $m_{\lambda}=\sum_{i \geq 0} i \theta_{\lambda}(i)$ be the mean number of children. By \eqref{generating}, we can compute $m_{\lambda}=g'_{\lambda}(1)$ and obtain \eqref{eqn_m_lambda}.
In particular, we have $m_{\lambda} \leq 1$, with equality if and only if $\lambda=\lambda_c$. Hence, the Galton--Watson process $X_{\lambda}$ is subcritical for $\lambda<\lambda_c$. We can therefore see the perimeter process of $\T_{\lambda}$ for $\lambda<\lambda_c$ as a time-reversed subcritical branching process. Note that the perimeters and volumes of the hulls in $\T_{\lambda}$ are already quite well-known. Sharp exponential growth for a fixed $\lambda<\lambda_c$ is proved in Section 2 of \cite{CurPSHIT}, whereas the near-critical scaling limit as $\lambda \to \lambda_c$ is studied in Section 3 of \cite{B16}. Equation \eqref{branching2} together with Lemma \ref{itererg} can give explicit formulas for the generating function of the perimeters of the hull, so it should be possible to recover these results by using the same techniques as in \cite{M16}, but we do not do this in this work.

\subsection{Slicing the skeleton}

The goal of this subsection is twofold. First, we explain the link between the skeleton decomposition and the infinite leftmost geodesics of an infinite triangulation. Second, we introduce some formalism, that will later allow us to obtain a construction of $\mathrm{Skel}(\T_{\lambda})$ that is more suitable for our purpose than \eqref{LoiForet}.

\paragraph{Decomposition of a $1$-admissible forest in reverse trees.}
An infinite $1$-admissible forest $\mathbf{f}$ may contain several infinite reversed trees, and we will need to study the way these trees are placed with respect to each other. This can be encoded by a genealogical structure, which is described by the red tree on Figure \ref{tree_mathfrak_T}. If $\mathbf{t}$ is one of the infinite trees of $\mathbf{f}$, let $h^{\mathrm{rev}}_{\mathrm{min}}(\mathbf{t})$ be the reverse height of the lowest vertex of $\mathbf{t}$. We consider the set of pairs $(\mathbf{t},i)$ where $\mathbf{t}$ is an infinite tree of $\mathbf{f}$ and $i \geq h^{\mathrm{rev}}_{\mathrm{min}}(\mathbf{t})$. If $i>h^{\mathrm{rev}}_{\mathrm{min}}(\mathbf{t})$, the parent of $(\mathbf{t},i)$ is $(\mathbf{t},i-1)$. If $i=h^{\mathrm{rev}}_{\mathrm{min}}(\mathbf{t})>0$, let $\mathbf{t}'$ be the first infinite tree on the left of $\mathbf{t}$ such that $h^{\mathrm{rev}}_{\mathrm{min}}(\mathbf{t}') \leq i-1$ (note that $\mathbf{t}'$ always exists because $\mathbf{f}$ is admissible). Then the parent of $(\mathbf{t},i)$ is $(\mathbf{t}',i-1)$. Finally, if $i=0$, then $(\mathbf{t},i)$ has no parent. This genealogy is encoded in an infinite plane tree with no leaf that we denote by $\mathbf{U}(\mathbf{f})$ (see Figure \ref{tree_mathfrak_T}). More intuitively, the tree $\mathbf{U}(\mathbf{f})$ is the tree whose branches pass between the infinite trees of $\mathbf{f}$. Note that the genealogy in $\mathbf{U}(\mathbf{f})$ is "reversed" compared to the genealogy in $\mathbf{f}$: the parent of a vertex $x$ of $\mathbf{U}(\mathbf{f})$ lies below $x$, whereas in the forest $\mathbf{f}$, the parent of a vertex lies above it. Therefore, the heights in $\mathbf{U}(\mathbf{f})$ match the reverse heights in $\mathbf{f}$. We also write $B_r(\mathbf{U}(\mathbf{f}))$ for the subtree of $\mathbf{U}(\mathbf{f})$ whose vertices are the $(\mathbf{t},i)$ with $i \leq r$. Note that the tree $B_r(\mathbf{U}(\mathbf{f}))$ is not a function of $B_r(\mathbf{f})$ (it is impossible by looking at $B_r(\mathbf{f})$ to know if two vertices belong to the same infinite tree). Finally, it is easy to see that a $1$-admissible forest is completely described by the tree $\mathbf{U}(\mathbf{f})$ and the infinite trees it contains.

\begin{figure}
\begin{center}
\begin{tikzpicture}
\draw(0,4)--(0,3);
\draw(1,4)--(0.5,3);
\draw(1,4)--(1,3);
\draw(1,3)--(1,2);
\draw(1,3)--(1.5,2);
\draw(1,2)--(1,1);
\draw(1.5,2)--(1.5,1);
\draw(1,2)--(0.5,1);
\draw(1,1)--(1,0);
\draw(3,4)--(2.5,3);
\draw(3,4)--(3.5,3);
\draw(3.5,3)--(3,2);
\draw(3.5,3)--(3.5,2);
\draw(3.5,2)--(3.5,1);
\draw(5,4)--(4.5,3);
\draw(5,4)--(5.5,3);
\draw(4.5,3)--(4.5,2);
\draw[dashed](0,4)to[bend left=15](0.5,5.5);
\draw[dashed](1,4)to[bend right=15](0.5,5.5);
\draw[dashed](0.5,5.5)--(0.5,6);
\draw[dashed](3,4)to[bend left=15](3.5,5.5);
\draw[dashed](4,4)to[bend right=15](3.5,5.5);
\draw[dashed](3.5,5.5)--(3.5,6);
\draw[dashed](2,4)--(2,6);
\draw[dashed](5,4)--(5,6);

\draw(0,4)node{};
\draw(0,3)node{};
\draw(1,4)node{};
\draw(0.5,3)node{};
\draw(1,3)node{};
\draw(1,2)node{};
\draw(1.5,2)node{};
\draw(1,1)node{};
\draw(0.5,1)node{};
\draw(1.5,1)node{};
\draw(1,0)node{};
\draw(2,4)node{};
\draw(3,4)node{};
\draw(2.5,3)node{};
\draw(3.5,3)node{};
\draw(3,2)node{};
\draw(3.5,2)node{};
\draw(3.5,1)node{};
\draw(4,4)node{};
\draw(5,4)node{};
\draw(4.5,3)node{};
\draw(5.5,3)node{};
\draw(4.5,2)node{};
\draw(0.5,5.5)node{};
\draw(3.5,5.5)node{};

\draw[red, thick](2,0)--(2,1);
\draw[red, thick](2,1)--(2,2);
\draw[red, thick](2,2)--(2,3);
\draw[red, thick](2,3)--(1.5,4);
\draw[red, thick](2,3)--(2.5,4);
\draw[red, thick](2,0) to[bend right=15] (4,1);
\draw[red, thick](4,1)--(4,2);
\draw[red, thick](4,2)--(4,3);
\draw[red, thick](4,3)--(4.5,4);
\draw[red, thick](4,1)--(5.5,2);
\draw[red, thick](5.5,2)--(6,3);
\draw[red, thick](6,3)--(6,4);
\draw[red, thick, dashed](1.5,4)--(1.5,6);
\draw[red, thick, dashed](2.5,4)--(2.5,6);
\draw[red, thick, dashed](4.5,4)--(4.5,6);
\draw[red, thick, dashed](6,4)--(6,6);

\draw(2,0)node[dual]{};
\draw(2,1)node[dual]{};
\draw(2,2)node[dual]{};
\draw(2,3)node[dual]{};
\draw(1.5,4)node[dual]{};
\draw(2.5,4)node[dual]{};
\draw(4,1)node[dual]{};
\draw(4,2)node[dual]{};
\draw(4,3)node[dual]{};
\draw(4.5,4)node[dual]{};
\draw(5.5,2)node[dual]{};
\draw(6,3)node[dual]{};
\draw(6,4)node[dual]{};

\draw(0.2,5.7)node[texte]{$t_1$};
\draw(1.8,5.7)node[texte]{$t_2$};
\draw(3.2,5.7)node[texte]{$t_3$};
\draw(4.8,5.7)node[texte]{$t_4$};

\draw[red](2,-0.3)node[texte]{$(t_1,0)$};
\draw[red](2.7,1)node[texte]{$(t_1,1)$};
\draw[red](4.5,0.7)node[texte]{$(t_3,1)$};
\draw[red](6,1.7)node[texte]{$(t_4,2)$};
\draw[red](6.6,3)node[texte]{$(t_4,3)$};
\draw[red](6.6,4)node[texte]{$(t_4,4)$};
\end{tikzpicture}
\end{center}
\caption{An infinite $1$-admissible forest $\mathbf{f}$ in which $4$ trees $t_1$, $t_2$, $t_3$ and $t_4$ reach reverse height $4$. In red, the tree $\mathbf{U}(\mathbf{f})$ and the names of some of its vertices. In the proof of Proposition \ref{otherconstructionForest} for $r=4$, the $p_j^r$ are equal to $2$, $1$, $2$ and $1$ and the $h(f_j)$ are equal to $4$, $0$, $3$ and $2$.}\label{tree_mathfrak_T}
\end{figure}
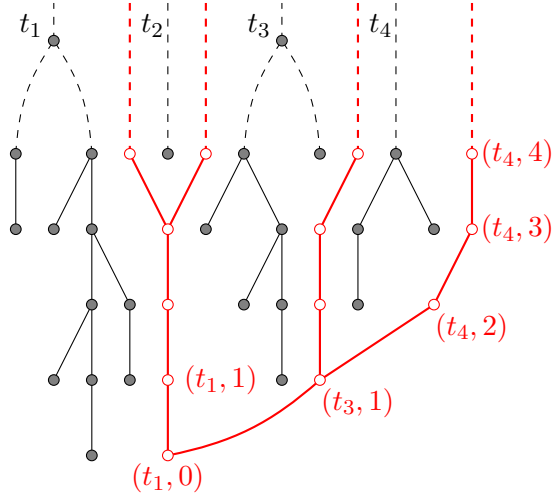

\paragraph{Leftmost infinite geodesics and decomposition of the skeleton in reverse trees.}
Let $T$ be an infinite triangulation of a $1$-gon, let $\rho$ be its root vertex (i.e. the unique point on its boundary), and let $\mathbf{f}=\mathrm{Skel}(T)$. We have seen that for every $r \geq 0$, the paths going between the trees of $B_r(\mathbf{f})$ correspond to the leftmost geodesics from $\rho$ to the vertices of $\partial B_r^{\bullet}(T)$ (cf. Figure \ref{Skeleton_decomposition}). Therefore, infinite paths started from $\rho$ in $\mathbf{U}(\mathbf{f})$ correspond to leftmost geodesic rays in $T$, so the tree of leftmost infinite geodesics in $T$ is isomorphic to $\mathbf{U}(\mathbf{f})$.

\paragraph{The skeleton decomposition of infinite strips.}
We will also need to describe the skeleton decomposition of \emph{strips}, which are infinite triangulations with two infinite geodesic boundaries. They correspond to the $S^0$ and $S^1$ appearing in Theorem \ref{thm1_GW}.

\begin{defn}\label{defn_strip}
An \emph{infinite strip} is a one-ended planar triangulation bounded by two infinite geodesics $\gamma_{\ell}$ (on its left) and $\gamma_r$ (on its right), and equipped with a root vertex $\rho$, such that:
\begin{itemize}
\item[(i)]
$\rho$ is the only common point of $\gamma_{\ell}$ and $\gamma_r$,
\item[(ii)]
for every $i,j \geq 0$, the path $\gamma_r$ is the only geodesic from $\gamma_r(i)$ to $\gamma_r(j)$,
\item[(iii)]
$\gamma_r$ and $\gamma_{\ell}$ are the only leftmost geodesic rays in $S$.
\end{itemize}
\end{defn}

Exactly as for infinite triangulations of the $p$-gon, if $S$ is an infinite strip, we define its ball $B_r(S)$ of radius $r$ as the union of all its faces containing a vertex at distance at most $r-1$ from $\rho$. We also define its hull $B_r^{\bullet}(S)$ of radius $r$ as the union of $B_r(S)$ and all the finite connected components of its complement.
To define the skeleton of an infinite strip $S$, we note that there is a simple transformation that associates to $S$ an infinite triangulation of the $p$-gon, where $p$ is the number of edges on $\partial B_1^{\bullet}(S)$. More precisely, we write $\widetilde{S}$ for the map obtained by rooting $S \backslash B_1^{\bullet}(S)$ at the leftmost edge of $\partial B_1^{\bullet}(S)$ and gluing $\gamma_{\ell}$ and $\gamma_r$ together. We define the skeleton of $S$ as the skeleton of $\widetilde{S}$ (see Figure \ref{Skeleton_decomposition_strip}).
Like for triangulations of the plane, the skeleton decomposition is a bijection between infinite strips and pairs consisting:
\begin{itemize}
\item
on the one hand of an infinite reverse tree $\mathbf{t}$ that is rooted at its leftmost vertex of reverse height $0$,
\item
on the other hand of a family of maps $(M_v)_{v \in \mathbf{t}}$ such that $M_v$ is a triangulation of a $(c_v+2)$-gon for every $v$.
\end{itemize}
The fact that $\mathbf{t}$ must be connected follows from the uniqueness of the leftmost geodesic rays $\gamma_{\ell}$ and $\gamma_r$ in a strip (recall that the trees of the skeleton are separated by infinite leftmost geodesics).
Note that this time, the triangulations filling the holes are indexed by $\mathbf{t}$ and not $\mathbf{t}^*$ since the $M_v$ for $v$ at reverse height $0$ are used to encode $B_1^{\bullet}(S)$ (see Figure \ref{Skeleton_decomposition_strip}).

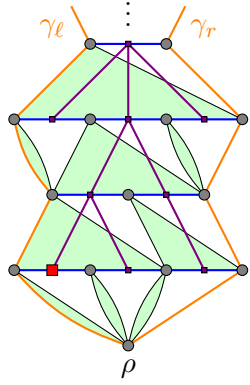
\begin{figure}
\begin{center}
\begin{tikzpicture}
\fill[green!20] (0,0) to[bend left=15] (-1.5,1)--(0,0);
\fill[green!20] (0,0) to[bend left=15] (-0.5,1) to[bend left=15] (0,0);
\fill[green!20] (0,0) to[bend left=15] (0.5,1) to[bend left=15] (0,0);
\fill[green!20] (-1.5,1)--(0.5,1)--(-1,2);
\fill[green!20] (0.5,1)--(1.5,1)--(0,2);
\fill[green!20] (-1,2) to[bend left=15] (-1.5,3) to[bend left=15] (-1,2);
\fill[green!20] (-1,2)--(1,2)--(-0.5,3);
\fill[green!20] (1,2) to[bend left=15] (0.5,3) to[bend left=15] (1,2);
\fill[green!20] (-1.5,3)--(1.5,3)--(-0.5,4);

\draw[thick, blue](-1.5,1)--(1.5,1);
\draw[thick, blue](-1,2)--(1,2);
\draw[thick, blue](-1.5,3)--(1.5,3);
\draw[thick, blue](-0.5,4)--(0.5,4);

\draw[orange, thick](0,0) to[bend left=15] (-1.5,1) -- (-1,2) to[bend left=15] (-1.5,3) -- (-0.5,4)--(-0.75,4.5);
\draw[orange, thick](0,0) -- (1.5,1)--(1,2)--(1.5,3)--(0.5,4)--(0.75,4.5);
\draw(0,0)--(-1.5,1);
\draw(0,0) to[bend left=15] (-0.5,1);
\draw(0,0) to[bend right=15] (-0.5,1);
\draw(0,0) to[bend left=15](0.5,1);
\draw(0,0) to[bend right=15](0.5,1);
\draw(0.5,1)--(-1,2);
\draw(0.5,1)--(0,2);
\draw(1.5,1)--(0,2);
\draw(-1,2) to[bend right=15] (-1.5,3);
\draw(-1,2)--(-0.5,3);
\draw(1,2)--(-0.5,3);
\draw(1,2) to[bend left=15] (0.5,3);
\draw(1,2) to[bend right=15] (0.5,3);
\draw(1.5,3)--(-0.5,4);

\draw(0,0) node{};
\draw(-1.5,1) node{};
\draw(-0.5,1) node{};
\draw(0.5,1) node{};
\draw(1.5,1) node{};
\draw(-1,2) node{};
\draw(0,2) node{};
\draw(1,2) node{};
\draw(-1.5,3) node{};
\draw(-0.5,3) node{};
\draw(0.5,3) node{};
\draw(1.5,3) node{};
\draw(-0.5,4) node{};
\draw(0.5,4) node{};

\draw[violet,thick] (-1,1)--(-0.5,2);
\draw[violet,thick] (0,1)--(-0.5,2);
\draw[violet,thick] (1,1)--(0.5,2);
\draw[violet,thick] (-0.5,2)--(0,3);
\draw[violet,thick] (0.5,2)--(0,3);
\draw[violet,thick] (-1,3)--(0,4);
\draw[violet,thick] (0,3)--(0,4);
\draw[violet,thick] (1,3)--(0,4);

\draw(-1,1) node[rootedge]{};
\draw(0,1) node[edge']{};
\draw(1,1) node[edge']{};
\draw(-0.5,2) node[edge']{};
\draw(0.5,2) node[edge']{};
\draw(-1,3) node[edge']{};
\draw(0,3) node[edge']{};
\draw(1,3) node[edge']{};
\draw(0,4) node[edge']{};

\draw(0,-0.3) node[texte]{$\rho$};
\draw[orange](1,4.2) node[texte]{$\gamma_r$};
\draw[orange](-1,4.2) node[texte]{$\gamma_{\ell}$};
\draw(0,4.5) node[texte]{$\vdots$};

\end{tikzpicture}
\end{center}
\caption{An infinite strip $S$ and its skeleton.} \label{Skeleton_decomposition_strip}
\end{figure}

Finally, let $T$ be an infinite triangulation of a $1$-gon, and let $\mathbf{f}$ be its skeleton. As we have seen above, the tree $\mathbf{U}(\mathbf{f})$ can be seen as the tree of leftmost infinite geodesics of $T$. Moreover, $\mathbf{U}(\mathbf{f})$ cuts $T$ into strips, whose skeletons are the infinite reverse trees of $\mathbf{f}$. This reduces the proof of Theorem \ref{thm1_GW} to the study of $\mathrm{Skel}(\mathbb{T}_{\lambda}^1)$.

\subsection{The distribution of \texorpdfstring{$S_{\lambda}^0$}{TEXT} and \texorpdfstring{$S_{\lambda}^1$}{TEXT}}

The goal of this subsection is to describe (without to build them explicitly) the strips $S_{\lambda}^0$ and $S_{\lambda}^1$ in a way that will allow us to prove Theorem \ref{thm1_GW}. This description will involve the quasi-stationary distribution of the branching process $X_{\lambda}$. Hence, we start with two explicit computations about $X_{\lambda}$. First, we give an explicit formula for the iterates of the generating function $g_{\lambda}$. We define $g_{\lambda}^{\circ r}$ by $g_{\lambda}^{\circ 0}=\mathrm{Id}$ and $g_{\lambda}^{\circ (r+1)}=g_{\lambda} \circ g_{\lambda}^{\circ r}$ for every $r \geq 0$. Note that $g_{\lambda_c}^{\circ r}$ is explicitly computed in \cite{CLGmodif}: we have
\begin{equation}\label{itererg_critique}
g_{\lambda_c}^{\circ r}(x)=1- \left(r+\frac{1}{\sqrt{1-x}} \right)^{-2}
\end{equation}
for every $x \in [0,1]$. The iterates of $g_{\lambda}$ in the subcritical case are also computed in \cite{M16}, with different notations. Note that when $\lambda \to \lambda_c$ in the formula below, we recover \eqref{itererg_critique}.

\begin{lem} \label{itererg}
Let $0  < \lambda< \lambda_c$. For every $r \geq 0$, we have
\[g_{\lambda}^{\circ r}(x)= 1-\frac{1-4h}{4h \sinh^2 \left( \argsh \sqrt{\frac{1-4h}{4h(1-x)}} +r b_{\lambda} \right)},\]
where $b_{\lambda}= \argch \frac{1}{\sqrt{4h}} =-\frac{1}{2} \ln m_{\lambda}$, and $h$ is as in \eqref{eqn_h_lambda}.
\end{lem}

\begin{proof}
This computation is already done in Section 3.1 of \cite{M16}, with different notation. In \cite{M16}, the function $\varphi_t(u)$ for $0 \leq t \leq 1$ is defined by
\[\varphi_t(u)= \frac{\rho s}{\alpha t} \sum_{i \geq 0} (\alpha t)^i w_{\rho s}(i+2) u^i,\]
where $\alpha=\frac{1}{12}$, $\rho=\frac{1}{12 \sqrt{3}}=\lambda_c$ and $s=t\sqrt{3-2t}$. For $t=\frac{12h}{1+8h}$, we have $s=\frac{\lambda}{\lambda_c}$, so by \eqref{theta_lambda}, we have $\frac{\rho s}{\alpha t} (\alpha t)^i w_{\rho s}(i+2)=\theta_{\lambda}(i)$ for every $i \geq 0$, and our function $g_{\lambda}$ corresponds to the $\varphi_t$ of \cite{M16} for $t=\frac{12h}{1+8h}$. Our formula follows then immediately from Lemma 3 of \cite{M16}. Finally, using \eqref{eqn_m_lambda}, it is easy to check that $e^{-2b_{\lambda}}=m_{\lambda}$.
\end{proof}

\begin{rem}
It is also possible to prove the last lemma directly by using \eqref{branching2}. Indeed, the probability must add up to $1$ when we sum them over $q$, which shows that $\left( h_{\lambda}(p) \right)$ is an invariant measure for $X_{\lambda}$. If we write $H_{\lambda}(x)=\sum_{p  \geq 1} h_{\lambda}(p) x^p$, this easily gives 
\[H_{\lambda}(g_{\lambda}(x))=H_{\lambda}(x)+H_{\lambda} (g_{\lambda}(0)),\]
so $g_{\lambda}^{\circ r}(x)=H_{\lambda}^{-1} \left( H_{\lambda}(x)+rH_{\lambda}(g_{\lambda}(0))\right)$. Since $H_{\lambda}$ can be explicitly computed, this also gives the result.
\end{rem}

Our second computation deals with the quasi-stationary measure of $X_{\lambda}$. We first recall a few facts  about quasi-stationary measures of Galton--Watson processes (see for example Chapter 7 of \cite{AN72}). For every $p \geq 1$, the ratio $\frac{\P_1 \left( X_{\lambda}(n)=p \right)}{\P_1 \left( X_{\lambda}(n)=1 \right)}$, where $\P_1$ is the distribution of $X_{\lambda}$ started from $1$, is nondecreasing in $n$ and converges to $\pi_{\lambda}(p)<+\infty$. Moreover, let $\Pi_{\lambda}$ be the generating function of $\left( \pi_{\lambda}(p) \right)_{p \geq 1}$. In our case, it is explicit:
\begin{lem} \label{Piexact}
For $0 < \lambda<\lambda_c$, we have
\[ \Pi_{\lambda}(x)=\frac{1}{\sqrt{1-4h}} \bigg(1- (1-x) \left( \frac{\sqrt{1-4h}+1}{\sqrt{1-4h}+\sqrt{1-4hx}} \right)^2 \bigg). \]
We also have
\[\Pi_{\lambda_c}(x)=2 \left( \frac{1}{\sqrt{1-x}}-1 \right).\]
In particular, we have $\Pi_{\lambda}(\theta_{\lambda}(0))=\Pi_{\lambda}(1-h)=\frac{1-\sqrt{1-4h}}{2h} $.
\end{lem}

\begin{proof}
By the definition of $\pi_{\lambda}(p)$, we have
\[ \Pi_{\lambda} (x)=\lim_{n \to +\infty} \frac{g_{\lambda}^{\circ n}(x)-g_{\lambda}^{\circ n}(0)}{(g_{\lambda}^{\circ n})'(0)}.\]
Hence, by using Lemma \ref{itererg}, the computation of $\Pi_{\lambda}(x)$ is straightforward. Note that the case $\lambda=\lambda_c$ already appears in \cite{CLGmodif} (in the proof of Lemma 3).
\end{proof}

We can now describe the distribution of $S_{\lambda}^0$ and $S_{\lambda}^1$. We will first admit the existence of reverse trees with a certain distribution (described by the next lemma), and later (Section 2.6) build these trees by a spine decomposition approach.

\begin{lem}\label{distrib_ball_tau}
There are two infinite reverse trees $\boldsymbol{\tau}^0_{\lambda}$ and $\boldsymbol{\tau}^1_{\lambda}$ whose distributions are characterized as follows. For every $r \geq 0$ and every forest $(t_1, \dots, t_p)$ of height exactly $r$, we have
\begin{equation}\label{tauforest0}
\mathbb{P} \left( B_r ( \boldsymbol{\tau}^0_{\lambda} )=(t_1, \dots, t_p) \right)=\frac{\pi_{\lambda}(p) m_{\lambda}^{-r}}{\Pi_{\lambda}(\theta_{\lambda}(0))} \prod_{i=1}^p \prod_{v \in t_i} \theta_{\lambda}(c_v)
\end{equation}
and, if $f$ has only one vertex at height $r$,
\begin{equation}\label{tauforest1}
\P \left( B_r ( \boldsymbol{\tau}_{\lambda}^1 ) =(t_1, \dots, t_p) \right)=\frac{\pi_{\lambda}(p) m_{\lambda}^{-r}}{\theta_{\lambda}(0)} \prod_{i=1}^p \prod_{v \in t_i} \theta_{\lambda} (c_v),
\end{equation}
where $c_v$ is the number of children of $v$ for every vertex $v$.
\end{lem}

Moreover, let $\boldsymbol{\tau}^{1,*}_{\lambda}$ be the tree $\boldsymbol{\tau}^1_{\lambda}$ in which we have cut the only vertex at reverse height $0$. The reverse heights in $\boldsymbol{\tau}^{1,*}$ are shifted by $1$, so that the minimal reverse height in $\boldsymbol{\tau}^{1,*}$ is $0$.
This allows us to define the two strips that appear in Theorem \ref{thm1_GW}.
\begin{defn}
We denote by $S^0_{\lambda}$ (resp. $S^1_{\lambda}$) the random infinite strip whose skeleton is $\boldsymbol{\tau}^0_{\lambda}$ (resp. $\boldsymbol{\tau}^{1,*}_{\lambda}$) and where conditionally on the skeleton, all the holes are filled with independent Boltzmann triangulations with parameter $\lambda$.
\end{defn}
The reason why we need to replace $\boldsymbol{\tau}^{1}_{\lambda}$ by $\boldsymbol{\tau}^{1,*}_{\lambda}$ in this last definition is linked to the root transformation between $\T_{\lambda}$ and $\T_{\lambda}^1$, and will be explained in details in the end of Subsection 2.5 (see Figure \ref{root_transformation_strip}).

\subsection{Proof of Theorem \ref{thm1_GW}}

The goal of this subsection is to prove Theorem \ref{thm1_GW}. For this, we build a random infinite $1$-admissible forest $\mathbf{F}_{\lambda}$ directly in terms of its decomposition in reverse infinite trees, and we show that it has the same distribution as $\mathrm{Skel}(\T_{\lambda}^1)$.

We recall that $\mu_{\lambda}(0)=0$ and $\mu_{\lambda}(k)=m_{\lambda}(1-m_{\lambda})^{k-1}$ for $k \geq 1$. We have seen that an infinite $1$-admissible forest $\mathbf{f}$ is completely described by the tree $\mathbf{U}(\mathbf{f})$ and the infinite trees that $\mathbf{f}$ contains. Therefore, there is a unique (in distribution) random $1$-admissible forest $\mathbf{F}_{\lambda}$ such that $\mathbf{U}(\mathbf{F}_{\lambda})$ is a Galton--Watson tree with offspring distribution $\mu_{\lambda}$ and, conditionally on $\mathbf{U}(\mathbf{F}_{\lambda})$:
\begin{itemize}
\item[(i)]
the trees of $\mathbf{F}_{\lambda}$ are independent,
\item[(ii)]
the unique tree that reaches reverse height $0$ has the same distribution as $\boldsymbol{\tau}^1_{\lambda}$,
\item[(iii)]
all the other trees have the same distribution as $\boldsymbol{\tau}^0_{\lambda}$ described above.
\end{itemize}
A more rigorous (but heavier) way to define $\mathbf{F}$ would be to build explicitly $B_r(\mathbf{F})$ by concatenating independent forests of the form $B_{j}(\boldsymbol{\tau}^0)$ and $B_{j}(\boldsymbol{\tau}^1)$.

In order to prove that $\mathbf{F}_{\lambda}$ has indeed the same distribution as $\mathrm{Skel}(\T_{\lambda}^1)$, we need to introduce one last notation. Let $\mathbf{f}$ be a $1$-admissible forest, and let $r \geq 1$. Let $\ell$ be the number of infinite reverse trees of $\mathbf{f}$ that intersect $B_r(\mathbf{f})$, and let $\mathbf{t}_1^r, \dots, \mathbf{t}_{\ell}^r$ be these trees, from left to right. For every $1 \leq j \leq \ell$, we denote by $p_j^r(\mathbf{f})$ the number of vertices of $\mathbf{t}_j^r$ whose reverse height in $\mathbf{f}$ is exactly $r$ (see Figure \ref{tree_mathfrak_T} for an example). Note that for each $1 \leq j \leq \ell$, the trees of $B_r(\mathbf{f})$ that belong to $\mathbf{t}_j^r$ are $p_j^r(\mathbf{f})$ consecutive trees of $B_r(\mathbf{f})$. Therefore, if we already know $B_r(\mathbf{f})$, knowing the values $p_j^r(\mathbf{f})$ is equivalent to knowing which of the trees of $B_r(\mathbf{f})$ lie in the same infinite reverse tree of $\mathbf{f}$.
We write $\widetilde{B}_r(\mathbf{f})$ for the pair $\left( B_r(\mathbf{f}), (p_1^r(\mathbf{f}), \dots, p_{\ell}^r(\mathbf{f})) \right)$. The reason why we introduce this object is that its distribution for $\mathbf{f}=\mathbf{F}_{\lambda}$ will be easier to compute.

The key result is the next proposition. As explained in Section 2.1, a $1$-admissible forest $\mathbf{f}$ is characterized by its reordered balls $B'_r(\mathbf{f})$, so the distribution of $\mathrm{Skel}(\T_{\lambda}^1)$ is characterized by the distribution of its reordered balls. Therefore, the next proposition will imply that $\mathbf{F}_{\lambda}$ and $\mathrm{Skel}(\T_{\lambda}^1)$ have the same distribution, which implies Theorem \ref{thm1_GW}.

\begin{prop} \label{otherconstructionForest}
For every $r \geq 0$, the forests $B'_r(\mathbf{F}_{\lambda})$ and $\mathrm{Skel} \left(B_r^{\bullet}(\T_{\lambda}^1) \right)$ have the same distribution.
\end{prop}

\begin{proof}
In all this proof, we will fix $0 < \lambda \leq \lambda_c$ and omit the parameter $\lambda$ in the notation. The idea of the proof is the following: we will first compute the law of $\widetilde{B}_r(\mathbf{F})$, then that of $B_r(\mathbf{F})$ and finally that of $B_r'(\mathbf{F})$.

First, we know that $B_r(\mathbf{U}(\mathbf{F}))$ is a tree with height $r$ in which all the leaves lie at height $r$. Moreover, if $t$ is such a tree, we have
\begin{equation}\label{loiarbre}
\mathbb{P} \left( B_r(\mathbf{U}(\mathbf{F}))=t \right)=\prod_{v \in t_{< r}} m(1-m)^{c_v-1}=m^{|t_{<r}|} (1-m)^{|t|-1-|t_{<r}|},
\end{equation}
where $t_{< r}$ is the set of vertices of $t$ at height strictly less than $r$, and $c_v$ is the number of children of a vertex $v$.

Now let $\left( f, (p_1, \dots, p_{\ell}) \right)$ be a possible value of $\widetilde{B}_r(\mathbf{F})$, i.e. $f=(t_1, \dots, t_p)$ is a forest of height $r$, the positive integers $p_1, \dots, p_{\ell}$ satisfy $p_1+ \dots + p_{\ell}=p$, and $f$ has a unique vertex at reverse height $0$ which lies in one of the trees $t_1, \dots, t_{p_1}$. For every $1 \leq j \leq \ell$, we write $f_j=(t_{p_1+\dots+p_{j-1} +1}, \dots, t_{p_1+\dots+p_j})$ and $h(f_j)$ for the height of the forest $f_j$. Each of these forests corresponds to one of the infinite trees that reach reverse height $r$.

We now check that the tree $B_r(\mathbf{U}(\mathbf{F}))$ is a measurable function of $\widetilde{B}_r(\mathbf{F})$ (although not of $B_r(\mathbf{F})$). The reader may find helpful to look at Figure \ref{tree_mathfrak_T} while reading what follows. We consider the tree $u$ whose vertices are the pairs $(j,i)$ with $1 \leq j \leq \ell$ and $r-h(f_j) \leq i \leq r$, and in which:
\begin{itemize}
\item[(i)]
the pair $(1,0)$ is the root vertex, 
\item[(ii)]
if $i>r-h(f_j)$, then the parent of $(j,i)$ is $(j,i-1)$,
\item[(iii)]
if $i=r-h(f_j)$, then the parent of $(j,i)$ is $(k,i-1)$, where $k$ is the greatest integer such that $k<j$ and $h(f_k) \geq r-i+1$.
\end{itemize}
This is the natural analog of $\mathbf{U}(\mathbf{f})$ for the finite forest $f$ (cf. Figure \ref{tree_mathfrak_T}). Note that for every $1 \leq j \leq \ell$, there are exactly $h(f_j)+1$ vertices of $u$ of the form $(j,i)$, exactly one of which lies at height $r$. Hence, we have
\begin{equation}\label{sum_h_equal_numer_vertices}
\sum_{j=1}^{\ell} h(f_j)=|u_{<r}|.
\end{equation}
It is easy to see that if $\widetilde{B}_r(\mathbf{F})=\left( f, (p_1, \dots, p_{\ell}) \right)$, then $B_r(\mathbf{U}(\mathbf{F}))=u$.
It is also possible to read the heights $h(f_j)$ on the tree $u$: first, we have $h(f_1)=r$. Moreover, if $2 \leq j \leq \ell$, let $x_j$ be the $j$-th leaf in $u$ starting from the left. Then $h(f_j)$ is the smallest $h$ such that the ancestor of $x_j$ at height $r-h$ is not the leftmost child of its parent (cf. Figure \ref{tree_mathfrak_T}).
Hence, conditionally on $B_r(\mathbf{U}(\mathbf{F}))=u$, the forest $\widetilde{B}_r(\mathbf{F})$ is a concatenation of $\ell$ independent forests of heights $h(f_1), \dots, h(f_{\ell})$. The first forest has the same distribution as $B_r(\boldsymbol{\tau}^1)$ and, for $j \geq 2$, the $j$-th forest has the same distribution as $B_{h(f_j)}(\boldsymbol{\tau}^0)$.
By combining this observation with \eqref{loiarbre} and Lemma \ref{distrib_ball_tau}, we obtain
\begin{align} \label{distrib_B_tilde}
\mathbb{P} \left( \widetilde{B}_r(\mathbf{F})=\left( f, (p_i)_{1 \leq i \leq \ell} \right) \right) &= m^{|u_{<r}|} (1-m)^{|u|-|u_{<r}|-1} \times \frac{\pi(p_1) m^{-r}}{\theta(0)} \prod_{v \in f_1} \theta (c_v) \nonumber \\
& \hspace{4.52cm} \times \prod_{j=2}^{\ell} \frac{\pi(p_j) m^{-h(f_j)}}{\Pi(\theta(0))} \prod_{v \in f_j} \theta (c_v) \nonumber \\
&= \frac{m^{|u_{<r}|-\sum_{j=1}^{\ell} h(f_j)} (1-m)^{|u|-|u_{<r}|-1} }{\theta(0) \, \Pi(\theta(0))^{\ell-1}} \prod_{j=1}^{\ell} \pi(p_j) \times \prod_{v \in f} \theta(c_v) \nonumber \\
&= \frac{1}{\theta(0)} \left(\frac{1-m}{\Pi(\theta(0))} \right)^{\ell-1} \prod_{j=1}^{\ell} \pi(p_j) \times \prod_{v \in f} \theta(c_v),
\end{align}
where in the end, we used \eqref{sum_h_equal_numer_vertices}. Moreover, by Lemma \ref{Piexact}, we can compute $\frac{1-m}{\Pi(g(0))} = \sqrt{1-4h}$.

We now compute the distribution of $B_r(\mathbf{F})$: let $f=(t_1, \dots, t_p)$ be a possible value of $B_r(\mathbf{F})$, and let $i_0$ be the index such that the only vertex of reverse height $0$ belongs to $t_{i_0}$. We need to sum \eqref{distrib_B_tilde} over all the possible values of $\ell \geq 1$ and $p_1, \dots, p_{\ell} \geq 1$ such that $\sum_{j=1}^{\ell} p_j=p$ and $p_1 \geq i_0$ (by construction of $\mathbf{F}$, the lowest vertex always belongs to the leftmost infinite tree). We obtain
\begin{equation} \label{loiBr}
\mathbb{P} \left( B_r(\mathbf{F})=f \right)=\frac{1}{\theta(0)} \bigg( \sum_{\ell \geq 1} (1-4h)^{\frac{\ell-1}{2}} \sum_{\underset{p_1 \geq i_0}{p_1+ \dots +p_{\ell}=p}} \prod_{j=1}^{\ell} \pi(p_j) \bigg) \prod_{v \in f} \theta(c_v).
\end{equation}

Now let $f'=(t_1, \dots, t_p)$ be a possible value of $B'_r(\mathbf{F})$, i.e. a forest of height $r$ in which the only vertex of reverse height $0$ lies in $t_1$. To obtain $\mathbb{P} \left( B'_r(\mathbf{F})=f' \right)$, we need to sum Equation \eqref{loiBr} over all the forests one may get by applying a cyclic permutation to the trees of $f'$. The values of $p$ and $\prod_{v \in f} \theta(c_v)$ are the same for all these forests, but the value of $i_0$ ranges from $1$ to $p$, so we have
\begin{eqnarray*}
\mathbb{P} \left( B'_r(\mathbf{F})=f' \right) &=& \frac{1}{\theta(0)} \bigg( \sum_{i_0=1}^p \sum_{\ell \geq 1} (1-4h)^{\frac{\ell-1}{2}} \sum_{\underset{p_1 \geq i_0}{p_1+ \dots +p_{\ell}=p}} \prod_{j=1}^{\ell} \pi(p_j) \bigg) \prod_{v \in f} \theta(c_v)\\
&=& \frac{1}{\theta(0)} \bigg( \sum_{\ell \geq 1} (1-4h)^{\frac{\ell-1}{2}} \sum_{p_1+ \dots +p_{\ell}=p} p_1 \prod_{j=1}^{\ell} \pi(p_j) \bigg) \prod_{v \in f} \theta(c_v).
\end{eqnarray*}
By comparing this to \eqref{LoiForet}, we only need to prove that for every $p \geq 1$,
\begin{equation}\label{coefficients}
 \sum_{\ell \geq 1} (1-4h)^{\frac{\ell-1}{2}} \sum_{p_1+ \dots +p_{\ell}=p} p_1 \prod_{j=1}^{\ell} \pi(p_j) =\frac{p h_{\lambda}(p)}{h_{\lambda}(1)}.
\end{equation}
It is enough to show that the generating functions of both sides coincide. But the generating function of the left-hand side is
\begin{eqnarray*}
\sum_{p \geq 1} \bigg( \sum_{\ell \geq 1} (1-4h)^{\frac{\ell-1}{2}} \sum_{p_1+ \dots +p_{\ell}=p} p_1 \prod_{j=1}^{\ell} \pi(p_j) \bigg) y^p &=& \sum_{\ell \geq 1} (1-4h)^{\frac{\ell-1}{2}} \,y\, \Pi'(y) \Pi(y)^{\ell-1}\\
&=& \frac{y \, \Pi'(y)}{1-\sqrt{1-4h} \, \Pi(y)},
\end{eqnarray*}
which is explicitly known by Lemma \ref{Piexact}.
On the other hand, we recall from Lemma \ref{skeleton} that $h_{\lambda}(p)=\frac{1}{p} \left( 8+\frac{1}{h} \right)^{-p} c_{\lambda}(p)$. Hence, the generating function of the right-hand side of \eqref{coefficients} is
\[ \frac{1}{c_{\lambda}(1)} \left( 8+\frac{1}{h} \right) \sum_{p \geq 1} \left( 8+\frac{1}{h} \right)^{-p} c_{\lambda}(p) \, y^{p} =\frac{1}{y} \frac{1}{c_{\lambda}(1)} \left( 8+\frac{1}{h} \right) C_{\lambda} \left( \frac{h}{1+8h} y \right),\]
where $C_{\lambda}$ is given by \eqref{generating_cplambda}. Hence, it is easy to check that the generating functions of both sides of \eqref{coefficients} coincide, which concludes the proof.
\end{proof}

\begin{proof}[End of the proof of Theorem \ref{thm1_GW}]
By Proposition \ref{otherconstructionForest}, the skeleton of $\T_{\lambda}^1$ has the same distribution as $\mathbf{F}_{\lambda}$. In particular, the infinite leftmost geodesics are the paths separating infinite trees in $\mathbf{F_{\lambda}}$, so their union is isomorphic to $\mathbf{U}(\mathbf{F}_{\lambda})$. Moreover, consider the strips delimited by $\mathbf{U}(\mathbf{F}_{\lambda})$ in $\T^1_{\lambda}$. Let $e_1$ be the root edge of $\T_{\lambda}^1$ (i.e. the loop on its boundary).
The skeletons of the strips that are not adjacent to $e_1$ are independent copies of $\boldsymbol{\tau}^0$ and all the holes are filled with independent Boltzmann triangulations with parameter $\lambda$, so these strips are independent copies of $S^0$.

The case of the strip containing $e_1$ needs to be handled more carefully because of the presence of the bottom boundary (actually, this is not exactly a strip in the sense of Definition \ref{defn_strip}). More precisely, let $S^{1,+}$ be the random strip obtained from $\boldsymbol{\tau}^1$ as on the left part of Figure \ref{root_transformation_strip}, where all the green holes are filled with independent Boltzmann triangulations with parameter $\lambda$. Then the strip of $\T_{\lambda}^1$ adjacent to $e_1$ has skeleton $\boldsymbol{\tau}^1$ and its holes are filled with Boltzmann triangulations, but it has a bottom boundary of length one, so it has the same distribution as $S^{1,+}$.

Finally, we recall that $\T_{\lambda}$ is the image of $\T_{\lambda}^1$ by the following root transformation. Let $f_1$ be the face of $\T_{\lambda}^1$ adjacent to its boundary loop $e_1$. We obtain $\T_{\lambda}$ by contracting $e_1$ and gluing together the two other sides of $f_1$. This transformation does not affect the tree of infinite leftmost geodesics and the strips that are not adjacent to $e_1$. Its effect on the strip $S^{1,+}$ adjacent to $e_1$ is described on Figure \ref{root_transformation_strip}, and the strip we obtain is then $S^1$. Note also that the root edge on the right of Figure \ref{root_transformation_strip} may be on the boundary of the strip, but in this case it is necessarily on the left boundary, so $S^1$ is always the strip containing the face on the \emph{right} of the root edge, as in the statement of the Theorem. This proves Theorem \ref{thm1_GW}.
\end{proof}

\begin{figure}
\begin{center}
\begin{tikzpicture}
\fill[green!20] (0.5,2)--(3,1)--(0,1);
\fill[green!20] (1.5,2) to[bend left=15] (3,1)--(1.5,2);
\fill[green!20] (1,0) to[bend left] (0,1)--(1,0);
\fill[green!20] (1,0)--(2,0)--(1,1)--(1,0);
\fill[green!20] (2,0) to[bend left] (2,1) to[bend left] (2,0);

\draw(1.5,2)--(3,1);
\draw(1.5,2) to[bend left=15] (3,1);
\draw(0.5,2)--(3,1);
\draw(1,1)--(1,0);
\draw(1,1)--(2,0);
\draw(2,1) to[bend left] (2,0);
\draw(2,1) to[bend right] (2,0);
\draw(0,1)--(1,0);

\draw[thick, blue] (0.5,2)--(1.5,2);
\draw[thick, blue] (1.5,2)--(2.5,2);
\draw[thick, blue] (0,1)--(1,1);
\draw[thick, blue] (1,1)--(2,1);
\draw[thick, blue] (2,1)--(3,1);

\draw[very thick, red, ->] (1,0)--(1.8,0);
\draw[very thick, red] (1,0)--(2,0);
\draw[thick, red] (1,0)--(1.3,0.4);
\draw[thick, red] (2,0)--(1.3,0.4);

\draw[thick, violet](1.25,2.5)--(1,2);
\draw[thick, violet](1.75,2.5)--(2,2);
\draw[thick, violet](1,2)--(0.5,1);
\draw[thick, violet](1,2)--(1.5,1);
\draw[thick, violet](1,2)--(2.5,1);
\draw[thick, violet](1.5,1)--(1.5,0);

\draw[thick, orange](1,0) to[bend left] (0,1)--(0.5,2)--(0.5,2.5);
\draw[thick, orange](2,0)--(3,1)--(2.5,2)--(2.5,2.5);

\draw (0.5,2) node{};
\draw (1.5,2) node{};
\draw (2.5,2) node{};
\draw (0,1) node{};
\draw (1,1) node{};
\draw (2,1) node{};
\draw (3,1) node{};
\draw (1,0) node{};
\draw (2,0) node{};
\draw (1.3,0.4) node{};

\draw (1,2) node[edge']{};
\draw (2,2) node[edge']{};
\draw (0.5,1) node[edge']{};
\draw (1.5,1) node[edge']{};
\draw (2.5,1) node[edge']{};
\draw (1.5,0) node[rootedge]{};

\draw[red] (1.5, -0.3) node[texte]{$e_1$};
\draw (5.5,1) node[texte]{$\xrightarrow[]{\mathrm{root \, transformation}}$};

\begin{scope}[shift={(8,0)}]
\fill[green!20] (0.5,2)--(3,1)--(0,1);
\fill[green!20] (1.5,2) to[bend left=15] (3,1)--(1.5,2);
\fill[green!20] (1.5,0) to[bend left] (0,1)--(1.5,0);
\fill[green!20] (1.5,0) to[bend left] (1,1) to[bend left] (1.5,0);
\fill[green!20] (1.5,0) to[bend left=15] (2,1) to[bend left=15] (1.5,0);

\draw(1.5,2)--(3,1);
\draw(1.5,2) to[bend left=15] (3,1);
\draw(0.5,2)--(3,1);
\draw(1,1) to[bend left] (1.5,0);
\draw(1,1) to[bend right] (1.5,0);
\draw(2,1) to[bend left=15] (1.5,0);
\draw(2,1) to[bend right=15] (1.5,0);
\draw(0,1)--(1.5,0);

\draw[thick, blue] (0.5,2)--(1.5,2);
\draw[thick, blue] (1.5,2)--(2.5,2);
\draw[thick, blue] (0,1)--(1,1);
\draw[thick, blue] (1,1)--(2,1);
\draw[thick, blue] (2,1)--(3,1);

\draw[thick, violet](1.25,2.5)--(1,2);
\draw[thick, violet](1.75,2.5)--(2,2);
\draw[thick, violet](1,2)--(0.5,1);
\draw[thick, violet](1,2)--(1.5,1);
\draw[thick, violet](1,2)--(2.5,1);

\draw[thick, orange](1.5,0) to[bend left] (0,1)--(0.5,2)--(0.5,2.5);
\draw[thick, orange](1.5,0)--(3,1)--(2.5,2)--(2.5,2.5);

\draw[very thick, red,->] (1.5,0)--(1.35,0.3);
\draw[very thick, red] (1.5,0)--(1.25,0.5);

\draw (0.5,2) node{};
\draw (1.5,2) node{};
\draw (2.5,2) node{};
\draw (0,1) node{};
\draw (1,1) node{};
\draw (2,1) node{};
\draw (3,1) node{};
\draw (1.5,0) node{};
\draw (1.25,0.5) node{};

\draw (1,2) node[edge']{};
\draw (2,2) node[edge']{};
\draw (0.5,1) node[edge']{};
\draw (1.5,1) node[edge']{};
\draw (2.5,1) node[edge']{};

\end{scope}
\end{tikzpicture}
\end{center}
\caption{On the left, the strip $S^{1,+}$ and its skeleton $\boldsymbol{\tau}^1$. On the right, the strip $S^1$ and its skeleton $\boldsymbol{\tau}^{1,*}$. The root transformation contracting the face $f_1$ (whose sides are in red on the left) sends the first strip to the second. The skeletons are in violet.}\label{root_transformation_strip}
\end{figure}
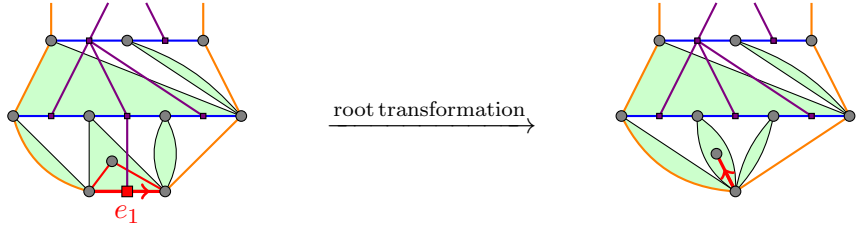

\subsection{Construction of reverse Galton--Watson trees and infinite strips}\label{subsec:construction_reversee_trees}

Our goal in this subsection is to construct the reverse subcritical trees $\boldsymbol{\tau}^0$ and $\boldsymbol{\tau}^1$ and to prove Lemma \ref{distrib_ball_tau} and a few useful estimates. We will also show that $\boldsymbol{\tau}^0$ is the local limit as $n \to +\infty$ of a subcritical Galton--Watson tree conditioned to die at generation $n$, seen from its last generation. This is different from the more usual Galton--Watson tree conditioned to survive \cite{Kes86}, where the tree is seen from its root. Moreover, $\boldsymbol{\tau}^1$ is just $\boldsymbol{\tau}^0$, conditioned to have only one vertex at reverse height $0$. Our trees will be built by a spine decomposition approach, which will be useful to obtain geometric estimates on these trees in Section 4. Since our proofs hold in a more general setting, we present the results of this subsection for any critical or subcritical Galton--Watson process, and denote by $g$ the generating function of the offspring distribution. The trees $\boldsymbol{\tau}^0_{\lambda}$ and $\boldsymbol{\tau}^1_{\lambda}$ that we need are obtained by taking $g=g_{\lambda}$.

We start with a vertical half-line which is infinite on the top side, that we call the \emph{spine}. The root of the tree will be the lowest point on the spine. For every $r  \geq 1$, we sample a random pair $(L_r, R_r)$ of integers such that all these pairs are independent and, for every $i,j \geq 0$, we have
\begin{equation}\label{distribLR}
\mathbb{P} \left( L_r=i, R_r=j \right)=\frac{g^{\circ r}(0)-g^{\circ (r-1)}(0)}{g^{\circ (r+1)} (0)-g^{\circ r}(0)} \theta(i+j+1) \, g^{\circ (r-1)}(0)^i \, g^{\circ r}(0)^j.
\end{equation}
To check that \eqref{distribLR} defines a probability measure, we just need to sum the right-hand side over pairs $(i,j)$ with a fixed value of $k=i+j$ and then sum over $k$. 
Similarly, by a straightforward computation, we obtain
\begin{equation}\label{eqn_expect_LR}
\E \left[ L_r+R_r \right] = \frac{g^{\circ r} (0) g'( g^{\circ r} (0))-g^{\circ (r-1)} (0) g'( g^{\circ (r-1)} (0)) }{g^{\circ (r+1)} (0)-g^{\circ r}(0)}-1.
\end{equation}

We call $s_r$ the vertex at height $r$ in the spine. For all $r \geq 2$, let $t_1^r, \dots, t_{L_r}^r$ be $L_r$ independent Galton--Watson trees with offspring distribution $\theta$ conditioned on having height at most $r-2$. We graft $L_r$ edges to the left of the vertex $s_r$, and the trees $t_1^r, \dots, t_{L_r}^r$ to the other ends of these edges. Similarly, for all $r \geq 1$, let $u_1^r, \dots, u_{R_r}^r$ be $R_r$ independent Galton--Watson trees with offspring distribution $\theta$ conditioned on having height at most $r-1$. We graft $R_r$ edges to the right of $s_r$ and the trees $u_1^r, \dots, u_{R_r}^r$ to the other ends of these edges (see Figure \ref{figureTau0}).

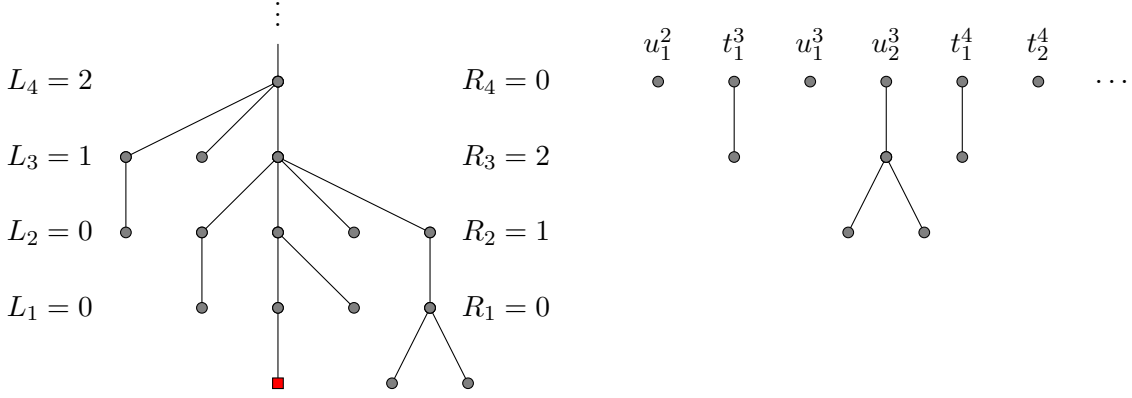
\begin{figure}
\begin{center}
\begin{tikzpicture}

\draw (0,0)node[rootedge]{}--(0,1)node{};
\draw (0,1)node{}--(0,2)node{};
\draw (0,2)node{}--(0,3)node{};
\draw (0,3)node{}--(0,4)node{};
\draw (0,4)node{}--(0,4.5);
\draw (0,2)node{}--(1,1)node{};
\draw (0,3)node{}--(-1,2)node{};
\draw (-1,2)node{}--(-1,1)node{};
\draw (0,3)node{}--(1,2)node{};
\draw (0,3)node{}--(2,2)node{};
\draw (2,2)node{}--(2,1)node{};
\draw (2,1)node{}--(1.5,0)node{};
\draw (2,1)node{}--(2.5,0)node{};
\draw (0,4)node{}--(-1,3)node{};
\draw (0,4)node{}--(-2,3)node{};
\draw (-2,3)node{}--(-2,2)node{};

\draw (0,5)node[texte]{$\vdots$};
\draw (-3,1)node[texte]{$L_1=0$};
\draw (-3,2)node[texte]{$L_2=0$};
\draw (-3,3)node[texte]{$L_3=1$};
\draw (-3,4)node[texte]{$L_4=2$};
\draw (3,1)node[texte]{$R_1=0$};
\draw (3,2)node[texte]{$R_2=1$};
\draw (3,3)node[texte]{$R_3=2$};
\draw (3,4)node[texte]{$R_4=0$};

\draw (5,4.5)node[texte]{$u^2_1$};
\draw (5,4)node{};
\draw (6,4.5)node[texte]{$t^3_1$};
\draw (6,4)node{}--(6,3)node{};
\draw (7,4.5)node[texte]{$u^3_1$};
\draw (7,4)node{};
\draw (8,4.5)node[texte]{$u^3_2$};
\draw (8,4)node{}--(8,3)node{};
\draw (8,3)node{}--(7.5,2)node{};
\draw (8,3)node{}--(8.5,2)node{};
\draw (9,4.5)node[texte]{$t^4_1$};
\draw (9,4)node{}--(9,3)node{};
\draw (10,4.5)node[texte]{$t^4_2$};
\draw (10,4)node{};
\draw (11,4)node[texte]{$\dots$};
\end{tikzpicture}
\end{center}
\caption{Construction of the tree $\boldsymbol{\tau}^0$. Every vertex lies below its parent.}\label{figureTau0}
\end{figure}

We denote by $\boldsymbol{\tau}^0$ the infinite tree we obtain, and we define a genealogy on it: for every $r \geq 1$, the children of $s_r$ are $s_{r-1}$ and the roots of the trees $t_i^r$ and $u_j^r$. Inside the trees $t_i^r$ and $u_j^r$, the genealogy is the usual one in a Galton--Watson tree. We fix the reverse height of the root at $0$ and declare that the parent of a vertex of reverse height $r$ has reverse height $r+1$ (it corresponds to the height of the vertices on Figure \ref{figureTau0}). By the conditioning we have chosen for the trees $t_i^r$ and $u_j^r$, every vertex of the tree has a nonnegative reverse height, and the root is the leftmost vertex with reverse height $0$.

\begin{lem}\label{taufinite}
For every $r \geq 0$, the tree $\boldsymbol{\tau}^0$ has a finite number of vertices at height $r$.
\end{lem}

\begin{proof}
We fix $r \geq 0$ and take $r'>r$, and we define $A_r^{r'}$ as the event
\[\{ \mbox{one of the trees grafted at $s_{r'}$ reaches reverse height $r$ in $\boldsymbol{\tau}^0$} \}.\]
We want to show that $\P \big( A_r^{r'} \big)$ decreases fast enough in $r'$.
Consider one of the trees $t_i^{r'}$ grafted on the left of the spine at $s_{r'}$. This is a Galton--Watson tree conditioned to have height at most $r'-2$, so the probability that $t_i^{r'}$ has height at least $r'-r-1$ is $\frac{g^{\circ (r'-1)}(0)-g^{\circ (r'-r-1)}(0)}{g^{\circ (r'-1)}(0)}$. The denominator is larger than $\frac{1}{2}$ for $r'$ large enough.
Hence, the probability that one of the $L_{r'}$ trees grafted on the left of $s_{r'}$ reaches reverse height $r$ is at most $2 \left( g^{\circ (r'-1)}(0)-g^{\circ (r'-r-1)}(0) \right) \E \left[ L_{r'} \right]$.
Similarly, the probability of the analog event on the right is at most $2 \left( g^{\circ r'}(0)-g^{\circ (r'-r-1)}(0) \right) \E \left[ R_{r'} \right]$. Therefore, for $r'$ large enough, we have
\[ \P \left( A_r^{r'} \right) \leq 2 \left( g^{\circ r'}(0)-g^{\circ (r'-r-1)}(0) \right) \E \left[ L_{r'}+R_{r'} \right]. \]
Moreover, the right-hand side of \eqref{eqn_expect_LR} can be rewritten $\frac{\widetilde{g}(x_r)-\widetilde{g}(x_{r-1})}{x_{r+1}-x_r}$, where we write $x_r=g^{\circ r}(0)$ and $\widetilde{g}(x)=x g'(x)-g(x)$ for $x \in [0,1]$. We want to prove $\sum_{r'>r} \P \left( A_r^{r'} \right)<+\infty$, that is
\begin{equation}\label{eqn:bidouille_xr}
2 \sum_{r'>r} \frac{x_{r'}-x_{r'-r-1}}{x_{r'+1}-x_{r'}} \left( \widetilde{g}(x_{r'})-\widetilde{g}(x_{r'-1}) \right)<+\infty.
\end{equation}
Since $g$ is $1$-Lipschitz, the sequence $(x_{r'+1}-x_{r'})_{r' \geq 0}$ is nonincreasing, so we have
\[ \frac{x_{r'}-x_{r'-r-1}}{x_{r'+1}-x_{r'}} \leq (r+1) \frac{x_{r'-r}-x_{r'-r-1}}{x_{r'+1}-x_{r'}} = \frac{(r+1) \left( x_{r'-r}-x_{r'-r-1} \right) }{g^{\circ (r+1)}(x_{r'-r})-g^{\circ (r+1)}(x_{r'-r-1})} \to \frac{r+1}{\left( g^{\circ(r+1)} \right)'(1)} \]
as $r' \to +\infty$. In particular, in \eqref{eqn:bidouille_xr}, the first factor is bounded while the second is nonnegative and telescopic, so the sum converges. Therefore, a.s. $A_r^{r'}$ does not occur for $r'$ large enough, which proves the lemma.
\end{proof}

For $r \geq 0$, let $d_r(\boldsymbol{\tau}^0)$ be the tree of descendants of the $r$-th vertex $s_r$ of the spine.

\begin{lem}\label{GW_conditioned_die_at_r}
The tree $d_r(\boldsymbol{\tau}^0)$ has the same distribution as a Galton--Watson tree with offspring distribution $\theta$, conditioned to have height exactly $r$.
\end{lem}

\begin{proof}
Let $t$ be a finite plane tree of height $r$. Let $s_0(t)$ be its leftmost vertex of reverse height $0$ and let $s_r(t)$ be its root. Let $\left( s_0(t), s_1(t), \dots, s_r(t) \right)$ be the unique geodesic path from $s_0(t)$ to $s_r(t)$. For every $1 \leq i \leq r$, let $\ell_i(t)$ (resp. $r_i(t)$) be the number of children of $s_i(t)$ on the left (resp. on the right) of $s_{i-1}(t)$, and let $\left( v^i_j(t) \right)_{1 \leq j \leq \ell_i(t)}$ (resp. $\left( w^i_j(t) \right)_{1 \leq j \leq r_i(t)}$) be those children. Finally, for every $i$ and $j$, let $t_j^i(t)$ (resp. $u_j^i(t)$) be the tree of descendants of $v^i_j(t)$ (resp. $w^i_j(t)$). Then we have
\begin{equation}\label{distrib_tr(tau0)}
\P \left( d_r(\boldsymbol{\tau}^0)=t \right)=\prod_{i=1}^r \P \left( L_i=\ell_i(t), R_i=r_i(t) \right) \times \prod_{i=1}^r \left( \prod_{j=1}^{\ell_i(t)} \P \left( t^i_j=t^i_j(t) \right) \times \prod_{j=1}^{r_i(t)} \P \left( u^i_j=u^i_j(t) \right) \right).
\end{equation}
Moreover, by definition of $t_j^i$ and $u_j^i$, we have
\[\P \left( t^i_j=t^i_j(t) \right)=\frac{1}{g^{\circ (r-1)}(0)} \prod_{v \in t^i_j(t)} \theta(c_v) \quad \mbox{ and } \quad \P \left( u^i_j=u^i_j(t) \right)=\frac{1}{g^{\circ r}(0)} \prod_{v \in u^i_j(t)} \theta(c_v). \]
By combining this with \eqref{distribLR} and \eqref{distrib_tr(tau0)}, we obtain
\[ \P \left( d_r(\boldsymbol{\tau}^0)=t \right)= \frac{1}{g^{\circ (r+1)}(0)-g^{\circ r}(0)} \prod_{v \in t} \theta(c_v),\]
which concludes.
\end{proof}

Now let $r \geq 0$. By Lemma \ref{taufinite}, there is $r' \geq r$ such that all the vertices of reverse height at most $r$ lie in $d_{r'}(\boldsymbol{\tau}^0)$. Hence, $\boldsymbol{\tau}^0$ is the a.s. local limit as $r' \to +\infty$ of $d_{r'}(\boldsymbol{\tau}^0)$ rooted at its leftmost vertex of height $r'$. By Lemma \ref{GW_conditioned_die_at_r}, this proves that $\boldsymbol{\tau}^0$ is the local limit (in distribution) as $r' \to +\infty$ of Galton--Watson trees conditioned on extinction at time $r'$, seen from the last generation.

In particular, for every $r \geq 0$, let $Y(r)$ be the number of vertices of $\boldsymbol{\tau}^0$ at reverse height $r$. Then $\left( Y(-r) \right)_{r \leq 0}$ is the limit in distribution as $n \to +\infty$ of a Galton--Watson process with offspring distribution $\theta$, started from $1$ at time $-n$ and conditioned on extinction at time exactly $1$. Hence, $Y$ has the same distribution as the reverse Galton--Watson process described by Esty in \cite{E75}, started from $0$ at time $-1$. In particular, by Equation 3 of \cite{E75}, we obtain explicitly the distribution of $Y(r)$.

\begin{lem}\label{distribY}
For all $r \geq 0$ and $p \geq 1$, we have
\[\mathbb{P} \left( Y(r)=p\right)=\frac{\pi(p) m^{-r}}{\Pi(\theta(0))} \left( g^{\circ (r+1)} (0)^p - g^{\circ r} (0)^p \right).\]
\end{lem}

We also define $\boldsymbol{\tau}^1$ as the tree $\boldsymbol{\tau}^0$, conditioned on $Y(0)=1$. We can now prove Lemma \ref{distrib_ball_tau}.

\begin{proof}[Proof of Lemma \ref{distrib_ball_tau}]
We start with the first part of the lemma (i.e. the part related to $\boldsymbol{\tau}^0$).

Let $R$ be the smallest $r'$ such that $s_{r'}$ is an ancestor of all the vertices at reverse height $r$ in $\boldsymbol{\tau}^0$. For every $r'>r$, the event $\{R \leq r'\}$ depends only on the trees $t_j^i$ and $u_j^i$ for $i>r'$, so it is independent of the tree $d_{r'}(\boldsymbol{\tau}^0)$. We write $Y(r,r')$ for the number of descendants of $s_{r'}$ at reverse height $r$, and $B_r^{r'} \! \left( \boldsymbol{\tau}^0 \right)$ for the forest consisting of the trees of descendants of these vertices. We have
\begin{align*}
\P \left( B_r ( \boldsymbol{\tau}^0 )=(t_1, \dots, t_p) \right) &= \lim_{r' \to +\infty} \P \left( R \leq r', B_r ( \boldsymbol{\tau}^0 )=(t_1, \dots, t_p) \right)\\
&= \lim_{r' \to +\infty} \P \left( R \leq r', B_r^{r'} \! ( \boldsymbol{\tau}^0 )=(t_1, \dots, t_p) \right)\\
&= \lim_{r' \to +\infty} \P \left(  R \leq r' \right) \P \left( Y(r,r')=p \right)\\
& \hspace{2cm} \times \P \left( B_r^{r'} \! ( \boldsymbol{\tau}^0 )=(t_1,\dots, t_p) \big| Y(r,r')=p \right).
\end{align*}
But by Lemma \ref{GW_conditioned_die_at_r}, the tree $d_{r'}(\boldsymbol{\tau}^0)$ has the same distribution as a Galton--Watson tree conditioned to have height $r'$. From here, it is easy to show that the distribution of $B_r^{r'} ( \boldsymbol{\tau}^0 )$ conditioned on $Y(r,r')=p$ is that of a Galton--Watson forest of $p$ trees conditioned to have height $r$, i.e.
\[ \P \left( B_r^{r'} \! ( \boldsymbol{\tau}^0 )=(t_1,\dots, t_p) \big| Y(r,r')=p \right)=\frac{1}{ g^{\circ (r+1)} (0)^p - g^{\circ r} (0)^p} \prod_{i=1}^p \prod_{v \in t_i} \theta(c_v).\]
On the other hand, by using again the independence of $\{R \leq r'\}$ and $d_{r'}(\boldsymbol{\tau}^0)$, we have
\begin{eqnarray*}
\lim_{r' \to +\infty} \P \left(  R \leq r' \right) \P \left( Y(r,r') = p \right) &=& \lim_{r' \to +\infty} \P \left(  R \leq r', Y(r,r') = p \right)\\
&=& \lim_{r' \to +\infty} \P \left(  R \leq r', Y(r)=p \right)\\
&=& \P \left( Y(r)=p \right)\\
&=& \frac{\pi(p) m^{-r}}{\Pi(\theta(0))} \left( g^{\circ (r+1)} (0)^p - g^{\circ r} (0)^p \right)
\end{eqnarray*}
by Lemma \ref{distribY}, which proves the first part of the lemma.

Given the definition of $\boldsymbol{\tau}^1$, the proof of the second part is now easy.
For every forest $(t_1, \dots, t_p)$ of height $r$ with exactly one vertex at reverse height $0$, we have
\begin{eqnarray*}
\P \left( B_r ( \boldsymbol{\tau}^1 )=(t_1, \dots, t_p) \right) &=& \P \left( B_r ( \boldsymbol{\tau}^0 )=(t_1, \dots, t_p) \big| Y(0)=1 \right)\\
&=& \frac{\P \left( B_r ( \boldsymbol{\tau}^0 )=(t_1, \dots, t_p) \right)}{\P \left( Y(0)=1 \right)}.
\end{eqnarray*}
By Lemma \ref{distribY} we have $\P \left( Y(0)=1 \right)=\frac{\theta(0)}{\Pi(\theta(0))}$, so the second part of the lemma follows from the first one.
\end{proof}

We end up this subsection by showing that the strips $S^0_{\lambda}$ and $S^1_{\lambda}$ constructed from $\boldsymbol{\tau}^0_{\lambda}$ and $\boldsymbol{\tau}^1_{\lambda}$ are in some sense very close to each other. This will allow us in Section 4 to conclude a strip verifies a property if the other does, and therefore to avoid annoying case distinctions.

\begin{lem}\label{0almost1_strip}
Let $\lambda_n \to \lambda_c$ and let $(A_n)$ be measurable events. Then the following two assertions are equivalent:
\begin{itemize}
\item[(i)]
$\P \left( S^0_{\lambda_n} \in A_n \right) \xrightarrow[n \to +\infty]{} 0,$
\item[(ii)]
$\P \left( S^1_{\lambda_n} \in A_n \right) \xrightarrow[n \to +\infty]{} 0.$
\end{itemize}
\end{lem}

\begin{proof}
We will show that the strips $S^0_{\lambda_n}$ and $S^1_{\lambda_n}$ are absolutely continuous with respect to each other, uniformly in $n$. For this, we recall that $\boldsymbol{\tau}^{1,*}$ is the tree $\boldsymbol{\tau}^1$ in which we have cut the only vertex at reverse height $0$, and we have shifted the reverse heights by $1$.
We will first prove that $\boldsymbol{\tau}^{1,*}$ has the same distribution as $\boldsymbol{\tau}^0$ biased by $Y(0)$, and then extend this absolute continuity relation to the strips.

Indeed, for any forest $f \in \F_{p,q,r}$, we have
\[\P \left( B_r(\boldsymbol{\tau}^{1,*})=f \right)=\sum_{v \in f \backslash f^*} \P \left( B_{r+1}(\boldsymbol{\tau}^1)=f^+_v \right), \]
where $f \backslash f^*$ is the set of vertices of $f$ at reverse height $0$, and $f^+_v$ is the forest of $\mathscr{F}_{1,q,r+1}$ obtained by adding a unique child to $v$ in $f$. By the second part of Lemma \ref{distrib_ball_tau}, we obtain
\[ \P \left( B_r(\boldsymbol{\tau}^{1,*})=f \right)= \frac{\theta(1)}{\theta(0)} p \pi(q) m^{-r-1} \prod_{v \in f} \theta(c_v). \]
Combined with the first part of Lemma \ref{distrib_ball_tau}, this yields
\begin{equation} \label{expectationY0}
\frac{\P \left( B_r(\boldsymbol{\tau}^{1,*})=f \right)}{\P \left( B_r(\boldsymbol{\tau}^0)=f \right)} = \frac{\theta(1)}{\theta(0)} \frac{\Pi \left( \theta(0) \right)}{m} p.
\end{equation}
In other words, the forest $B_r(\boldsymbol{\tau}^{1,*})$ has the distribution of $B_r(\boldsymbol{\tau}^0)$, biased by $Y(0)$. Since it is true for all $r$, we can conclude that $\boldsymbol{\tau}^{1,*}$ has the distribution of $\boldsymbol{\tau}^0$ biased by $Y(0)$. Since $S^1$ is constructed from $\boldsymbol{\tau}^{1,*}$ in the exact same way as $S^0$ is constructed from $\boldsymbol{\tau}^0$, we deduce that $S^1_{\lambda_n}$ has the same distribution as $S^0_{\lambda_n}$ biased by $Y_{\lambda_n}(0)$. It remains to prove that this absolute continuity is "uniform in $n$".

More precisely, if $(A_n)$ is a sequence of measurable events, for any $n \geq 0$, we have
\begin{equation}\label{use_absolute_continuity_trees}
\P \left( S^1_{\lambda_n} \in A_n \right)=\frac{\E \left[ Y_{\lambda_n}(0) \mathbbm{1}_{S^0_{\lambda_n} \in A_n} \right]}{\E \left[ Y_{\lambda_n}(0) \right]}.
\end{equation}
Moreover, \eqref{expectationY0} shows that
\[ \E \left[ Y_{\lambda}(0) \right]=\frac{\theta_{\lambda}(0)}{\theta_{\lambda}(1)} \frac{m_{\lambda}}{\Pi_{\lambda} \left( \theta_{\lambda}(0) \right)}=\frac{1-h}{2h(1-h)}\frac{1-2h-\sqrt{1-4h}}{1-\sqrt{1-4h}},\]
which is continuous at $\lambda=\lambda_c$. Therefore, the denominator in \eqref{use_absolute_continuity_trees} converges to a positive limit, so it is enough to prove that $\P \left( S^0_{\lambda_n} \in A_n \right)\to 0$ if and only if $\E \left[ Y_{\lambda_n}(0) \mathbbm{1}_{S^0_{\lambda_n} \in A_n} \right] \to 0$. The indirect implication is immediate since $Y_{\lambda_n}(0) \geq 1$. To prove the direct one, it is enough to check the variables $Y_{\lambda_n}(0)$ are uniformly integrable.  We know that they converge to $Y_{\lambda_c}(0)$ in distribution. By the Skorokhod representation theorem, we may assume the convergence is almost sure. Since we also have convergence of the expectations by \eqref{use_absolute_continuity_trees}, by Scheffé's Lemma we have $Y_{\lambda_n}(0) \to Y_{\lambda_c}(0)$ in $L^1$. In particular, the family $\left( Y_{\lambda_n}(0) \right)$ is uniformly integrable, which proves the direct implication and finally the lemma.
\end{proof}

\section{The Poisson boundary of \texorpdfstring{$\T_{\lambda}$}{Tlambda}}

The goal of this section is to prove Theorem \ref{thm2_Poisson}. We fix $0<\lambda<\lambda_c$ until the end of the section and omit the parameter $\lambda$ in most of the notation. 

\subsection{Construction of the geodesic boundary}

We start by defining precisely the compactification of $\T$ that we will afterwards prove to be a realization of its Poisson boundary. We recall that $\partial \mathbf{T}^g$ is the set of ends of the tree $\mathbf{T}^g$, i.e. the set of infinite self-avoiding paths from the root in $\mathbf{T}^g$. If $\gamma, \gamma' \in \partial \mathbf{T}^g$, we write $\gamma \sim \gamma'$ if $\gamma=\gamma'$ or if $\gamma$ and $\gamma'$ are the left and right boundaries of one of the copies of $S^0$ or $S^1$ (cf. left part of Figure \ref{TreeAndSlices}). Note that a.s., every ray of $\mathbf{T}^g$ branches infinitely many times, so no ray is equivalent to two distinct other rays. It follows that $\sim$ is a.s. an equivalence relation, for which countably many equivalence classes have cardinal $2$, and all the others have cardinal $1$.
We write $\widehat{\partial} \mathbf{T}^g=\partial \mathbf{T}^g / \sim$, i.e. we identify two geodesic rays if they are the left and right boundaries of the same strip. We also write $\gamma \to \widehat{\gamma}$ for the canonical projection from $\partial \mathbf{T}^g$ to $\widehat{\partial} \mathbf{T}^g$. If $S$ is one the copies of $S^0$ or $S^1$ appearing in the strip decomposition, then the two geodesics bounding $S$ correspond to the same point of $\widehat{\partial} \mathbf{T}^g$, that we denote by $\widehat{\gamma}_{S}$.

Our goal is now to define a topology on $\T \cup \widehat{\partial} \mathbf{T}^g$. It should be possible to define it by an explicit distance, but such a distance would be tedious to write down, so we prefer to give an "abstract" construction. Let $S$ and $S'$ be two distinct strips appearing in the strip decomposition of $\T$ (cf. left part of Figure \ref{TreeAndSlices}). Consider the smallest $r$ such that $S$ and $S'$ both intersect $B_r^{\bullet} \left( \T \right)$. Then $\T \backslash \left( B_r^{\bullet} \left( \T \right) \cup S \cup S' \right)$ has two connected components, that we denote by $\left( S, S' \right)$ and $\left( S', S \right)$ (the vertices on the geodesics bounding $S$ and $S'$ do not belong to $\left( S, S' \right)$ and $\left( S', S \right)$). We also write 
\[\partial_g \left( S, S' \right) =\{ \widehat{\gamma} | \mbox{ $\gamma$ is a ray of $\mathbf{T}^g$ such that $\gamma(i) \in \left( S, S' \right)$ for $i$ large enough}\}\]
We define $\partial_g \left( S', S \right)$ similarly. Note that $\partial_g \left( S, S' \right)$ and $\partial_g \left( S', S \right)$ are disjoint, and their union is $\widehat{\partial} \mathbf{T}^g \backslash \{\widehat{\gamma}_{S}, \widehat{\gamma}_{S'}\}$.

\begin{defn}\label{geodesic_compactification}
The \textit{geodesic compactification} of $\T$ is the set $\T \cup \widehat{\partial} \mathbf{T}^g$, equipped with the topology generated by
\begin{itemize}
\item
the singletons $\{x\}$, where $x$ is a vertex of $\T$,
\item
the sets $\left( S, S' \right) \cup \partial_g \left( S, S' \right)$, where $S$ and $S'$ are two distinct strips appearing in the strip decomposition of $\T$.
\end{itemize}
\end{defn}

This topology is separated (if $\widehat{\gamma}_1 \ne \widehat{\gamma}_2$, then there are two strips separating $\gamma_1$ and $\gamma_2$) and has a countable basis, so it is induced by a distance. Moreover, any open set of our basis intersects $\T$, so $\T$ is dense in $\T \cup \widehat{\partial} \mathbf{T}^g$. The end of this subsection is devoted to the proof of two very intuitive topological properties of the geodesic compactification. The second one states that the boundary $\widehat{\partial} \mathbf{T}^g$ is homeomorphic to the circle, which is natural since this is also the standard topology on the space of ends quotiented by our equivalence relation.

\begin{lem} \label{compactness}
The space $\T \cup \widehat{\partial} \mathbf{T}^g$ is compact.
\end{lem}

\begin{proof}
Let $(x_n)$ be a sequence with values in $\T \cup \widehat{\partial} \mathbf{T}^g$. We first assume $x_n \in \T$ for every $n$. We may assume $d(\rho, x_n) \to +\infty$. If there is a strip $S$ that contains infinitely many $x_n$, then $\widehat{\gamma}_{S}$ is a subsequential limit of $(x_n)$. We now assume it is not the case. We recall that for every vertex $v$ of $\mathbf{T}^g$, the slice $\S[v]$ is the part of $\T$ lying between the leftmost and the rightmost rays passing through $v$, above $v$ (see Figure \ref{TreeAndSlices}).
We will construct a ray $\gamma$ of $\mathbf{T}^g$ step by step, in such a way that for every $k \geq 0$, there are infinitely many points $x_n$ in $\S [\gamma(k)]$.

Assume we have already built $\gamma(0), \dots, \gamma(k)$. If $\gamma(k)$ has only one child in $\mathbf{T}^g$, then $\gamma(k+1)$ must be this child. If $\gamma(k)$ has $d \geq 2$ children, we call them $y_1, \dots, y_d$. Then $\S [\gamma(k)]$ is the union of the slices $\S[y_i]$ for $1 \leq i \leq d$ and of the $d-1$ strips whose lowest vertex is $\gamma(k)$. We know that $\S [\gamma(k)]$ contains infinitely many of the vertices $x_n$, but the $d-1$ strips contain finitely many of them. Therefore, there is an index $1 \leq i_0 \leq d$ such that $\S[y_{i_0}]$ contains infinitely many of them. We choose $\gamma(k+1)=y_{i_0}$.
We can check that the class $\widehat{\gamma}$ of the geodesic we built is a subsequential limit of $(x_n)$, which concludes the case where $x_n \in \T$ for every $n$.

Finally, let $\delta$ be a distance on $\T \cup \widehat{\partial} \mathbf{T}^g$ that generates its topology, and let $(x_n)$ be any sequence in $\T \cup \widehat{\partial} \mathbf{T}^g$. If $x_n \in \widehat{\partial} \mathbf{T}^g$, let $y_n \in \T$ be such that $\delta(x_n, y_n) \leq \frac{1}{n}$ (it exists by density). If $x_n \in \T$, we take $y_n=x_n$. By the first case $(y_n)$ has a subsequential limit, so $(x_n)$ also has one.
\end{proof}

\begin{lem}
The boundary $\widehat{\partial} \mathbf{T}^g$ is homeomorphic to the circle.
\end{lem}

\begin{proof}
We build an explicit homeomorphism. Consider a ray $\gamma$ of $\mathbf{T}^g$. For every $i \geq 0$, let $c_{\gamma}(i)$ be the number of children of $\gamma(i)$ in $\mathbf{T}^g$. We denote these children by $x_0, \dots, x_{c_{\gamma}(i)-1}$ and we denote by $j_{\gamma}(i)$ the index $j \in \{0,1,\dots,c_{\gamma}(i)-1\}$ such that $\gamma(i+1)=x_j$. We also define
\[
\begin{array}{rcl}
\Psi :  \partial \mathbf{T}^g & \longrightarrow & \R / \Z\\
 \gamma & \longrightarrow & \sum_{k \geq 0} \frac{j_{\gamma}(k)}{\prod_{i=0}^k c_{\gamma}(i)} \pmod{1}.
\end{array}
\]
If $\gamma_{\ell}$ and $\gamma_r$ are the left and right boundaries of a copy of $S^0$, then there is $i_0$ such that:
\begin{itemize}
\item
for $i<i_0$, we have $c_{\gamma_r}(i)=c_{\gamma_{\ell}}(i)$ and $j_{\gamma_r}(i)=j_{\gamma_{\ell}}(i)$,
\item
we have $c_{\gamma_r}(i_0)=c_{\gamma_{\ell}}(i_0)$ and $j_{\gamma_r}(i_0)=j_{\gamma_{\ell}}(i_0)+1$,
\item
for $i>i_0$, we have $j_{\gamma_{\ell}}(i)=c_{\gamma_{\ell}}(i)-1$ and $j_{\gamma_r}(i)=0$,
\end{itemize}
which implies $\Psi(\gamma_{\ell})=\Psi(\gamma_r)$. Moreover, if $\gamma_{\ell}$ and $\gamma_r$ are respectively the leftmost and rightmost rays of $\mathbf{T}^g$, then $\Psi(\gamma_{\ell})=\Psi(\gamma_r)$ (the sum is equal to $0$ for $\gamma_{\ell}$ and to $1$ for $\gamma_r$). Hence, we have $\Psi(\gamma)=\Psi(\gamma')$ as soon as $\gamma \sim \gamma'$, so $\Psi$ defines an application from $\widehat{\partial} \mathbf{T}^g$ to $\R / \Z$. The verification that this application is a homeomorphism is easy and left to the reader.
\end{proof}

\subsection{Proof of Theorem \ref{thm2_Poisson}}

We first argue why the second point of Theorem \ref{thm2_Poisson} is an easy consequence of the first one. Assume the first point is proved. Since $\widehat{\partial} \mathbf{T}^g$ is homeomorphic to the circle, we can embed $\T \cup \widehat{\partial} \mathbf{T}^g$ in the unit disk $\mathbb{D}$ in such a way that $\widehat{\partial} \mathbf{T}^g$ is sent to $\partial \mathbb{D}$ (we do not describe the embedding explicitly). In this embedding, the simple random walk converges to a point of $\partial \mathbb{D}$ and the law of the limit point is a.s. non-atomic. Therefore, by Theorem 1.3 of \cite{HP17}, $\widehat{\partial} \mathbf{T}^g$ is a realization of the Poisson boundary of $\T$.

The idea of the proof of the first point is to first show that two independent random walks are quite well separated in terms of $\widehat{\partial} \mathbf{T}^g$. Let $X^1$ and $X^2$ be two independent simple random walks started from $\rho$. By Proposition 11 of \cite{CurPSHIT} we know that $\T$ is transient and does not have the intersection property (although \cite{CurPSHIT} deals with type-II triangulations, all the arguments of the proof still hold in our case). Hence, the complement of $\{ X^1_n | n \in \N \} \cup \{X^2_n | n \in \N\}$ has two infinite connected components with infinite boundaries. We denote them by $\left[ X^1, X^2 \right]$ and $\left[ X^2, X^1 \right]$. By point (ii) of Lemma \ref{halfplane_peeling} applied to the peeling along $X_1$ and $X_2$ (see also Section 3.2 of \cite{CurPSHIT} for a similar argument), the halfplanar triangulations $\left[ X^1, X^2 \right]$ and $\left[ X^2, X^1 \right]$ are independent copies of $\H=\H_{\lambda}$.

Therefore, we will need to study geodesics in halfplanar triangulations.
\begin{defn}
Let $H$ be a halfplanar triangulation and $\partial H$ its boundary. An \emph{infinite geodesic away from the boundary} is a sequence $(\gamma(n))_{n \geq 0}$ of vertices of $H$ such that for any $n \geq 0$:
\begin{itemize}
\item[(i)]
the vertices $\gamma(n)$ and $\gamma(n+1)$ are neighbours,
\item[(ii)]
we have $d(\gamma(n), \partial H)=n$.
\end{itemize}
\end{defn}

\begin{rem}
Contrary to geodesics away from a point, the existence of such geodesics is not obvious. For example, they do not exist in the UIHPT $\H_{\lambda_c}$.
\end{rem}

\begin{lem}\label{halfplane}
Almost surely, there is a geodesic away from the boundary in $\mathbb{H}_{\lambda}$.
\end{lem}

\begin{proof}
We write $P_r=|\partial B_r^{\bullet}(\T)|$ and $L_r=|\partial B_r^{\bullet}(\T) \cap \mathbf{T}^g|$. We recall that $\mathbf{T}^g$ is a Galton--Watson tree with offspring distribution $\mu$ given by Theorem \ref{thm1_GW}. In particular, we have $\sum_{i \geq 0} i \mu(i)=m_{\lambda}^{-1}$ and $\sum_{i \geq 0} (i \log i)  \mu(i)<+\infty$, so by the Kesten--Stigum Theorem $m_{\lambda}^r L_r$ converges a.s. to a positive random variable. On the other hand, as in Section 2 of \cite{CurPSHIT}, it holds that $m_{\lambda}^r P_r$ converges a.s. to a positive random variable (\cite{CurPSHIT} only deals with type-II triangulations but all the 	arguments of the proof still work in the type-I case). Hence, there is a constant $c>0$ such that, for $r$ large enough,
\[ \mathbb{P} \left( L_r \geq c P_r \right) \geq \frac{1}{2}.\]
Therefore, if $z_r$ is a random vertex chosen uniformly on $\partial B_r^{\bullet}(\T)$, for $r$ large enough we have $\mathbb{P} \left( z_r \in \mathbf{T}^g \right) \geq \frac{c}{2}$. Hence, for any $s>0$, with probability at least $\frac{c}{2}$, there is a point $x \in \partial B_{r+s}^{\bullet}(\T)$ at distance exactly $s$ from $z_r$.

But $\mathbb{H}$ is the local limit as $r \to +\infty$ of $\T \backslash B_r^{\bullet}(\T)$ (rooted at a uniform edge on its boundary), so if $\rho$ denotes the root vertex of $\mathbb{H}$, for any $s \geq 0$, we have
\[ \mathbb{P} \left( \mbox{there is $x \in \mathbb{H}$ such that $d(x, \rho)=d(x, \partial \mathbb{H})=s$} \right) \geq \frac{c}{2}.\]
This event is nonincreasing in $s$, so with probability at least $\frac{c}{2}$ it occurs for every $s$. By a compactness argument, with probability at least $\frac{c}{2}$, there is an infinite geodesic away from the boundary $\gamma$ with $\gamma(0)=\rho$. Finally, we claim that $\H$ is invariant under root translation, and that the root translation is ergodic, which is enough to conclude. Indeed, by local limit, the invariance is a consequence of the invariance of $\T^p$ for every $p$ under re-rooting along the boundary. The ergodicity is proved in the type-II case in \cite{AR13} (this is Proposition 1.3), and the proof adapts well here.
\end{proof}

We can now show that $X^1$ and $X^2$ are a.s. separated by an infinite leftmost geodesic.

\begin{lem}\label{gamma-X-gamma-X}
Almost surely, there is a ray $\widehat{\gamma}$ of $\mathbf{T}^g$ such that for $n$ large enough, we have $\widehat{\gamma}(n) \in \left[ X^1, X^2 \right]$, and the same is true for $\left[ X^2, X^1 \right]$.
\end{lem}

\begin{proof}
In this proof, we will write $H=\left[ X^1, X^2 \right]$ to avoid too heavy notations. Let $\gamma$ be an infinite geodesic away from the boundary in $H$, which exists by Lemma \ref{halfplane}. For $n \geq 0$, let $\left( \gamma_n(i) \right)_{0 \leq i \leq d(\rho, \gamma(n))}$ be the leftmost geodesic from $\rho$ to $\gamma(n)$ (cf. Figure \ref{figure_geodesic_hp}). For $n \geq i \geq d(\rho, \gamma(0))$, we have $i \leq n =d(\gamma(n), \partial H) \leq d(\gamma(n), \rho)$ so $\gamma_n(i)$ is well defined, and
\[\begin{array}{lll}
d(\gamma(n), \gamma_n(i))&= d(\rho, \gamma(n))-i & \mbox{(since $\gamma$ is a geodesic)}\\
& \leq n+d(\rho, \gamma(0))-i & \mbox{(by the triangular inequality)}\\
& \leq n &\\
&= d(\gamma(n), \partial H), &
\end{array}\]
so $\gamma_n(i) \in H$. But by a compactness argument, there is an infinite path $\widetilde{\gamma}$ in $\T$ such that for any $i$, there are infinitely many $n$ such that $\widetilde{\gamma}(i)=\gamma_n(i)$. It easy to check that $\widetilde{\gamma}$ is an infinite leftmost geodesic in $\T$. Moreover, since $\gamma_n(i) \in H$ for $n \geq i \geq d(\rho, \gamma(0))$, we have $\widetilde{\gamma}(i) \in H$ for $i$ large enough, so there is an infinite leftmost geodesic of $\T$ that lies in $H$ eventually.
\end{proof}

\begin{figure}
\begin{center}
\begin{tikzpicture}
\draw(0,0)--(6,0);
\draw[thick, red] (3,0) to[out=90,in=300] (2.5,1.5);
\draw[thick, red] (2.5,1.5) to[out=120,in=270] (3.5,4.5);
\draw[thick, blue] (2.5,-1) to[bend left] (2.5,1.5);
\draw(3,0) node{};
\draw(2.5,-1) node{};
\draw(2.5,1.5) node{};

\draw(2.8,-1) node[texte]{$\rho$};
\draw(5.5,0.3) node[texte]{$\partial H$};
\draw[red] (3.4,0.3) node[texte]{$\gamma(0)$};
\draw[red] (3,1.6) node[texte]{$\gamma(n)$};
\draw[red] (3.5,3.5) node[texte]{$\gamma$};
\draw[blue] (1.9,0.5) node[texte]{$\gamma_n$};
\end{tikzpicture}
\end{center}
\caption{From an infinite geodesic away from the boundary in $H$, we can build an infinite leftmost geodesic in $\T$, that lies in $H$ eventually.} \label{figure_geodesic_hp}
\end{figure}
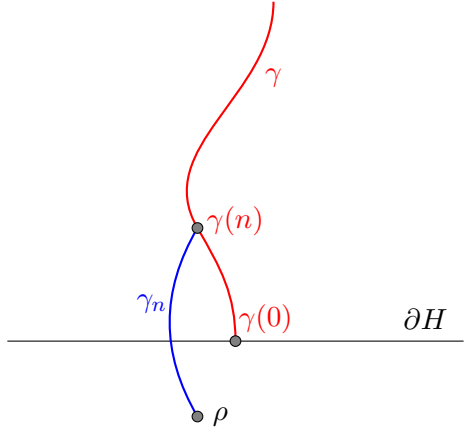

\begin{proof}[Proof of Theorem \ref{thm2_Poisson}]
We can now prove the convergence of the simple random walk to a point of $\widehat{\partial} \mathbf{T}^g$.
If $X$ is a simple random walk, we write $\mathrm{Acc}(X)$ for the set of accumulation points of $X$ on $\widehat{\partial} \mathbf{T}^g$. By Lemma \ref{compactness}, it is enough to prove that $\mathrm{Acc}(X)$ is reduced to a point.
We first claim that $\mathrm{Acc}(X)$ is a circle arc of $\widehat{\partial} \mathbf{T}^g$. Indeed, assume $\widehat{\gamma}_1 \ne \widehat{\gamma}_2$ are two points of $\mathrm{Acc}(X)$.
Then $\widehat{\partial} \mathbf{T}^g \backslash \{ \widehat{\gamma}_1, \widehat{\gamma}_2 \}$ has two connected components, that we denote by $\left( \widehat{\gamma}_1, \widehat{\gamma}_2 \right)$ and $\left( \widehat{\gamma}_2, \widehat{\gamma}_1 \right)$.
To oscillate between $\widehat{\gamma}_1$ and $\widehat{\gamma}_2$, the walk $X$ must intersect infinitely many times either all the $\gamma$ such that $\widehat{\gamma} \in \left( \widehat{\gamma}_1, \widehat{\gamma}_2 \right)$ or all the $\gamma$ such that $\widehat{\gamma} \in \left( \widehat{\gamma}_2, \widehat{\gamma}_1 \right)$ (see Figure \ref{acc_is_an_arc}). In both cases, $\mathrm{Acc}(X)$ contains one of the two arcs from $\widehat{\gamma}_1$ to $\widehat{\gamma}_2$. Hence, $\mathrm{Acc}(X)$ is closed and connected, so it is a circle arc.

\begin{figure}
\begin{center}
\begin{tikzpicture}

\draw[red, thick] (0,0)--(0,0.5);
\draw[red, thick] (0,0.5)--(60:1);
\draw[red, thick] (60:1)--(30:3);
\draw[red, thick] (60:1)--(100:3);
\draw[red, thick] (0,0.5)--(135:1);
\draw[red, thick] (135:1)--(150:3);
\draw[red, thick] (135:1)--(240:3);

\draw[orange, very thick] (0,0) to[out=270,in=270] (45:1); 
\draw[orange, very thick] (45:1) to[out=90,in=270] (150:1.5); 
\draw[orange, very thick] (150:1.5) to[out=90,in=180] (30:2);
\draw[orange, very thick] (30:2) to[out=0,in=270] (30:2.3);
\draw[orange, very thick] (30:2.3) to[out=90,in=0] (90:2);
\draw[orange, very thick] (90:2) to[out=180,in=0] (150:2.5);
\draw[orange, very thick] (150:2.5) to[out=180,in=270] (150:2.8);
\draw[orange, very thick] (150:2.8) to[out=90,in=210] (120:2.7);
\draw[orange, very thick] (120:2.7) to[out=30,in=120] (30:2.8);

\draw (0,0) circle(3cm);
\draw[blue, thick] (30:3) arc (30:150:3);
\draw[darkgreen, thick] (150:3) arc (150:390:3);

\draw (0,0) node{};
\draw (0,-0.3) node[texte]{$\rho$};
\draw[red] (30:3.3) node[texte]{$\gamma_1$};
\draw[red] (150:3.3) node[texte]{$\gamma_2$};
\draw[red] (240:3.3) node[texte]{$\gamma'$};
\draw[red] (100:3.3) node[texte]{$\gamma$};
\draw[orange] (30:1) node[texte]{$X$};
\draw[blue] (75:3.4) node[texte]{$\left( \widehat{\gamma}_1, \widehat{\gamma}_2 \right)$};
\draw[darkgreen] (310:3.6) node[texte]{$\left( \widehat{\gamma}_2, \widehat{\gamma}_1 \right)$};

\end{tikzpicture}
\end{center}
\caption{Proof that $\mathrm{Acc}(X)$ is an arc circle: assume $X$ oscillates between $\gamma_1$ and $\gamma_2$ and there is $\gamma' \in \left( \widehat{\gamma}_2, \widehat{\gamma}_1 \right)$ that $X$ intersects only finitely many times. Then $X$ intersects infinitely many times every $\gamma$ with $\widehat{\gamma} \in \left( \widehat{\gamma}_1, \widehat{\gamma}_2 \right)$.} \label{acc_is_an_arc}
\end{figure}
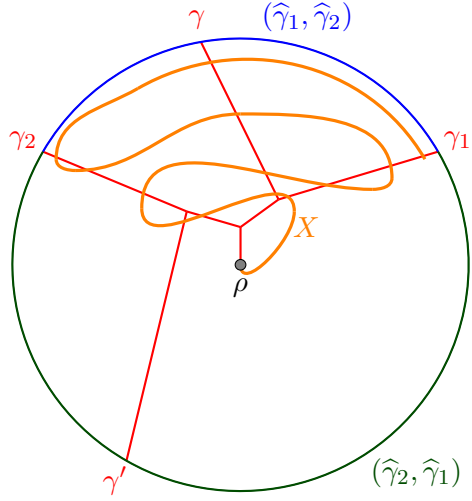

Now let $\nu$ be a probability measure with no atom and full support on $\widehat{\partial} \mathbf{T}^g$ (for example, one can consider the exit measure of the nonbacktracking random walk on $\mathbf{T}^g$). Either $\mathrm{Acc}(X)$ is a singleton, or it has positive measure, so it is enough to show that $\P \left( \nu \left( \mathrm{Acc}(X) \right)>0 \right)=0$. By Lemma \ref{gamma-X-gamma-X}, we know that if $X^1$ and $X^2$ are two independent simple random walks started from $\rho$, then there are two rays of $\mathbf{T}^g$ lying respectively in $\left[ X_1, X_2 \right]$ and $\left[ X_2, X_1 \right]$ eventually, so they separate $X^1$ and $X^2$. Therefore, $\mathrm{Acc}(X^1) \cap \mathrm{Acc}(X^2)$ contains at most two points, so $\nu \left( \mathrm{Acc}(X^1) \cap \mathrm{Acc}(X^2) \right)=0$ a.s.. Let $X^i$ for $i \in \N$ be infinitely many independent simple random walks started from $\rho$. We have
\[ \sum_{i \geq 0} \nu \left( \mathrm{Acc}(X^i) \right) = \nu \left( \bigcup_{i \geq 0} \mathrm{Acc}(X^i) \right) \leq 1. \]
But the $\nu \left( \mathrm{Acc}(X^i) \right)$ are i.i.d. (conditionally on $\T$), so they must be $0$ a.s.. Therefore,  $\mathrm{Acc}(X)$ cannot have positive measure so it is a.s. a point. 

Hence, the simple random walk a.s. converges to a point of $\widehat{\partial} \mathbf{T}^g$, so it defines an exit measure on $\widehat{\partial} \mathbf{T}^g$, that we denote by $\nu_{\partial}$.
We now prove that $\nu_{\partial}$ is nonatomic. Once again, our main tool is Lemma \ref{gamma-X-gamma-X}. Assume $\nu_{\partial}$ has an atom with positive probability, and let $(X^i)_{1 \leq i \leq 4}$ be four independent SRW started from $\rho$. For every $1 \leq i \leq 4$, let $X^i_{\infty}$ be the limit of $X^i$ in $\widehat{\partial} \mathbf{T}^g$. Then $\mathbb{P} \left( X^1_{\infty}=X^2_{\infty}=X^3_{\infty}=X^4_{\infty} \right)>0$. If this happens, we can assume (up to a factor $\frac{1}{24}$) that they lie in clockwise order, and that $\widehat{\partial} \mathbf{T}^g \backslash \{ X^1_{\infty} \}$ lies in the part between $X^4$ and $X^1$. By Lemma \ref{gamma-X-gamma-X}, there are at least one ray of $\mathbf{T}^g$ between $X^1$ and $X^2$, one between $X^2$ and $X^3$ and one between $X^3$ and $X^4$ (cf. Figure \ref{figure_nonatomic}). Hence, there are at least three rays of $\mathbf{T}^g$ in the part between $X^1$ and $X^4$ that does not contain $\widehat{\partial} \mathbf{T}^g \backslash \{ X^1_{\infty} \}$. In particular, two of them are not equivalent for $\sim$ (we recall that the equivalence classes have cardinal at most $2$). Hence, the part lying between $X^1$ and $X^4$ containing $X^2$ and $X^3$ (the left part on Figure \ref{figure_nonatomic}) also contains a slice of the form $\S[y]$ (see Figure \ref{TreeAndSlices}), so $X^1_{\infty} \ne X^4_{\infty}$. We get a contradiction.

\begin{figure}
\begin{center}
\begin{tikzpicture}
\fill[color=green!20] (115:3) to[out=310, in=100] (0,0) to[out=140, in=290] (125:3);
\fill[color=green!20] (155:3) to[out=350, in=140] (0,0) to[out=180, in=330] (165:3);
\fill[color=green!20] (195:3) to[out=30, in=180] (0,0) to[out=220, in=10] (205:3);
\fill[color=green!20] (235:3) to[out=70, in=220] (0,0) to[out=260, in=50] (245:3);

\draw (0,0) to[out=100, in=310] (115:3);
\draw (0,0) to[out=140, in=290] (125:3);
\draw (0,0) to[out=140, in=350] (155:3);
\draw (0,0) to[out=180, in=330] (165:3);
\draw (0,0) to[out=180, in=30] (195:3);
\draw (0,0) to[out=220, in=10] (205:3);
\draw (0,0) to[out=220, in=70] (235:3);
\draw (0,0) to[out=260, in=50] (245:3);

\draw[thick, red] (0,0) -- (160:0.5);
\draw[thick, red] (160:0.5) to[bend right=15] (140:3);
\draw[thick, red] (160:0.5) -- (180:1);
\draw[thick, red] (180:3) -- (180:1);
\draw[thick, red] (180:1) to[bend left] (220:3);

\draw (0,0) node{};
\draw (0.3,0) node[texte]{$\rho$};
\draw (-2,0) node{};
\draw (-2.2,0.2) node[texte]{$y$};
\draw (240:2) node[texte]{$X^1$};
\draw (200:2) node[texte]{$X^2$};
\draw (160:2) node[texte]{$X^3$};
\draw (120:2) node[texte]{$X^4$};
\end{tikzpicture}
\end{center}
\caption{Sketch of the proof that $\nu_{\partial}$ is nonatomic: if $X^1, \dots, X^4$ are four independent simple random walks, then there are three infinite leftmost geodesics (in red) between $X^1$ and $X^4$, so there is a slice $\S[y]$ there. The hulls of the four random walks are in green.}\label{figure_nonatomic}
\end{figure}
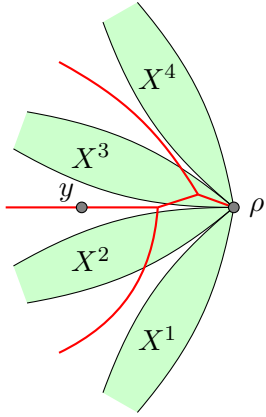

We finally show that $\nu_{\partial}$ has full support. Since $\nu_{\partial}$ is nonatomic, we have $\nu_{\partial} \left( \{ \widehat{\gamma}_{S^1} \} \right)=0$ a.s.. Hence, a.s., we have $X_n \in \S[\rho]=\S$ for $n$ large enough. Therefore, we also have
\[\P \left( \forall k \geq 0, X_k \in \S[\rho] | X_0=\rho \right)>0 \qquad \mathrm{a.s.}\]
Now fix $r>0$, condition on $B_r^{\bullet} \left( \T \right)$ and take $x \in \mathbf{T}^g$ such that $d(\rho,x)=r$. Since $\S[x]$ has the same distribution as $\S[\rho]$ (see Figure \ref{TreeAndSlices}), we have
\[\P \left( \forall k \geq 0, X_k \in \S[x] | X_0=x \right)>0 \qquad \mathrm{a.s.}\]
Hence, we have $\P \left( X_k \in \S[x] \mbox{ for $k$ large enough} | X_0=\rho \right)>0$ a.s. If this occurs, then $X_{\infty}$ is of the form $\widehat{\gamma}$, where $\gamma(k) \in \S[x]$ for $k$ large enough. Therefore, we have
\[\nu_{\partial} \left( \left\{ \widehat{\gamma} | \mbox{$\gamma$ lie in $\S[x]$ eventually} \right\} \right)>0 \qquad \mathrm{a.s.}\]
Almost surely, this holds for every $x \in \mathbf{T}^g$, which is enough to ensure that $\nu_{\partial}$ has full support. This ends the proof of Theorem \ref{thm2_Poisson}.
\end{proof}

\begin{rem}
We end this section with a remark about the Gromov boundary \cite{Gro81}, which is another natural notion of boundary for an infinite graph $G=(V,E)$. Let $\mathcal{C}(G)$ be the space of functions $f : V \to \R$ equipped with the product topology. We say that two functions of $\mathcal{C}(G)$ are equivalent if they are equal up to an additive constant. We quotient $\mathcal{C}(G)$ by this equivalence relation to obtain the quotient space $\mathcal{C}(G) / \mathbb{R}$. If $x \in V$, we define $f_x \in \mathcal{C}(G) / \mathbb{R}$ by $f_x (y)=d_G(x,y)$ for any $y \in V$. The \emph{Gromov compactification} $\widehat{G}$ of $G$ is the closure of $\{ f_x | x \in V\}$ in $\mathcal{C}(G) / \mathbb{R}$ and the \emph{Gromov boundary} $\partial_{Gr} G$ of $G$ is the set $\widehat{G} \backslash \{ f_x | x\in V\}$.

It is easy to show that for any geodesic ray $\gamma$, the sequence $f_{\gamma(n)}$ converges in $\mathcal{C}(G) / \mathbb{R}$, so it defines a point $f_{\gamma} \in \partial_{Gr} \T$. A natural question is to ask whether there is a natural correspondence between $\widehat{\partial} \mathbf{T}^g$ and $\partial_{Gr} \T$. The answer is no.

To prove it, we show that if $\gamma_1$ and $\gamma_2$ are respectively the left and right boundary of the same strip, then $f_{\gamma_1} \ne f_{\gamma_2}$. Indeed, let $n$ be such that $\gamma_1(n) \notin \gamma_2$. We have $f_{\gamma_1}(\gamma_1(n))-f_{\gamma_1}(\rho)=-n$ by definition of $f_{\gamma_1}$. But if we had $f_{\gamma_2}(\gamma_1(n))-f_{\gamma_2}(\rho)=-n$, this would mean that there is $m>n$ such that $d(\gamma_2(m), \gamma_1(n))-d(\gamma_2(m), \rho)=-n$, i.e. $d(\gamma_2(m), \gamma_1(n))=m-n$. Take such an $m$ minimal. Then by concatenating $\gamma_1$ from $\rho$ to $\gamma_1(n)$ and a geodesic from $\gamma_1(n)$ to $\gamma_2(m)$, we obtain a geodesic from $\rho$ to $\gamma_2(m)$ that lies strictly to the left of $\gamma_2$. This contradicts the fact that $\gamma_2$ is a leftmost geodesic in $\T$. This suggests that $\partial_{Gr} \T$ should not be homeomorphic to the circle, but rather to a Cantor set.
\end{rem}

\section{The tree of infinite geodesics in the hyperbolic Brownian plane}

\subsection{The tree \texorpdfstring{$\mathbf{T}^g(\hp)$}{Tg(hp)}}

The goal of this section is to take the scaling limit of Theorem \ref{thm1_GW} and to prove Theorem \ref{thm3_hbp}. For all this section, we fix a sequence $(\lambda_n)$ of numbers in $(0, \lambda_c]$ such that $\lambda_n = \lambda_c \left( 1-\frac{2}{3n^4} \right) + o \left( \frac{1}{n^4} \right)$. We know, by the main result of \cite{B16}, that $\frac{1}{n} \T_{\lambda_n}$ converges for the local Gromov--Hausdorff distance to $\hp$. Therefore, it seems reasonable that $\mathbf{T}^g(\hp)$ should be the scaling limit of the trees $\mathbf{T}^g_{\lambda_n}$. This scaling limit is easy to describe. We recall that $\mathbf{B}$ is the infinite tree in which every vertex has exactly two children, except the root which has one. For $\alpha>0$, we denote by $\mathbf{Y}_{\alpha}$ the Yule tree of parameter $\alpha$, i.e. the tree $\mathbf{B}$ in which the lengths of the edges are i.i.d. exponential variables with parameter $\alpha$.

\begin{lem}\label{GHtree}
The trees $\frac{1}{n} \mathbf{T}^g_{\lambda_n}$ converge for the local Gromov--Hausdorff distance to $\mathbf{Y}_{2\sqrt{2}}$.
\end{lem}

\begin{idproof}
We recall that $\mathbf{T}^g_{\lambda_n}$ is a Galton--Watson tree with offspring distribution $\mu_{\lambda_n}$, where $\mu_{\lambda_n} (k)=m_{\lambda_n} \left( 1-m_{\lambda_n} \right)^{k-1}$ for every $k \geq 1$, and $m_{\lambda}$ is explicitly given by \eqref{eqn_m_lambda}. We can compute $m_{\lambda_n}=1-\frac{2\sqrt{2}}{n}+O \left( \frac{1}{n^2} \right)$. Let $\widetilde{\mu}_{\lambda}$ be the distribution defined by $\widetilde{\mu}_{\lambda}(1)=m_{\lambda}$ and $\widetilde{\mu}_{\lambda}(2)=1-m_{\lambda}$, and let $\mathbf{\widetilde{T}}^g_{\lambda}$ be a Galton--Watson tree with offspring distribution $\widetilde{\mu}_{\lambda}$. Then $\mathbf{\widetilde{T}}^g_{\lambda}$ is a copy of $\mathbf{B}$ where the length of each edge is a geometric variable of parameter $1-m_{\lambda}$, so $\frac{1}{n} \mathbf{\widetilde{T}}^g_{\lambda_n}$ converges to $\mathbf{Y}_{2\sqrt{2}}$. Moreover, $\mathbf{T}^g_{\lambda}$ can be obtained by adding children to some of the vertices of $\mathbf{\widetilde{T}}^g_{\lambda}$ with two children. Since $\mu_{\lambda_n} \left( [3,+\infty[ \right)=O \left( \frac{1}{n^2} \right)$, the probability to affect a vertex is of order $\frac{1}{n^2}$. The number of vertices at height of order $n$ in $\mathbf{T}^g_{\lambda_n}$ is of order $n$, so the difference bewteen $\mathbf{T}^g_{\lambda_n}$ and $\mathbf{\widetilde{T}}^g_{\lambda_n}$ does not affect the scaling limit. 
\end{idproof}

\medskip

However, taking the scaling limit of $\mathbf{T}^g_{\lambda_n}$ is not enough to obtain a description of infinite geodesics in $\hp$. Three different kinds of problems could prevent this:
\begin{itemize}
\item[(i)]
it is not completely clear that the infinite geodesics in $\hp$ form a tree,
\item[(ii)]
two different discrete leftmost geodesics might be too close and collide in the scaling limit,
\item[(iii)]
discrete paths that are not infinite leftmost geodesics might become infinite geodesics in the scaling limit.
\end{itemize}
We take care of item (i) right now, while the goal of Lemmas \ref{nocollision} and \ref{No_new_geodesics} will be to rule out items (ii) and (iii).

The fact that the infinite geodesics of $\hp$ indeed form a tree is a quite strong result, that follows from the confluence of geodesics properties in the Brownian map. More precisely, it will be a consequence of \cite[Proposition~28]{AKM15}, which we recall here. We write $m_{\infty}$ for the Brownian map, and denote its root by $\rho$.

\begin{prop}\cite{AKM15}\label{no-o-}
Almost surely, for any $x \in m_{\infty}$, if $\gamma$ and $\gamma'$ are two geodesics from $\rho$ to $x$ that coincide on a neighbourhood of $x$, then $\gamma=\gamma'$.
\end{prop}

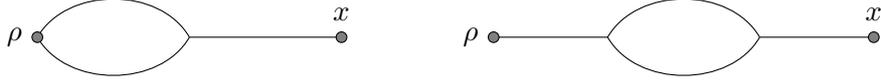
\begin{figure}
\begin{center}
\begin{tikzpicture}

\draw (0,0) to[bend left=60] (2,0);
\draw (0,0) to[bend right=60] (2,0);
\draw (2,0)--(4,0);
\draw (0,0)node{};
\draw (-0.3,0) node[texte]{$\rho$};
\draw (4,0)node{};
\draw (4,0.3) node[texte]{$x$};

\begin{scope}[shift={(6,0)}]
\draw (0,0)--(1.5,0);
\draw (1.5,0) to[bend left=60] (3.5,0);
\draw (1.5,0) to[bend right=60] (3.5,0);
\draw (3.5,0)--(5,0);
\draw (0,0)node{};
\draw (-0.3,0) node[texte]{$\rho$};
\draw (5,0)node{};
\draw (5,0.3) node[texte]{$x$};
\end{scope}

\end{tikzpicture}
\end{center}
\caption{Proposition \ref{no-o-} : almost surely, for any $x$, none of these two cases occurs.}
\end{figure}

We write $\mathbf{T}^g(\hp)$ for the set of those points of $\hp$ that lie on an infinite geodesic of $\hp$ started from $\rho$. By local isometry of the Brownian plane and the Brownian map and scale invariance of the Brownian plane \cite{CLGplane}, Proposition \ref{no-o-} also holds for the Brownian plane. By the absolute continuity relations of \cite{B16}, it also holds for $\hp$. We claim that it implies that for every $x \in \mathbf{T}^g(\hp)$, there is a unique geodesic from $\rho$ to $x$ in $\hp$. Indeed, let $x \in \mathbf{T}^g(\hp)$ and let $\gamma$ be an infinite geodesic of $\hp$ passing through $x$. Let $\gamma'$ be a geodesic from $\rho$ to $x$. Finally, let $y$ be a point on the ray $\gamma$ such that $d(\rho,y)>d(\rho,x)$. The concatenation $\gamma''$ of $\gamma'$ and the part of $\gamma$ between $x$ and $y$ is a geodesic from $\rho$ to $y$ that coincides with $\gamma$ in a neighbourhood of $y$. By Proposition \ref{no-o-}, the path $\gamma''$ must coincide with $\gamma$, so $\gamma'$ must coincide with $\gamma$, which proves our claim.

We equip $\mathbf{T}^g(\hp)$ with its natural tree metric: if $x,y \in \mathbf{T}^g(\hp)$, the intersection of the geodesics from $\rho$ to $x$ and $y$ is compact so there is a unique point $z$ on it that maximizes $d(\rho, z)$. We write $d_{\mathbf{T}^g(\hp)}(x,y)=d_{\hp}(x,z)+d_{\hp}(z,y)$. Equipped with this distance $\mathbf{T}^g(\hp)$ is a real tree. However, it is not obvious that it is locally compact (for example it is not if $\hp$ is replaced by $\R^2$ equipped with the Euclidean norm).

\subsection{Two lemmas about near-critical strips}

We first show that two disjoint geodesics in $\mathbf{T}^g_{\lambda}$ are quite well-separated, which rules out problem (ii). We denote by $\gamma_{\ell}$ and $\gamma_r$ the left and right boundaries of the infinite strip $S^1_{\lambda_n}$.

\begin{lem} \label{nocollision}
Let $b>a>0$ and $\eps>0$. Then there is $\delta>0$ such that, for $n$ large enough, the following holds: 
\[ \mathbb{P} \left( \forall i,j \in [an, bn], d_{S^1_{\lambda_n}} \! (\gamma_{\ell}(i), \gamma_r(j)) \geq \delta n \right) \geq 1-\eps.\]
\end{lem}

Note that this lemma is very similar to Lemma 3.4 of \cite{LG11}. However, the slices considered here are not exactly the same, and the lemma of \cite{LG11} only gives the result with positive probability.

\begin{proof}
The idea of the proof is as follows: assume that two points on $\gamma_r$ and $\gamma_{\ell}$ at height in $[a,b]$ are too close from each other. Then with positive probability $S^1_{\lambda}$ is the only strip of $\T_{\lambda_n}$ that reaches height $b$. In this case, we have a small separating cycle in $\T_{\lambda_n}$, which becomes a pinch point in the scaling limit. This contradicts the homeomorphicity of $\hp$ to the plane (this is Proposition 18 of \cite{B16}, and a consequence of the homeomorphicity of the Brownian map to the sphere \cite{LGP08}).

More precisely, assume the lemma is not true. Then up to extraction, for all $\delta>0$ and for $n$ large enough, we have
\[ \mathbb{P} \left( \exists i,j \in [an, bn], d_{S^1_{\lambda_n}}(\gamma_{\ell}(i), \gamma_r(j)) < \delta n \right) \geq \eps.\]
Note that if this happens, then by the triangle inequality we must have $|i-j|<\delta n$.

On the other hand, by Theorem \ref{thm1_GW}, the probability that the tree $\mathbf{T}^g_{\lambda_n}$ has only one vertex at height $(b+1)n$ is $\mu_{\lambda_n}(1)^{(b+1)n}=m_{\lambda_n}^{(b+1)n} \longrightarrow e^{-2 \sqrt{2} (b+1)}$. Since the tree $\mathbf{T}^g_{\lambda_n}$ is independent of the strip $S^1_{\lambda_n}$, for $n$ large enough, the following event occurs with probability at least $\frac{1}{2} e^{-2\sqrt{2} (b+1)}\eps$:

\begin{center}
$\{$there are $i,j \in [an, bn]$ with $d_{S^1_{\lambda_n}}(\gamma_{\ell}(i), \gamma_r(j)) < \delta n \}$

AND

$\{ S^1_{\lambda_n}$ is the only strip of $\T_{\lambda_n}$ at height $(b+1)n \}$.
\end{center}

If this event occurs, the hull of radius $(b+1)n$ in $\T_{\lambda_n}$ is the hull of radius $(b+1)n$ in $S^1_{\lambda_n}$, where the boundary geodesics $\gamma_{\ell}$ and $\gamma_r$ have been glued together. Hence, by concatenating the geodesic in $S^1_{\lambda_n}$ between $\gamma_{\ell}(i)$ and $\gamma_{r}(j)$ and a portion of $\gamma_{\ell}=\gamma_r$ of length $|i-j|$, we get a cycle of length not greater than $2 \delta n$, at height at least $(a-\delta)n$, and that separates $\rho$ from infinity. In other words, for all $\delta>0$, with probability at least $\frac{1}{2} e^{-2\sqrt{2} (b+1)}\eps$, the following event occurs:

\begin{center}
$\{$there is a point $z \in \T_{\lambda_n}$ with $a-\delta \leq \frac{1}{n} d(\rho, z) \leq b+\delta$ such that for any continuous path $p$ from $\rho$ to infinity, there is a point $y$ of $p$ such that $\frac{1}{n} d(y,z) \leq 2 \delta\}$. \label{separatingcycle}
\end{center}

This last property is closed for the Gromov--Hausdorff topology. To prove it properly, we would need to replace the path to infinity by a path to a point $x$ at distance from $\rho$ large enough so that $x$ cannot lie in $B_{(b+1)n}^{\bullet}(\T_{\lambda_n})$, and then to replace the continuous path by a $\delta$-chain. We omit the details here, see Appendix of \cite{B16} for something very similar.

Hence, by the main Theorem of \cite{B16}, for any $\delta>0$, with probability at least $\frac{1}{2} e^{-2\sqrt{2} (b+1)}\eps$, there is a point $z \in \hp$ with $a-\delta \leq d(\rho, z) \leq b+\delta$ such that any continuous path from $\rho$ to infinity contains a point at distance at most $\delta$ from $z$. Since this event is nondecreasing in $\delta$, with probability at least $\frac{1}{2} e^{-2\sqrt{2} (b+1)}\eps$ it occurs for every $\delta>0$. If it does, let $z_n$ be such a point for $\delta=\frac{1}{n}$, and let $z$ be a subsequential limit of $(z_n)$. Then we have $a \leq d(\rho, z) \leq b$ and every infinite path from $\rho$ to infinity must contain $z$. Hence, there is a single point that separates the origin from infinity in $\hp$, which is impossible by homeomorphicity of $\hp$ to the plane.
\end{proof}

The next lemma shows that for any $x \in \T_{\lambda_n}$, there is a geodesic from $x$ to $\rho$ that coincides with a leftmost infinite geodesic on a quite long distance. Combined with the uniqueness of geodesics between $\rho$ and points of $\mathbf{T}^g(\hp)$, this will rule out problem (iii). We consider a strip $S^0_{\lambda_n}$ and we denote by $\gamma_{\ell}$ and $\gamma_r$ its left and right geodesic boundaries.

\begin{defn}
Let $x \in S^0_{\lambda_n}$ and $a>0$. We say that $x$ is \emph{$a$-close to the boundary} if there is a geodesic $\gamma$ from $\rho$ to $x$ that contains either $\gamma_{\ell}(i)$ for every $0 \leq i \leq a$, or $\gamma_r(i)$ for every $0 \leq i \leq a$.
\end{defn}

\begin{lem} \label{No_new_geodesics}
Let $\eps, r>0$.
\begin{enumerate}
\item[(i)]
There is $C>0$ such that for $n$ large enough
\[ \P \left( \mbox{every $x \in \partial B_{Cn}^{\bullet} \left( S^0_{\lambda_n} \right)$ is $(rn)$-close to the boundary} \right) \geq 1-\eps.\]
\item[(ii)]
There is $C'>0$ such that for $n$ large enough
\[ \P \left( \mbox{every $x \in S^0_{\lambda_n} \backslash B_{C'n} \left( S^0_{\lambda_n} \right)$ is $(rn)$-close to the boundary} \right) \geq 1-2 \eps.\]
\end{enumerate}
\end{lem}

Note that we gave two quite similar versions of the lemma. The version (i) is the most natural to prove, whereas (ii) passes more  easily to the Gromov--Hausdorff limit and is the one we will use later.
\begin{proof}
\begin{enumerate}
\item[(i)]
We recall that $\boldsymbol{\tau}^0_{\lambda_n}$ is the skeleton of $S^0_{\lambda_n}$. We first argue that showing point (i) of the lemma is equivalent to bounding the heights of the trees grafted on the spine of $\boldsymbol{\tau}^0$.
Let $k>0$ (we will precise its value later). We denote by $x_1, \dots, x_{p+1}$ the vertices of $\partial B_{k}^{\bullet} \left( S^0_{\lambda_n} \right)$ that lie on the right of the spine of $\boldsymbol{\tau}^0$, from left to right. For $1 \leq i \leq p$, let also $t_i$ be the subtree of descendants of the edge $\{ x_i, x_{i+1} \}$ in $\boldsymbol{\tau}^0$. For $1 \leq i \leq p$, let also $\gamma_i$ be the leftmost geodesic from $x_i$ to $\rho$. It is clear (see Figure \ref{Skeleton_decomposition_strip}) that the distance between $x_i$ and the point at which $\gamma_i$ and $\gamma_r$ merge is equal to the maximum of the heights of the trees starting between $x_i$ and $\gamma_r$. Hence, we have $\gamma_i(j) \in \gamma_{r}$ as soon as $j$ is greater than the heights of all the trees $t_i, t_{i+1}, \dots, t_p$. Therefore, if we denote by $H^r_{\lambda_n}(k)$ the height of the forest $(t_1, \dots, t_p)$, then for any $i$, there is a geodesic from $x_i$ to $\rho$ that contains $\gamma_r(j)$ for $0 \leq j \leq k-H^r_{\lambda_n}(k)$. We can do the same reasoning for vertices on the left of the spine. We denote by $H^{\ell}_{\lambda_n}(k)$ the height of the forest that is defined similarly on the left of the spine, and write $H_{\lambda_n}(k)=\max \left( H^{\ell}_{\lambda_n}(k), H^r_{\lambda_n}(k) \right)$.
It is then enough to find $C$ such that for $n$ large enough:
\begin{equation} \label{sufficientcond}
\mathbb{P} \left(H_{\lambda_n}(Cn) \leq (C-r)n \right) \geq 1-\eps.
\end{equation}

We write $P_{\lambda_n}(Cn)=|\partial B_{Cn}^{\bullet} \left( S^0_{\lambda_n} \right)|-2$ (this is the number of vertices of $\boldsymbol{\tau}^0$ at height $Cn$ that are not on the spine). Conditionally on $P_{\lambda_n}(Cn)$, all the trees of descendants of the vertices $x_i$ are independent Galton--Watson trees conditioned on extinction before a finite time $\lfloor Cn \rfloor$, so we have
\[ \mathbb{P} \left( H_{\lambda_n}(Cn) \leq (C-r)n | P_{\lambda_n}(Cn)=p \right) = \left( \frac{g_{\lambda_n}^{\circ \lfloor (C-r)n \rfloor}(0)}{g_{\lambda_n}^{\circ \lfloor Cn \rfloor}(0)} \right)^p \geq g_{\lambda_n}^{\circ \lfloor (C-r)n \rfloor}(0)^p.\]
By Lemma \ref{itererg}, we get
\[g_{\lambda_n}^{\circ \lfloor (C-r)n \rfloor }(0)= 1-\frac{2}{\sinh^2 \left( \sqrt{2} (C-r) \right)} \frac{1}{n^2} +o \left(\frac{1}{n^2} \right), \]
so $g_{\lambda_n}^{\circ \lfloor (C-r)n \rfloor}(0) \geq 1-\frac{z}{n^2}$ for $n$ large enough, where $z=\frac{3}{\sinh^2 \left( \sqrt{2} (C-r) \right)}$. Hence, we get
\[ \mathbb{P} \left( H_{\lambda_n}(Cn) \leq (C-r)n \right) \geq \mathbb{E} \left[ \left( 1-\frac{z}{n^2} \right)^{P_{\lambda_n}(Cn)} \right].\]
By using the distribution of $P_{Cn}$ given by Lemma \ref{distribY}, we obtain
\begin{align*}
\mathbb{E} \left[ \left( 1-\frac{z}{n^2} \right)^{P_{\lambda_n}(Cn)} \right] &= \frac{m_{\lambda_n}^{-\lfloor Cn \rfloor}}{\Pi_{\lambda_n} (\theta_{\lambda_n}(0))} \\
& \hspace{1cm} \times \bigg( \Pi_{\lambda_n} \left( \left( 1-\frac{z}{n^2} \right) g_{\lambda_n}^{\circ (\lfloor Cn \rfloor +1)} (0) \right)-\Pi_{\lambda_n} \left( \left( 1-\frac{z}{n^2} \right) g_{\lambda_n}^{\circ \lfloor Cn \rfloor} (0) \right) \bigg)\\
& \geq \frac{1}{\Pi_{\lambda_n} (\theta_{\lambda_n}(0))} \times m_{\lambda_n}^{-\lfloor Cn \rfloor} \times \left( 1-\frac{z}{n^2} \right) \times \left( g_{\lambda_n}^{\circ (\lfloor Cn \rfloor +1)} (0) - g_{\lambda_n}^{\circ \lfloor Cn \rfloor} (0)\right)\\
& \hspace{2.6cm} \times \Pi'_{\lambda_n} \left( \left( 1-\frac{z}{n^2}\right) g_{\lambda_n}^{\circ \lfloor Cn \rfloor} (0) \right).
\end{align*}
by convexity of $\Pi_{\lambda_n}$. We can now compute everything using Lemmas \ref{itererg} and \ref{Piexact}. As $n \to +\infty$, the first factor goes to $\frac{1}{2}$, the second one goes to $e^{2 \sqrt{2} C}$, the third one goes to $1$. Moreover, by Lemma \ref{itererg} we have
\[ g_{\lambda_n}^{\circ \lfloor Cn \rfloor}(0)=1-\frac{2}{\sinh^2 (\sqrt{2} C)} \frac{1}{n^2}+ O \left( \frac{1}{n^3} \right) \]
and, if $x_n = 1-\frac{c}{n^2} +o \left( \frac{1}{n^2} \right)$, by \eqref{generating} we have
\[g_{\lambda_n}(x_n)-x_n \underset{n \to +\infty}{\sim} \frac{2c \sqrt{c+2}}{n^3},\]
so the fourth factor is equivalent to $4 \sqrt{2} \frac{\cosh \left( \sqrt{2} C \right) }{\sinh^3 \left( \sqrt{2} C \right) } \frac{1}{n^3}$. Finally, by taking the derivative of Lemma \ref{Piexact}, we get
\[\Pi'_{\lambda_n} \left( 1-\frac{c}{n^2} + o \left( \frac{1}{n^2}\right) \right)=\frac{1}{\sqrt{2} \left( \sqrt{2}+\sqrt{c+2} \right)^2} n^3 +o \left( n^3 \right).\]
By putting all these estimates together we obtain
\[\mathbb{E} \left[ \left( 1-\frac{z}{n^2} \right)^{P_{\lambda_n}(Cn)} \right] \xrightarrow[n \to +\infty]{}  e^{2 \sqrt{2} C} \frac{\cosh \left( \sqrt{2} C \right)}{\sinh^3 \left( \sqrt{2} C \right)} \left( 1+\sqrt{z+\coth^2 \left( \sqrt{2} C\right)} \right)^{-2}. \]
This goes to $1$ as $C \to +\infty$ (remember that $z=\frac{3}{\sinh^2 \left( \sqrt{2} (C-r) \right)}$ with $r$ fixed), so if $C$ is chosen large enough, then $\mathbb{E} \left[ \left( 1-\frac{z}{n^2} \right)^{P_{\lambda_n}(Cn)} \right] \geq 1-\eps$ for $n$ large enough. This proves \eqref{sufficientcond} and the version (i) of the lemma.
\item[(ii)]
This is quite easy using version (i). Let $C$ be given by point (i). Note that if any $x \in \partial B_{Cn}^{\bullet} \left( S^0_{\lambda_n} \right)$ is $(rn)$-close to the boundary, then so  is any $x' \in S^0_{\lambda_n} \backslash B_{Cn}^{\bullet} \left( S^0_{\lambda_n} \right)$. Indeed, any geodesic $\gamma$ from $x'$ to $\rho$ must contain a point $x \in \partial B_{Cn}^{\bullet} \left( S^0_{\lambda_n} \right)$, and we can replace the portion of $\gamma$ between $x$ and $\rho$ by a geodesic that coincides with $\gamma_{\ell}$ or $\gamma_r$ between height $0$ and $rn$.

Hence, it is enough to find $C'$ such that with probability $1-\eps$, any point $x' \in S^0_{\lambda_n}$ such that $d(x',\rho) \geq C'n$ is not in $B_{Cn}^{\bullet} \left( S^0_{\lambda_n} \right)$. In other words, we want to prove that the radius of $\frac{1}{n} B_{Cn}^{\bullet} \left( S^0_{\lambda_n} \right)$ from $\rho$ is tight as $n \to +\infty$. Since $S^0_{\lambda_n}$ can be embedded in $\T_{\lambda_n}$in a way that preserves the distances from $\rho$, this is a consequence of the local Gromov--Hausdorff tightness of $\frac{1}{n} \T_{\lambda_n}$.
\end{enumerate}
\end{proof}

Note that by Lemma \ref{0almost1_strip}, Lemma \ref{nocollision} holds if we replace $S^1_{\lambda_n}$ by $S^0_{\lambda_n}$ and Lemma \ref{No_new_geodesics} also holds if we replace $S^0_{\lambda_n}$ by $S^1_{\lambda_n}$. We will use these results for both $S^0_{\lambda_n}$ and $S^1_{\lambda_n}$.

\subsection{Identification of the geodesic tree via Gromov--Hausdorff-closed events}\label{subsec:GH_closed}

The last two lemmas together with Lemma \ref{GHtree} and the fact that $\mathbf{T}^g(\hp)$ is a tree are basically enough to prove Theorem \ref{thm3_hbp}. However, to prove it properly, we need to express the distribution of $\mathbf{T}^g(\hp)$ in terms of closed events for the Gromov--Hausdorff topology, which turns out to be a bit technical. 

Let $\mathbf{t}$ be a (finite or infinite) plane tree with a root vertex $\rho$. If $v \in V(\mathbf{t}) \backslash \{\rho\}$, we write $p_v$ for its parent. Let $(h_v)_{v \in V(\mathbf{t})}$ be a family of nonnegative numbers satisfying $h_{\rho}=0$ and $h_v>h_{p_v}$ for every $v \in V(\mathbf{t}) \backslash \{\rho\}$. We write $\mathbf{t}[h]$ for the metric space obtained from $\mathbf{t}$ by giving, for every $v \in V (\mathbf{t}) \backslash \{ \rho \}$, a length $h_v-h_{p_v}$ to the edge between $p_v$ and $v$. We also recall that $\mathbf{B}$ is the infinite tree in which every vertex has two children, except the root which has only one.

We will now define a large family of events, whose probability will characterize the distribution of a random tree of the form $\mathbf{B}[H]$. Let $t$ be a finite binary tree (that is, a tree in which every vertex has $0$ or $2$ children, except the root which has exactly one). We write $V^*(t)$ for the set of vertices of $t$ that are not leaves and not $\rho$. Let $r>0$, and let $(a_v)_{v \in V^*(t)}$, $(b_v)_{v \in V^*(t)}$ be such that $0<a_v<b_v<r$ for every $v \in V^*(t)$. We write $\mathscr{A}_r^{t}(a,b)$ for the set of unbounded trees $\mathbf{T}$ (considered as metric spaces) such that $B_r(\mathbf{T})$ is of the form $t[h]$, where $h_{\rho}=0$, $h_v=r$ if $v$ is a leaf of $t$ and $a_v \leq h_v \leq b_v$ for every $v \in V^*(t)$.

In order to prove Theorem \ref{thm3_hbp}, we will estimate the probability that $\mathbf{T}^g(\hp)$ belongs to $\mathscr{A}_{r}^t(a,b)$. Unfortunately, for the reasons listed in of Section 4.1, the events $\{\mathbf{T}^g(X) \in \mathscr{A}_{r}^t(a,b) \}$ are not closed for the Gromov--Hausdorff distance. To compute $\P \left( \mathbf{T}^g(\hp) \in \mathscr{A}_{r}^t(a,b) \right)$ from our discrete estimates, we need to approximate the event $\{\mathbf{T}^g(X) \in \mathscr{A}_{r}^t(a,b) \}$ by closed events. Since such approximations are tedious to write down explicitly in the general case, we will focus on the case where $t=t_0$ is the binary tree with two leaves and one vertex of degree $3$. Note that $|V^*(t_0)|=1$, so $a$ and $b$ are just two real numbers.

Let $C> r$ and let $R \geq C+1$. If $\delta, \eps>0$, we write $\mathscr{A}_{r,C,R}^{\delta,\eps}(a,b)$ for the set of compact metric spaces $(X,d)$ satisfying the following property.
\begin{center}
"There are points $x_0$, $x_1$ and $x_2$ in $X$ and geodesics $\gamma_1$ (resp. $\gamma_2$) from $x_0$ to $x_1$ (resp. to $x_2$) such that:
\begin{itemize}
\item[(i)]
$d(\rho,x_0)=a$,
\item[(i)]
$d(\rho, x_1)=d(\rho, x_0)+ d(x_0,x_1)=r$ and $d(\rho, x_2)=d(\rho, x_0)+ d(x_0,x_2)=r$,
\item[(iii)]
for every $x \in X$ with $d(\rho, x)>C$, the distance $d(\rho,x)$ is equal to $d(\rho,x_1)+d(x_1,x)$ or to $d(\rho, x_2)+d(x_2,x)$ (it may be equal to both),
\item[(iv)]
there are two points $y,z \in X$ with $d(\rho, y) \geq R$ and $d(\rho, z) \geq R$ such that $d(\rho,y)=d(\rho, x_1)+d(x_1,y)$ and $d(\rho,z)=d(\rho, x_2)+d(x_2,z)$,
\item[(v)]
if $u \in \gamma_1$ and $v \in \gamma_2$ with $d(\rho, u) > b+2\eps$ and $d(\rho,v) > b+2\eps$, then $d(u,v) \geq \delta$."
\end{itemize}
\end{center}

We refer to the Appendix for the proof that this event is closed for the pointed Gromov--Hausdorff topology. More precisely, it is easy to check that $\mathscr{A}_{r,C,R}^{\delta,\eps}(a,b)$ is simply generated by geodesics as in Definition \ref{simplegeodesics}, so by Proposition \ref{simplegeodesics_closed} it is closed. By the convergence of $\frac{1}{n} \T_{\lambda_n}$ to $\hp$, we have
\begin{equation} \label{limsup}
\P \left( B_R(\hp) \in \mathscr{A}_{r,C,R}^{\delta,\eps}(a,b) \right) \geq \limsup_{n \to +\infty} \P \left( \frac{1}{n} B_{Rn} \left( \T_{\lambda_n} \right) \in \mathscr{A}_{r,C,R}^{\delta,\eps}(a,b) \right).
\end{equation}

We now try to estimate the right-hand side. By Lemma \ref{GHtree}, we have
\begin{equation}\label{discrete_tree_continuous_tree}
\lim_{n \to +\infty} \P \left( \frac{1}{n} \mathbf{T}^g_{\lambda_n} \in \mathscr{A}_{r}^{t_0}(a,b) \right) = \P \left( \mathbf{Y}_{2\sqrt{2}} \in \mathscr{A}_{r}^{t_0}(a,b) \right).
\end{equation}
Note that to deduce \eqref{discrete_tree_continuous_tree} from Lemma \ref{GHtree}, we need to show $\P \left( \mathbf{Y}_{2\sqrt{2}} \in \partial \mathscr{A}_{r}^{t_0}(a,b) \right)=0$, where $\partial \mathscr{A}_{r}^{t_0}(a,b)$ is the boundary of $\mathscr{A}_{r}^{t_0}(a,b)$ in the space of rooted metric trees, equipped with the local Gromov--Hausdorff distance. This is true because if $T \in \partial \mathscr{A}_{r}^{t_0}(a,b)$, then $T$ must have a branching point at height exactly $a$, $b$ or $r$, which a.s. does not happen.

If the event in the left-hand side of \eqref{discrete_tree_continuous_tree} occurs, let $x_0$ be the unique point of $\mathbf{T}^g_{\lambda_n}$ at height $an$, and let $x_1$ (resp. $x_2$) be the vertex on the left (resp. right) branch  of $\mathbf{T}^g_{\lambda_n}$ at height $rn$. Then the vertices $x_0$, $x_1$ and $x_2$ and the geodesics $\gamma_1$ and $\gamma_2$ joining $x_1$ and $x_2$ to $x_0$ in $\mathbf{T}^g_{\lambda_n}$ satisfy assumptions (i) and (ii) in the definition of $\mathscr{A}_{r,C,R}^{\delta,\eps}(a,b)$. Since the tree $\mathbf{T}^g_{\lambda_n}$ is infinite, they also satisfy assumption (iv) for any $R>0$.

We now fix $\eps>0$ and apply Lemma \ref{No_new_geodesics} (version (ii)) to the two strips $S^1_{\lambda_n}$ and $S^0[x_0]$ (the strip whose lowest point is $x_0$). Lemma \ref{No_new_geodesics} shows that there is $C>0$ such that, with probability at least $1-4 \eps$, for any $x$ in one of the two strips $S^1_{\lambda_n}$ and $S^0(x_0)$ such that $d(x,\rho)>Cn$, there is a geodesic from $\rho$ to $x$ that coincides with $\gamma_1$ between $\rho$ and $x_1$ or with $\gamma_2$ between $\rho$ and $x_2$. We claim that this is also the case if $x$ does not belong to one of these two strips. Indeed, a geodesic from $x$ to $\rho$ must hit the boundary of one of the two strips above $x_1$ or $x_2$ (cf. Figure \ref{x_out_strips}). Therefore, the probability that $\frac{1}{n} \mathbf{T}^g_{\lambda_n} \in \mathscr{A}_r^{t_0}(a,b)$ but assumption (iii) is not satisfied for $C$ is at most $4\eps$.

\begin{figure}
\begin{center}
\begin{tikzpicture}
\fill[orange!30] (3.5,0.75) to[bend left=15] (7,0.75)--(7,-1) to[bend left=15] (4,-1)--(2,0)--(3.5,0.75);
\fill[yellow!40] (3.5,0.75) to[bend right=15] (-3,1)--(-3,-1) to[bend right=15] (4,-1)--(2,0)--(3.5,0.75);

\draw[red, thick] (0,0)--(2,0);
\draw[red, thick] (2,0)--(3.5,0.75);
\draw(2,0)--(4,-1);
\draw(3.5,0.75) to[bend left=15] (7,0.75);
\draw(4,-1) to[bend right=15] (7,-1);
\draw(3.5,0.75) to[bend right=15] (-3,1);
\draw(4,-1) to[bend left=15] (-3,-1);
\draw[red, thick] (0.5,4) to[bend left=15] (1.5,1.24);
\draw[red, thick] (1.5,1.24) to[bend left=3] (3.5,0.75);

\draw[red, thick] (-2,0)--(-0.62,1.36);
\draw[red, thick] (-0.63,1.36) to[bend left=5] (1.5,1.24);

\draw[red, thick] (5,0)--(5,-1.2);
\draw[red, thick] (5,-1.2) to[bend left=2] (4,-1);
\draw[red, thick] (4,-1)--(2,0);

\draw(0,0)node{};
\draw(-0.5,0)node[texte]{$\rho$};
\draw(2,0)node{};
\draw(2,-0.3)node[texte]{$x_0$};
\draw(3,0.5)node{};
\draw(2.7,0.7)node[texte]{$x_1$};
\draw(3,-0.5)node{};
\draw(2.7,-0.7)node[texte]{$x_2$};
\draw(0.5,4)node{};
\draw(0.8,4,0)node[texte]{$x$};
\draw(-2,0)node{};
\draw(-2,-0.3)node[texte]{$x$};
\draw(5,0)node{};
\draw(5,0.3)node[texte]{$x$};
\draw(-1,-1.1)node[texte]{$S^1_{\lambda}$};
\draw(6.3,-0.1)node[texte]{$S^0[x_0]$};
\end{tikzpicture}
\end{center}
\caption{If $x \notin S^1_{\lambda} \cup S^0[x_0]$, then any geodesic from $x$ to $\rho$ must cross the boundary of one of the two strips $S^1_{\lambda}$ and $S^0[x_0]$ above $x_1$ or $x_2$. Therefore, for every $x$ with $d(\rho, x) \geq Cn$, we have a geodesic (in red) from $x$ to $\rho$ that passes through $x_1$ or $x_2$.} \label{x_out_strips}
\end{figure}
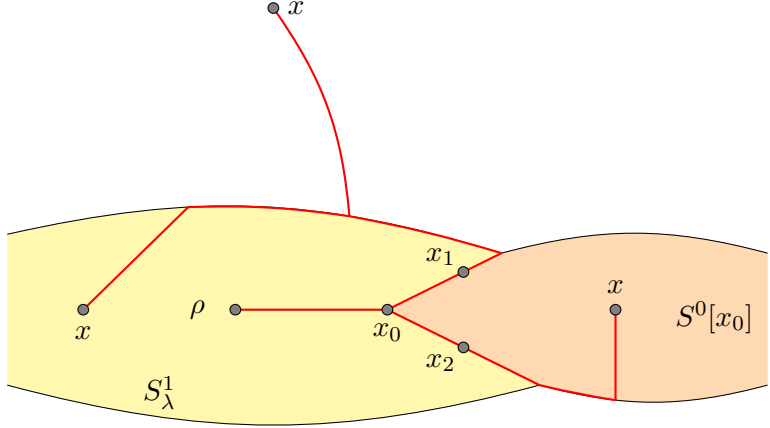

Finally, if $\frac{1}{n} \mathbf{T}^g_{\lambda_n} \in \mathscr{A}_r^{t_0}(a,b)$ but assumption (v) is not satisfied for some $0<\delta<\eps$, there are two vertices $v_1$ on $\gamma_1$ and $v_2$ on $\gamma_2$, at distance from $\rho$ between $(b+2\eps)n$ and $rn$, such that $d_{\T_{\lambda_n}}(v_1, v_2) < \delta n$. A geodesic from $v_1$ to $v_2$ must cross either $S_{\lambda_n}^1$ or $S^0[x_0]$ and, since $\delta<\eps$, it must stay at distance at least $(b+\eps)n$ from $\rho$. Therefore, there are two vertices $v'_1$ on $\gamma_1$ and $v'_2$ on $\gamma_2$, at distance from $\rho$ between $(b+\eps)n$ and $(r+\eps)n$, such that $d_{S^1_{\lambda_n}}(v'_1, v'_2) < \delta n$ or $d_{S^0[x_0]}(v'_1, v'_2) < \delta n$. By applying Lemma \ref{nocollision} to $S^1_{\lambda_n}$ between heights $b+\eps$ and $r+\eps$, and to $S^0[x_0]$ between heights $\eps$ and $r+\eps$, we can find $\delta>0$ such that this occurs with probability at most $2 \eps$. Hence, for every $\eps>0$, there is $\delta>0$ such that the probability that $\frac{1}{n} \mathbf{T}^g_{\lambda_n} \in \mathscr{A}_r^{t_0}(a,b)$ but assumption (v) is not satisfied is at most $2\eps$.

Therefore, for all $\eps>0$, there are $C>r$ and $\delta>0$ such that, for any $R \geq C+1$,
\[ \liminf_{n \to +\infty} \P \left( \frac{1}{n} B_{Rn} \left( \T_{\lambda_n} \right) \in \mathscr{A}_{r,C,R}^{\delta,\eps}(a,b) \right) \geq \P \left( \mathbf{Y}_{2\sqrt{2}} \in \mathscr{A}_{r}^{t_0}(a,b) \right)-6\eps, \]
so by \eqref{limsup},
\[ \P \left( B_{R} \big( \hp \big) \in \mathscr{A}_{r,C,R}^{\delta,\eps}(a,b) \right) \geq \P \left( \mathbf{Y}_{2\sqrt{2}} \in \mathscr{A}_{r}^{t_0}(a,b) \right)-6\eps. \]
Since the event $\left\{ B_R(X) \in \mathscr{A}_{r,C,R}^{\delta,\eps}(a,b) \right\}$ is nonincreasing in $R$, we obtain that for any $\eps>0$, we have
\[ \P \left( \exists C>r, \exists \delta>0, \forall R \geq C+1, B_{R} \big( \hp \big) \in \mathscr{A}_{r,C,R}^{\delta,\eps}(a,b) \right) \geq \P \left( \mathbf{Y}_{2\sqrt{2}} \in \mathscr{A}_{r}^{t_0}(a,b) \right)-6\eps. \]
Finally, the event above is increasing in $\eps$ so 
\begin{equation}\label{4quantifs}
\P \left( \forall \eps>0, \exists \delta>0, \exists C>0, \forall R \geq C+1, B_{R} \big( \hp \big) \in \mathscr{A}_{r,C,R}^{\delta,\eps}(a,b) \right) \geq \P \left( \mathbf{Y}_{2\sqrt{2}} \in \mathscr{A}_{r}^{t_0}(a,b) \right).
\end{equation}

\begin{lem}\label{deterministic}
Almost surely, if
\[\forall \eps>0, \exists \delta>0, \exists C>0, \forall R \geq C+1, B_{R} \big( \hp \big) \in \mathscr{A}_{r,C,R}^{\delta,\eps}(a,b),\]
then $\mathbf{T}^g(\hp) \in \mathscr{A}_r^{t_0}(a,b)$.
\end{lem}

\begin{proof}
Fix $C$, $\delta$ and $\eps$ and assume that $B_R(\hp) \in \mathscr{A}_{r,C,R}^{\delta,\eps}(a,b)$ for any $R \geq C+1$. Let $x_0$, $x_1$ and $x_2$ be given by the definition of $\mathscr{A}_{r,C,R}^{\delta,\eps}(a,b)$. We first check that these points do not depend on the parameters $\delta$, $\eps$, $C$ and $R$. By assumption (iv), the points $x_1$ and $x_2$ lie on geodesics of length $C+1$ started from $\rho$. We claim they are the only two points at distance $r$ from $\rho$ satisfying this property. Indeed, if $y$ is a point with $d(\rho,y)=C+1$ and $\gamma$ a geodesic from $\rho$ to $y$, let $z$ be the point of $\gamma$ such that $d(\rho, z)=C$. By assumption (iii) and the fact that $\hp$ is a length space, there is a geodesic $\gamma'$ from $z$ to $\rho$ passing through $x_1$ or $x_2$. By concatenating $\gamma'$ from $\rho$ to $z$ and $\gamma$ from $z$ to $y$, we obtain a geodesic from $\rho$ to $y$ that coincides with $\gamma$ between $z$ and $y$. By Proposition \ref{no-o-}, this geodesic must be equal to $\gamma$, so $\gamma$ must pass through $x_1$ or $x_2$.

Hence, if $B_R(\hp) \in \mathscr{A}_{r,C,R}^{\delta,\eps}(a,b)$, then $x_1$ and $x_2$ are the only two points at distance $r$ from $\rho$ that lie on a geodesic of length $C+1$ started from $\rho$. In particular, they do not depend on $\delta$, $\eps$ and $R$. Moreover, let $C' \geq C$, and assume that $B_R(\hp) \in \mathscr{A}_{r,C,R}^{\delta,\eps}(a,b)$ and $B_R(\hp) \in \mathscr{A}_{r,C',R}^{\delta,\eps}(a,b)$ for all $R \geq C'$. Then the two points at distance $r$ from $\rho$ that lie on a geodesic of length $C'+1$ from $\rho$ are the same as the two points that lie on a geodesic of length $C+1$ started from $\rho$, so the points $x_1$ and $x_2$ do not depend on $C$. Similarly, the point $x_0$ is the only point at distance $a$ from the root that lies on a geodesic of length $C+1$, so it does not depend on $\delta, \eps, C, R$. Hence, we can find $x_0$, $x_1$ and $x_2$ in $\hp$ and $\gamma_1, \gamma_2$ such that:
\begin{itemize}
\item
assumptions (i) and (ii) in the definition of $\mathscr{A}_{r,C,R}^{\delta,\eps}(a,b)$ are satisfied,
\item
there is $C>0$ such that assumption (iii) is satisfied,
\item
for every $R \geq C+1$, assumption (iv) is satisfied,
\item
for every $\eps>0$, there is $\delta>0$ such that assumption (v) is satisfied, which means that the geodesics from $\rho$ to $x_1$ and $x_2$ are disjoint between heights $b$ (excluded) and $r$ (included).
\end{itemize}
Since assumption (iv) is satisfied for any $R$ large enough, there are arbitrarily large geodesics started from $\rho$ passing through $x_1$. By a compactness argument, there are infinite geodesics started from $\rho$ and passing through $x_1$, and the same is true for $x_2$. By assumption (iii), the points $x_1$ and $x_2$ are the only ones with this property. By assumptions (i) and (v), the branching point between the geodesics from $\rho$ to $x_1$ and $x_2$ lies between heights $a$ and $b$, so $\mathbf{T}^g(\hp) \in \mathscr{A}_r^{t_0}(a,b)$.
\end{proof}

The end of the proof of Theorem \ref{thm3_hbp} is now easy. By Lemma \ref{deterministic} and \eqref{4quantifs}, we get
\[ \P \left( \mathbf{T}^g(\hp) \in \mathscr{A}_r^{t_0}(a,b) \right) \geq \P \left( \mathbf{Y}_{2\sqrt{2}} \in \mathscr{A}_r^{t_0}(a,b) \right).\]
The general case for $t$ can be treated along the same lines. This shows that the distribution of $\mathbf{T}^g(\hp)$ dominates that of $\mathbf{Y}_{2\sqrt{2}}$. Since they are both probability measures, they are the same.

\appendix

\section{A Gromov--Hausdorff closedness result}

The goal of this appendix is to prove Proposition \ref{simplegeodesics_closed}. It shows that a wide class of events related to geodesics are closed for the Gromov--Hausdorff distance. We believe it might be of interest in other settings. We write $\mathscr{G}$ for the space of pointed compact metric spaces, equipped with the Gromov--Hausdorff distance. We will be interested in some events depending on a metric space $(X,d,\rho) \in \mathscr{G}$.

\begin{defn}\label{simplepoints}
We say that a subset $\mathscr{A}$ of $\mathscr{G}$ is \emph{simply generated by points} if it has the following form.
Let $k \geq 1$, and let $F \subset \R^{(k+2)^2}$ be closed. Then $\mathscr{A}$ is the set of those $(X,d,\rho) \in \mathscr{G}$ for which there are $(x_i)_{0 \leq i \leq k}$ in $X$ with $x_0=\rho$ such that, for any $x_{k+1} \in X$, the matrix
\[\left( d(x_i,x_j) \right)_{0 \leq i, j \leq k+1} \]
lies in $F$.
\end{defn}

A simple example of such a subset would be the set of metric spaces that can be covered by $k$ balls of radius $r$ for some fixed $k \geq 1$ and $r$.

\begin{lem}\label{simplepoints_closed}
Any subset of $\mathscr{G}$ that is simply generated by points is closed.
\end{lem}

\begin{proof}
Assume that $\mathscr{A}$ is simply generated by points and let $k$ and $F$ be as above. Let $(X_n, d_n, \rho_n)$ converge to a space $(X,d,\rho)$ with $X_n \in \mathscr{A}$ for every $n$. By Gromov--Hausdorff convergence, we can embed $X$ and all the $X_n$ isometrically in a space $(Z,d_Z)$ such that the Hausdorff distance $D_n$ between $X_n$ and $X$ goes to $0$.

For every $n$, let $x_0^n, \dots, x_k^n \in X_n$ satisfy the condition given by Definition \ref{simplepoints}. We take $y_0^n, \dots, y_k^n \in X$ such that $d_Z(x_i^n, y_i^n) \leq 2D_n$. For all $0 \leq i \leq k$, let $y_i$ be a subsequential limit of $(y_i^n)_{n \geq 0}$ in $X$ (which exists by compactness). To complete the proof that $X \in \mathscr{A}$, all we need to show is that $y_0=\rho$ and that for any $y_{k+1} \in X$, we have
\[\left( d(y_i,y_j) \right)_{0 \leq i,j \leq k+1} \in F.\]
The first point is easy because the distances $d_Z(\rho, \rho_n)$, $d_Z(\rho_n, y_0^n)$ and $d(y_0^n,y_0)$ all go to $0$ along some subsequence. Moreover, let $y_{k+1} \in X$. There is $x_{k+1}^n \in X_n$ such that $d_Z(x_{k+1}^n, y_{k+1}) \leq 2D_n$. For every $0 \leq i,j \leq k+1$, we then have
\[d(y_i,y_j)=\lim_{n \to +\infty} d(x_i^n, x_j^n)\]
along some subsequence.
But we know that for all $n$ we have $\left( d(x_i^n,x_j^n) \right)_{0 \leq i,j \leq k+1} \in F$, so we can conclude.
\end{proof}

\begin{defn}\label{simplegeodesics}
We say that a subset $\mathscr{A}$ of $\mathscr{G}$ is \emph{simply generated by geodesics} if it has the following form. Let $k \geq 1$, and let $F \subset \R^{(2k+2)^2}$ be closed. Then $\mathscr{A}$ is the set of those $(X,d,\rho) \in \mathscr{G}$ for which there are $(x_i)_{0 \leq i \leq k}$ in $X$ with $x_0=\rho$ and geodesics $(\gamma_i)_{1 \leq i \leq k}$ from $\rho$ to $x_i$, satisfying the following property. For any $(x_{k+i})_{1 \leq i \leq k}$ such that $x_{k+i} \in \gamma_i$ for every $i$, and for every $x_{2k+1} \in X$, the matrix
\[\left( d(x_i,x_j) \right)_{0 \leq i,j \leq 2k+1} \]
lies in $F$.
\end{defn}

For example, the events studied in Section \ref{subsec:GH_closed} are simply generated by geodesics. A simpler example of such an event would be "there are $k$ geodesics such that any point lies at distance at most $r$ from one of these geodesics".

\begin{prop}\label{simplegeodesics_closed}
Any subset of $\mathscr{G}$ that is simply generated by geodesics is closed.
\end{prop}

To go from Lemma \ref{simplepoints_closed} to Proposition \ref{simplegeodesics_closed}, we will need the following definition.

\begin{defn}
Let $\ell \geq 0$, and let $x,y$ be two points of a metric space $(X,d)$. An \emph{$\ell$-geodesic chain from $x$ to $y$} is a finite sequence $(x(i))_{0 \leq i \leq 2^{\ell}}$ of points of $X$ such that
\begin{enumerate}[(i)]
\item
$x(0)=x$ and $x(2^{\ell})=y$,
\item
$d(x(i),x(i+1))=\frac{1}{2^{\ell}} d(x,y)$ for any $0 \leq i \leq 2^{\ell}-1$.
\end{enumerate}
\end{defn}

\begin{proof}[Proof of Proposition \ref{simplegeodesics_closed}]
Let $\mathscr{A}$ be a subset of $\mathscr{G}$ that is simply generated by geodesics. For $\ell>0$, we write $\mathscr{A}^{\ell}$ for the subset of $\mathscr{G}$ we obtain if we replace continuous geodesics by $\ell$-geodesic chains in Definition \ref{simplegeodesics}. Then $\mathscr{A}^{\ell}$ is simply generated by points (because the conditions in the definition of an $\ell$-geodesic chain are closed), so $\mathscr{A}^{\ell}$ is closed by Lemma \ref{simplepoints_closed}. Hence, to conclude, it is enough to show
\[ \mathscr{A}=\bigcap_{\ell \geq 0} \mathscr{A}^{\ell}.\]
The inclusion from left to right is immediate since any continuous geodesic contains an $\ell$-geodesic chain. Now let $(X,d,\rho) \in \bigcap_{\ell \geq 0} \mathscr{A}^{\ell}$. For every $\ell \geq 0$ and $1 \leq i \leq k$, let $(x_i^{\ell})$ in $X$ and let $\left( \gamma_i^{\ell} (j) \right)_{0 \leq j \leq 2^{\ell}}$ be $\ell$-geodesic chains from $\rho$ to $x_i^{\ell}$ satisfying the assumptions of definition \ref{simplegeodesics}. Up to extraction, we may assume that for every $1 \leq i \leq k$, the points $x_i^{\ell}$ converge to a point $x_i \in X$. Up to further extraction, by a diagonal argument, for every $t$ of the form $\frac{j}{2^{m}}$ with $0 \leq j \leq 2^m$, the sequence $\left( \gamma_i^{m+{\ell}}(2^{\ell} j) \right)_{\ell \geq 0}$ converges to a point $\gamma_i(t)$. Moreover, for all such $t,t'$, we have $d(\gamma_i(t),\gamma_i(t'))=|t-t'|d(\rho,x_i)$, so we can extend $\left( \gamma_i(t) \right)_{0 \leq t \leq 1, t=j/2^m}$ to a continuous geodesic from $\rho$ to $x_i$. It is then easy to check that the geodesics $\gamma_i$ satisfy the required hypothesis.
\end{proof}

\bibliographystyle{abbrv}
\bibliography{bibli_curien}

\begin{thebibliography}{10}

\bibitem{Ab13}
C.~Abraham.
\newblock Rescaled bipartite planar maps converge to the {B}rownian map.
\newblock {\em Ann. Inst. H. Poincaré Probab. Statist.}, 52(2):575--595, 05
  2016.

\bibitem{AA13}
L.~Addario-Berry and M.~Albenque.
\newblock The scaling limit of random simple triangulations and random simple
  quadrangulations.
\newblock {\em Ann. Probab.}, 45(5):2767--2825, 09 2017.

\bibitem{Ang03}
O.~Angel.
\newblock Growth and percolation on the uniform infinite planar triangulation.
\newblock {\em Geom. Funct. Anal.}, 13(5):935--974, 2003.

\bibitem{AHNR15}
O.~Angel, T.~Hutchcroft, A.~Nachmias, and G.~Ray.
\newblock Unimodular hyperbolic triangulations: circle packing and random walk.
\newblock {\em Inventiones mathematicae}, 206(1):229--268, Oct 2016.

\bibitem{AKM15}
O.~Angel, B.~Kolesnik, and G.~Miermont.
\newblock Stability of geodesics in the {B}rownian map.
\newblock {\em Ann. Probab.}, 45(5):3451--3479, 09 2017.

\bibitem{ANR14}
O.~Angel, A.~Nachmias, and G.~Ray.
\newblock Random walks on stochastic hyperbolic half planar triangulations.
\newblock {\em Random Structures and Algorithms}, 49(2):213--234, 2016.

\bibitem{AR13}
O.~Angel and G.~Ray.
\newblock Classification of half planar maps.
\newblock {\em Ann. Probab.}, 43(3):1315--1349, 2015.

\bibitem{AR18}
O.~Angel and G.~Ray.
\newblock The half plane {UIPT} is recurrent.
\newblock {\em Probability Theory and Related Fields}, 170(3):657--683, Apr
  2018.

\bibitem{AS03}
O.~Angel and O.~Schramm.
\newblock Uniform infinite planar triangulations.
\newblock {\em Comm. Math. Phys.}, 241(2-3):191--213, 2003.

\bibitem{AN72}
K.~B. Athreya and P.~E. Ney.
\newblock {\em Branching processes}, volume 196 of {\em Die Grundlehren der
  mathematischen Wissenschaften}.
\newblock Springer-Verlag, 1972.

\bibitem{BLG13}
J.~Beltran and J.-F. Le~Gall.
\newblock Quadrangulations with no pending vertices.
\newblock {\em Bernoulli}, 19:1150--1175, 2013.

\bibitem{BT17}
I.~Benjamini and R.~Tessera.
\newblock First passage percolation on a hyperbolic graph admits bi-infinite
  geodesics.
\newblock {\em Electron. Commun. Probab.}, 22:8 pp., 2017.

\bibitem{BJM13}
J.~Bettinelli, E.~Jacob, and G.~Miermont.
\newblock The scaling limit of uniform random plane maps, via the
  {A}mbj{\o}rn--{B}udd bijection.
\newblock {\em Electronic J. Probab.}, 19, 2014.

\bibitem{B16}
T.~Budzinski.
\newblock The hyperbolic {B}rownian plane.
\newblock {\em Probability Theory and Related Fields}, 171(1):503--541, Jun
  2018.

\bibitem{B18}
T.~Budzinski.
\newblock Supercritical causal maps: geodesics and simple random walk.
\newblock {\em arxiv:1806.10588}, 2018.

\bibitem{CurPSHIT}
N.~Curien.
\newblock Planar stochastic hyperbolic triangulations.
\newblock {\em Probability Theory and Related Fields}, 165(3):509--540, 2016.

\bibitem{CLGplane}
N.~Curien and J.-F. Le~Gall.
\newblock The {B}rownian plane.
\newblock {\em J. Theoret. Probab.}, 27(4):1249--1291, 2014.

\bibitem{CLGmodif}
N.~Curien and J.-F. Le~Gall.
\newblock First-passage percolation and local modifications of distances in
  random triangulations.
\newblock {\em arXiv:1511.04264}, 2015.

\bibitem{CLGpeeling}
N.~Curien and J.-F. Le~Gall.
\newblock Scaling limits for the peeling process on random maps.
\newblock {\em Ann. Inst. H. Poincaré Probab. Statist.}, 53(1):322--357, 02
  2017.

\bibitem{CMMinfini}
N.~Curien, L.~M{\'e}nard, and G.~Miermont.
\newblock A view from infinity of the uniform infinite planar quadrangulation.
\newblock {\em Lat. Am. J. Probab. Math. Stat.}, 10(1):45--88, 2013.

\bibitem{CM18}
N.~Curien and L.~Ménard.
\newblock The skeleton of the {UIPT}, seen from infinity.
\newblock {\em Annales Henri Lebesgue}, 1, 03 2018.

\bibitem{E75}
W.~Esty.
\newblock The reverse {G}alton-{W}atson process.
\newblock {\em J. of Applied Prob.}, 12:574--580, 1975.

\bibitem{Gro81}
M.~Gromov.
\newblock Hyperbolic manifolds, groups and actions.
\newblock In {\em Riemann surfaces and related topics: {P}roceedings of the
  1978 {S}tony {B}rook {C}onference ({S}tate {U}niv. {N}ew {Y}ork, {S}tony
  {B}rook, {N}.{Y}., 1978)}, volume~97 of {\em Ann. of Math. Stud.}, pages
  183--213. Princeton Univ. Press, Princeton, N.J., 1981.

\bibitem{HP17}
T.~Hutchcroft and Y.~Peres.
\newblock Boundaries of planar graphs: a unified approach.
\newblock {\em Electron. J. Probab.}, 22:20 pp., 2017.

\bibitem{Kes86SF}
H.~Kesten.
\newblock Aspects of first passage percolation.
\newblock In {\em \'{E}cole d'\'et\'e de probabilit\'es de {S}aint-{F}lour,
  {XIV}---1984}, volume 1180 of {\em Lecture Notes in Math.}, pages 125--264.
  Springer, Berlin, 1986.

\bibitem{Kes86}
H.~Kesten.
\newblock Subdiffusive behavior of random walk on a random cluster.
\newblock {\em Ann. Inst. H. Poincar\'e Probab. Statist.}, 22(4):425--487,
  1986.

\bibitem{Kri05}
M.~Krikun.
\newblock Local structure of random quadrangulations.
\newblock {\em arXiv:0512304}.

\bibitem{Kri04}
M.~Krikun.
\newblock A uniformly distributed infinite planar triangulation and a related
  branching process.
\newblock {\em Zap. Nauchn. Sem. S.-Peterburg. Otdel. Mat. Inst. Steklov.
  (POMI)}, 307(Teor. Predst. Din. Sist. Komb. i Algoritm. Metody. 10):141--174,
  282--283, 2004.

\bibitem{Kri07}
M.~Krikun.
\newblock Explicit enumeration of triangulations with multiple boundaries.
\newblock {\em Electron. J. Combin.}, 14(1):Research Paper 61, 14 pp.
  (electronic), 2007.

\bibitem{LG09}
J.-F. Le~Gall.
\newblock Geodesics in large planar maps and in the {B}rownian map.
\newblock {\em Acta Math.}, 205:287--360, 2010.

\bibitem{LG11}
J.-F. Le~Gall.
\newblock Uniqueness and universality of the {B}rownian map.
\newblock {\em Ann. Probab.}, 41:2880--2960, 2013.

\bibitem{LGP08}
J.-F. Le~Gall and F.~Paulin.
\newblock Scaling limits of bipartite planar maps are homeomorphic to the
  2-sphere.
\newblock {\em Geom. Funct. Anal.}, 18(3):893--918, 2008.

\bibitem{Mar16}
C.~Marzouk.
\newblock Scaling limits of random bipartite planar maps with a prescribed
  degree sequence.
\newblock {\em Random Structures and Algorithms}, 2018.

\bibitem{Mie11}
G.~Miermont.
\newblock The {B}rownian map is the scaling limit of uniform random plane
  quadrangulations.
\newblock {\em Acta Math.}, 210(2):319--401, 2013.

\bibitem{M16}
L.~Ménard.
\newblock Volumes in the {U}niform {I}nfinite {P}lanar {T}riangulation: from
  skeletons to generating functions.
\newblock {\em Combinatorics, Probability and Computing}, 27(6):946–973,
  2018.

\bibitem{St18}
R.~Stephenson.
\newblock Local convergence of large critical multi-type {G}alton--{W}atson
  trees and applications to random maps.
\newblock {\em Journal of Theoretical Probability}, 31(1):159--205, Mar 2018.

\end{thebibliography}

\end{document}